\documentclass[10pt]{article}
\usepackage{lmodern}
\usepackage{amsmath}
\usepackage[T1]{fontenc}
\usepackage[utf8]{inputenc}
\usepackage{authblk}
\usepackage{amsfonts}
\usepackage{graphicx}
\usepackage{rotating}
\usepackage{amssymb}
\usepackage[english]{babel}
\usepackage{xcolor}
\usepackage{amsthm}
\usepackage{graphicx}
\usepackage{mathrsfs}
\usepackage{makecell}
\usepackage{microtype}
\usepackage{mathscinet}
\usepackage{array}
\usepackage{multirow}
\usepackage{enumerate}
\usepackage[cal=boondoxo,bb=ams]{mathalfa}
\usepackage{hyperref}
\usepackage{booktabs}
\hypersetup{hidelinks}

\renewcommand{\theequation}{\arabic{section}.\arabic{equation}}

\newtheorem{theorem}{Theorem}[section]
\newtheorem{prop}{Proposition}[section]
\newtheorem{lemma}{Lemma}[section]

\newtheorem{remark}{Remark}[section]

\newcommand{\ml}{\mathcal}
\newcommand{\mb}{\mathbb}
\DeclareMathOperator{\intt}{int}
\DeclareMathOperator{\extt}{ext}
\DeclareMathOperator{\bdd}{bdd}
\DeclareMathOperator{\divv}{div}

\def\XXint#1#2#3{{\setbox0=\hbox{$#1{#2#3}{\int}$ }
		\vcenter{\hbox{$#2#3$ }}\kern-.6\wd0}}

\title{Sharp large time asymptotic behavior for the multi-dimensional thermoelastic systems of type II and type III}
\author[1]{Wenhui Chen\thanks{Wenhui Chen (wenhui.chen.math@gmail.com)}}
\affil[1]{School of Mathematics and Information Science, Guangzhou University,\authorcr 510006 Guangzhou, China}

\author[2]{Ryo Ikehata\thanks{Ryo Ikehata (ikehatar@hiroshima-u.ac.jp)}}
\affil[2]{Department of Mathematics, Division of Educational Sciences, Graduate School of Humanities and Social Sciences, Hiroshima University, 739-8524 Higashi-Hiroshima, Japan}
\date{}

\begin{document}
		\maketitle

		\begin{abstract}
			\medskip
		In this paper, we study large time asymptotic behavior of the elastic displacement $u$ and the temperature difference $\theta$ for the thermoelastic systems of type II and type III in the whole space $\mb{R}^n$ without using the thermal displacement transformation. For the type III model with hyperbolic thermal effect,  we derive optimal growth/decay estimates and the novel double diffusion waves profiles for the solutions $u,\theta$ with the $L^1$ integrable initial data as large time, which improve the results in \cite{Zhang-Zuazua=2003,Reissig-Wang=2005,Yang-Wang=2006,Jachmann=2008}. This hyperbolic thermal law produces a stronger singularity than the classical Fourier law in thermoelastic systems. For the type II model lacking of dissipation mechanism, we obtain optimal growth estimates and the new double waves profiles for the solutions $u,\theta$ as large time in low dimensions. Particularly, these results on the type II/III models show the infinite time blowup phenomena for $u$ if $n\leqslant 4$ and $\theta$ if $n\leqslant 2$ in the $L^2$ norm due to the hyperbolic influence. We also clarify competitions between the elastic waves effect and the hyperbolic thermal effect for the thermoelastic systems via the critical dimensions.
			\\
			
			
			\noindent\textbf{AMS Classification (2020)}  74F05, 74H40, 35B40, 35L52, 35Q74
		\end{abstract}
\fontsize{12}{15}
\selectfont
\newpage
\tableofcontents
\newpage
 
\section{Introduction}\label{Section_Introduction}\setcounter{equation}{0}
\subsection{Background for the thermoelastic systems of type II and type III}
\hspace{5mm}It is well-known that the reciprocal actions between elastic stress and thermal behavior (including the temperature difference of elastic heat conductive media) are modeled by thermoelastic systems mathematically (cf. \cite{Green-Naghdi=1991,Chand=1998,Jiang-Racke=2000}). The classical model of thermoelasticity is based on the Fourier law of heat conduction (the heat flux is proportional to the gradient of temperature), which is well-studied  in the last century by \cite{Dafermos=1968,Racke-Shibata=1991,Zheng=1995,Racke-Wang=1998,Jiang-Racke=2000} and rich references therein for some qualitative properties of solutions in bounded or unbounded domains. 

Unfortunately, the application of classical Fourier law in thermoelasticity leads to the physically unrealistic phenomenon: \emph{an infinite signal speed of paradox} that a sudden temperature or a thermal disturbance  changes at some points will be felt instantly everywhere. In order to overcome this paradox, A.E. Green and P.M. Naghdi in a series of pioneering papers \cite{Green-Naghdi=1991,Green-Naghdi=1992,Green-Naghdi=1993} re-examined the classical model of thermoelasticity based on an entropy equality rather than the usual entropy inequality (i.e. the constitutive assumption on the heat flux vector is different from the Fourier law) and introduced the so-called \emph{thermoelastic systems of type II and type III}, where the heat flux is also assumed to be proportional to the thermal displacement gradient. The theory of type II does not allow the dissipation of the energy and it is usually known as  the thermoelasticity without energy dissipation. The type III model that is the type II model with an additional dissipation can be understood by a system of wave equations with a kind of damping mechanism. A.E. Green and P.M. Naghdi in \cite{Green-Naghdi=1993} stated that they perhaps more natural candidates for their identifications as thermoelasticity than the usual theory.

 Precisely, from these hyperbolic heat conductions the thermoelastic systems in the whole space $\mb{R}^n$ are addressed by
\begin{align}\label{Eq-Original-01}
\begin{cases}
\ml{U}_{tt}-a^2\Delta\ml{U}-(b^2-a^2)\nabla\divv \ml{U}+\gamma_1\nabla\theta=0,&x\in\mb{R}^n,\ t>0,\\
\theta_{tt}-\kappa\Delta\theta-\delta\Delta\theta_t+\gamma_2\divv\ml{U}_{tt}=0,&x\in\mb{R}^n,\ t>0,\\
(\ml{U},\ml{U}_t)(0,x)=(\ml{U}_0,\ml{U}_1)(x),&x\in\mb{R}^n,\\
(\theta,\theta_t)(0,x)=(\theta_0,\theta_1)(x),&x\in\mb{R}^n,
\end{cases}
\end{align}
with the thermal parameter $\kappa>0$, the coupling coefficients such that $\gamma_1\gamma_2>0$, the dissipation coefficient $\delta\geqslant0$ in hyperbolic thermal laws ($\delta=0$ coincides with the thermoelasticity of type II, and $\delta>0$ coincides with the thermoelasticity of type III), where $\ml{U}=\ml{U}(t,x)\in\mb{R}^n$ and $\theta=\theta(t,x)\in\mb{R}$ are the elastic displacement and the temperature difference to the constant equilibrium state, respectively. The positive constants $b$ and $a$ denote the propagation speeds of longitudinal $P$-wave and transverse $S$-wave fulfilling $b>a>0$. Note that $a^2$ and $b^2-2a^2$ are the Lam\'e constants. Particularly, the heat flow for the type III model is analogous to the flow of one type of viscous fluids due to the dissipation $-\delta\Delta \theta_t$. In the last period, the mathematicians were wondering whether this dissipation is strong enough to produce some stabilities or not. They got the positive answer for some energy terms via extra assumptions on $|\nabla|^{-1}\theta_1$.\newpage
\begin{remark}
In the formal limit case $\gamma_1=\gamma_2=0$ with $\delta>0$, the temperature difference satisfies the strongly damped wave equation 
\begin{align}\label{Strongly-damped-waves}
\begin{cases}
\Theta_{tt}-\kappa\Delta\Theta-\delta\Delta\Theta_t=0,&x\in\mb{R}^n,\ t>0,\\
(\Theta,\Theta_t)(0,x)=(\Theta_0,\Theta_1)(x),&x\in\mb{R}^n,
\end{cases}	
\end{align}
which has been deeply studied by \cite{Shibata=2000,Ikehata-Todorova-Yordanov=2013,Ikehata=2014,Ikehata-Ono=2017} and references given therein. In the case with $\delta=0$ additionally it satisfies the well-known free wave equation. It is worth mentioning that the density and the velocity of linearized compressible Navier-Stokes system satisfy the model \eqref{Strongly-damped-waves} with suitable regular initial data where the parameter $\delta$ consists of the viscosity of compressible fluids. 
\end{remark}

As usual (e.g. \cite{Zhang-Zuazua=2003,Quintanilla-Racke=2003,Yang-Wang=2006,Jachmann=2008} in the type III model), thanks to the identity
\begin{align*}
\nabla \divv\ml{U} =\nabla\times(\nabla\times\ml{U})+\Delta\ml{U},
\end{align*}
the solution $\ml{U}$ can be decomposed into the potential part $\ml{U}^{p_0}$, i.e. $\nabla\times\ml{U}^{p_0}=0$, and the solenoidal part $\ml{U}^{s_0}$, i.e. $\divv\ml{U}^{s_0}=0$, in a weak sense via the Helmholtz decomposition for lower dimensions, or the decomposition of displacement vector field (with some decay conditions on the initial data \cite{Munoz-Rivera=1993,Chorin-Marsden=1993}) for higher dimensions. Then, the solenoidal part $\ml{U}^{s_0}=\ml{U}^{s_0}(t,x)$ solves the conservative free wave equation
\begin{align}\label{Thermoelastic-Solenoidal}
\begin{cases}
\ml{U}^{s_0}_{tt}-a^2\Delta \ml{U}^{s_0}=0,&x\in\mb{R}^n,\ t>0,\\
(\ml{U}^{s_0},\ml{U}^{s_0}_t)(0,x)=(\ml{U}^{s_0}_0,\ml{U}^{s_0}_1)(x),&x\in\mb{R}^n,
\end{cases}
\end{align}
moreover, the potential part $\ml{U}^{p_0}=\ml{U}^{p_0}(t,x)$ solves the thermoelastic systems
\begin{align}\label{Thermoelastic-Potential}
	\begin{cases}
		\ml{U}^{p_0}_{tt}-b^2\Delta\ml{U}^{p_0}+\gamma_1\nabla\theta=0,&x\in\mb{R}^n,\ t>0,\\
		\theta_{tt}-\kappa\Delta\theta-\delta\Delta\theta_t+\gamma_2\divv\ml{U}^{p_0}_{tt}=0,&x\in\mb{R}^n,\ t>0,\\
		(\ml{U}^{p_0},\ml{U}^{p_0}_t)(0,x)=(\ml{U}^{p_0}_0,\ml{U}^{p_0}_1)(x),&x\in\mb{R}^n,\\
		(\theta,\theta_t)(0,x)=(\theta_0,\theta_1)(x),&x\in\mb{R}^n.
	\end{cases}
\end{align}
Actually, the qualitative properties of solution to the free wave equation \eqref{Thermoelastic-Solenoidal} are well-known in the last centuries (cf. \cite{Strichartz=1970,Peral=1980} and the monograph \cite{Evans=1998}). Concerning the thermoelastic system of type III, i.e. the Cauchy problem \eqref{Thermoelastic-Potential} with $\delta>0$, \cite{Zhang-Zuazua=2003} and \cite{Quintanilla-Racke=2003} independently studied energy decays by using the classical energy method and the spectral method. Afterward, \cite{Reissig-Wang=2005} in 1d and \cite{Yang-Wang=2006} in 3d investigated several qualitative properties for the energy term via the multi-steps diagonalization method (it can show asymptotic representations of unknowns in different frequency zones), including well-posedness, propagation of singularities and $L^p-L^q$ decay estimates with the heat-type decay rate $(1+t)^{-\frac{n}{2}(\frac{1}{p}-\frac{1}{q})}$ on the H\"older's conjugate $1/p+1/q=1$. The parabolic-type asymptotic profile for the energy term was derived in \cite{Jachmann=2008}. In some sense, their large time behavior are mainly (or completely sometimes) described by the heat-like or parabolic-like effect, nevertheless the hyperbolic effect from the elastic waves and the hyperbolic thermal law is ignored. We in this paper will recover some hyperbolic influences via studying the solutions themselves instead of energy terms. For another, concerning the thermoelastic system of type II, i.e. the Cauchy problem \eqref{Thermoelastic-Potential} with $\delta=0$, the well-posed result for the energy term was obtained in \cite{Jachmann=2008} as a corollary of \cite{Reissig-Wang=2005,Yang-Wang=2006}. The other related models to the thermoelastic systems \eqref{Thermoelastic-Potential} with additional lower order terms are referred the interested reader to \cite{Jachmann=2008,Jachmann=2009,Yang=2022}.
\subsection{Motivations and difficulties}
\hspace{5mm}We next pose two main questions in the study of the thermoelastic systems \eqref{Thermoelastic-Potential}, which are also our motivations in this paper. 
 \begin{itemize}
 	\item Excluding the strong singularity $|\xi|^{-1}$ for small frequencies, the authors of \cite{Zhang-Zuazua=2003,Reissig-Wang=2005,Yang-Wang=2006,Jachmann=2008} have to propose some assumptions on $|\nabla|^{-1}\theta_1$ when they studied well-posedness and decay estimates for the energy term $(|\nabla|\ml{U}^{p_0},\ml{U}^{p_0}_t,\theta)$ in the type III model. The main reason for introducing these spaces consists in the use of thermal displacement transformation \eqref{Disp-ther-01} with $|\nabla|\Psi_0\in L^1$ or $L^2$. Note that $|\nabla|\Psi_0$ contains $|\nabla|^{-1}\theta_1$, see \cite[Remark 7.1]{Zhang-Zuazua=2003}, \cite[Lemma 4.1]{Reissig-Wang=2005}, \cite[Remark 2.2]{Yang-Wang=2006} for details. Roughly speaking, among other initial data, \cite{Zhang-Zuazua=2003,Reissig-Wang=2005,Yang-Wang=2006} showed
 	\begin{align}\label{Est-Previous-Pointwise}
 	\chi_{\intt}(\xi)\big(|\xi|\,|\widehat{\ml{U}}^{p_0}|+|\widehat{\ml{U}}_t^{p_0}|+|\widehat{\theta}|\big)\lesssim\chi_{\intt}(\xi)\,\mathrm{e}^{-c|\xi|^2t}\,\big(|\xi|^{-1}|\widehat{\theta}_1|\big),
 	\end{align}
 which leads to $\|\,|\nabla|^{-1}\theta_1\|_{L^1}$ with the heat-type decay rate $t^{-\frac{n}{4}}$ from the Gaussian kernel $\mathrm{e}^{-c|\xi|^2t}$.
 	 Nevertheless, it seems not a natural assumption on initial data, i.e. $\theta_1\in\dot{H}^{-1}_1$ the homogeneous Sobolev space with negative order. The first question arises.
 	\begin{center}
 		\textbf{Question I}: \emph{What are qualitative properties of solutions with the $L^1$ data? Namely, can one drop the assumption $|\nabla|^{-1}$ acting on the initial data?}
 	\end{center}
 This is non-trivial because of $\chi_{\intt}(\xi)\,\mathrm{e}^{-c|\xi|^2t}\,|\xi|^{-1}\notin L^2$ for any $t>0$ if $n\leqslant 2$ and the failure of Hardy-Littlewood-Sobolev inequality \cite{Lieb=1983}, i.e. $L^m\not\hookrightarrow \dot{H}^{-1}_1$ for any $m\geqslant 1$. For these reasons, we are going to try another approach to recover weak oscillating functions from the elastic waves  and the hyperbolic thermal law to compensate this singularity.
 \item  As we mentioned in the above, all related works in terms of the thermoelastic systems of type II and type III focused on the energy terms $|\nabla|\ml{Y},\ml{Y}_t$ with $\ml{Y}\in\{\ml{U}^{p_0},\Psi\}$, where $\Psi=\Psi(t,x)$ denotes the thermal displacement (cf. \cite[Equations (1.20), (1.21)]{Zhang-Zuazua=2003})
 \begin{align}
 \Psi(t,x):=\int_0^t\theta(\tau,x)\,\mathrm{d}\tau+\Psi_0(x) \ \ \mbox{with}\ \ \kappa\Delta\Psi_0(x)=\theta_1(x)-\delta\Delta\theta_0(x)+\gamma\divv u_1(x)\label{Disp-ther-01}
 \end{align} considered in \cite{Zhang-Zuazua=2003,Reissig-Wang=2005,Yang-Wang=2006,Jachmann=2008}, where $\Psi_0(x)$ contains the data $|\nabla|^{-1}\theta_1(x)$ from Question I. Hence, a natural and important question arises.
 \begin{center}
 	\textbf{Question II}: \emph{Can we describe more detailed information of the elastic displacement $\ml{U}^{p_0}$ and the temperature difference $\theta$ themselves for large time?}
 \end{center}
 To the best of knowledge of authors, this question remains open for the thermoelastic systems of type II and type III. Concerning the sharp estimates for $\ml{U}^{p_0}$, from the previous literature and \eqref{Est-Previous-Pointwise}, the solution $\widehat{\ml{U}}^{p_0}$ has the stronger singularity $|\xi|^{-2}$ implying $\chi_{\intt}(\xi)\,\mathrm{e}^{-c|\xi|^2t}\,|\xi|^{-2}\notin L^2$ for any $t>0$ even for higher dimensions $n\leqslant 4$. This strong singularity cannot be completely compensated by the weak oscillating functions (as Question I) from the hyperbolic part of the models, especially, when $n\leqslant 2$. To solve this difficulty, one should deeply understand the large time profiles of solutions. We later will introduce two new structures: \emph{double diffusion waves} and \emph{double waves} for answering this question. This also contributes to understand the hyperbolic thermal effect versus the elastic waves effect.
 \end{itemize}

\subsection{Main purposes of this manuscript}
\hspace{5mm}By denoting $u:=\ml{U}^{p_0}$ for the sake of simplicity from the original problem \eqref{Thermoelastic-Potential}, in the present paper we study the following thermoelastic systems of type II (when $\delta=0$) and type III (when $\delta>0$) in the whole space $\mb{R}^n$ for any $n\in\mb{N}$:
\begin{align}\label{Thermoelastic-Our-Problem}
	\begin{cases}
		u_{tt}-b^2\Delta u+\gamma\nabla\theta=0,&x\in\mb{R}^n,\ t>0,\\
		\theta_{tt}-\kappa\Delta\theta-\delta\Delta\theta_t+\gamma\divv u_{tt}=0,&x\in\mb{R}^n,\ t>0,\\
		(u,u_t)(0,x)=(u_0,u_1)(x),&x\in\mb{R}^n,\\
		(\theta,\theta_t)(0,x)=(\theta_0,\theta_1)(x),&x\in\mb{R}^n,
	\end{cases}
\end{align}
where we took $\gamma_1=\gamma_2=\gamma\in\mb{R}\backslash\{0\}$ without loss of generality (otherwise, one may consider the change of variable $\bar{\theta}(t,x):=\sqrt{\gamma_1/\gamma_2}\,\theta(t,x)$ to get our desired system). Again, we set $b,\kappa>0$ and all initial data belonging to the $L^2$ or $L^1$ frameworks. Our contributions are twofold.
\begin{itemize}
	\item Concerning the type III model ($\delta>0$), we find the critical dimensions $n=4$ for $u$ (optimal growth when $n\leqslant 4$ and optimal decay when $n\geqslant 5$) as well as $n=2$ for $\theta$ (optimal growth when $n\leqslant 2$ and optimal decay when $n\geqslant 3$). Furthermore, their large time profiles are described by the double diffusion waves with different singularities.
	\item Concerning the type II model ($\delta=0$), we derive optimal growth estimates when $n\leqslant 4$ for $u$ and when $n\leqslant 2$ for $\theta$, but bounded estimates for higher dimensions. Furthermore, their large time profiles are described by the double waves with different singularities.
\end{itemize}
The optimality, throughout this manuscript, is guaranteed by the same behavior of upper bound and lower
bound for large time. These results imply that the decisive role in low dimensions is the elastic waves effect, and the hyperbolic thermal effect works in high dimensions only. In Table \ref{Table-01}, we also compare different influences from the transversal waves $\ml{U}^{s_0}$, the longitudinal waves $\ml{U}^{p_0}$, the thermal waves $\theta$ from the original thermoelastic coupled systems \eqref{Eq-Original-01} for large time.

Our discoveries are the new structures of ``double diffusion waves'' i.e. $\ml{M}(t,x)$ with $c_j>0$, and ``double waves'' i.e. $\ml{M}(t,x)$ with $c_j\equiv0$, which do not appear in other evolution equations (by removing the waves parts $\sin(\nu_j|\xi|t)/(\nu_j|\xi|)$ and the crucial singularity $|\xi|^{-\sigma}$, the so-called ``double diffusion'' was firstly introduced by \cite{D'Abbicco-Ebert=2014}). For this reason, one of our innovations in this paper is to study optimal large time behavior for the following Fourier multipliers:
\begin{align*}
\ml{M}(t,x):=\ml{F}_{\xi\to x}^{-1}\left[\left(\ell_1\frac{\sin(\nu_1|\xi|t)}{\nu_1|\xi|}\,\mathrm{e}^{-c_1|\xi|^2t}-\ell_2\frac{\sin(\nu_2|\xi|t)}{\nu_2|\xi|}\,\mathrm{e}^{-c_2|\xi|^2t}\right)\frac{1}{|\xi|^{\sigma}}\right]
\end{align*}
with $\sigma\in\mb{N}_0$, $\ell_j\neq0$, $\nu_j>0$ and $c_j\geqslant0$ for $j\in\{1,2\}$. It indicates some interactions between two different waves or diffusion waves with the propagation speeds $\nu_1$ and $\nu_2$, respectively. The asymptotic behavior of double (diffusion) waves are quite different from the single (diffusion) waves due to some interactions. For one thing, their large time behavior heavily depend on the parameters $\nu_1,\nu_2,\ell_1,\ell_2$, which will cause the finite time blow-up in some ranges of parameters (see Remark \ref{Rem-parameter-blowup} in detail). For another, in the case $\sigma>0$ there is a stronger singularity $|\xi|^{-\sigma}$ (for small frequencies). We need to utilize the refined WKB analysis and the Fourier analysis in the extended $(t,|\xi|)$-space to derive optimal estimates for $\ml{M}(t,\cdot)$ in the $L^2$ norm. This is also the key to improve the previous results in \cite{Zhang-Zuazua=2003,Reissig-Wang=2005,Yang-Wang=2006,Jachmann=2008} by dropping some unnatural assumptions on the initial data and answer Questions I and II.

\paragraph{\large Notation} The constants $c$ and $C$ may be changed from line to line but are independent of the time variable. We write $f\lesssim g$ if there exists a positive constant $C$ such that $f\leqslant Cg$, analogously for $f\gtrsim g$. The sharp relation $f\approx g$ holds if and only if $g\lesssim f\lesssim g$. We take the following zones of the Fourier space:
\begin{align*}
	\ml{Z}_{\intt}(\varepsilon_0):=\{|\xi|\leqslant\varepsilon_0\ll1\}, \ \ 
	\ml{Z}_{\bdd}(\varepsilon_0,N_0):=\{\varepsilon_0\leqslant |\xi|\leqslant N_0\},\ \ 
	\ml{Z}_{\extt}(N_0):=\{|\xi|\geqslant N_0\gg1\}.
\end{align*}
Moreover, the cut-off functions $\chi_{\intt}(\xi),\chi_{\bdd}(\xi),\chi_{\extt}(\xi)\in \mathcal{C}^{\infty}$ having their supports in the corresponding zones $\ml{Z}_{\intt}(\varepsilon_0)$, $\ml{Z}_{\bdd}(\varepsilon_0/2,2N_0)$ and $\ml{Z}_{\extt}(N_0)$, respectively, such that
\begin{align*}
	\chi_{\bdd}(\xi)=1-\chi_{\intt}(\xi)-\chi_{\extt}(\xi)\ \ \mbox{for all}\ \ \xi \in \mb{R}^n.
\end{align*}
The mean of a summable function $f$ is denoted by $P_f:=\int_{\mb{R}^n}f(x)\,\mathrm{d}x$. We define the weighted space
\begin{align*}
L^{1,1}:=\big\{f\in L^1:\ \|f\|_{L^{1,1}}:=\|(1+|x|)f\|_{L^1}<+\infty \big\}.
\end{align*}
 Lastly, let us define two different time-dependent functions (that will be used in the optimal estimates of solutions) as follows:
\begin{align*}
	\ml{D}_n(1+t):=\begin{cases}
		(1+t)^{2-\frac{n}{2}}&\mbox{if}\ \ n\in\{1,3\},\\
		(1+t)^{2-\frac{n}{2}}[\ln(\mathrm{e}+t)]^{\frac{1}{2}}&\mbox{if}\ \ n\in\{2,4\},\\
		(1+t)^{1-\frac{n}{4}}&\mbox{if}\ \ n\geqslant5,
	\end{cases}
\end{align*}
containing an additional logarithmic growth when $n\in\{2,4\}$, and
\begin{align*}
 \ml{E}_n(1+t):=\begin{cases}
		(1+t)^{\frac{1}{2}}&\mbox{if}\ \ n=1,\\
		[\ln(\mathrm{e}+t)]^{\frac{1}{2}}&\mbox{if}\ \ n=2,\\
		(1+t)^{\frac{1}{2}-\frac{n}{4}}&\mbox{if}\ \ n\geqslant3,
	\end{cases}
\end{align*}
containing an additional logarithmic growth when $n=2$.

\section{Main results}\label{Section_Results}\setcounter{equation}{0}
\subsection{Results for the thermoelastic system of type III}
\hspace{5mm}In order to describe large time behavior of solutions $u,\theta$, respectively, strongly motivated by $\widehat{\ml{G}}_{1}$ and $\widehat{\ml{G}}_2$ in Subsection \ref{Subsection-Pointwise-Estimates}, let us first introduce the profiles $\varphi=\varphi(t,x)\in\mb{R}^n$ and $\psi=\psi(t,x)\in\mb{R}$ via Fourier multipliers as follows:
\begin{align*}
\varphi&:=\ml{F}^{-1}_{\xi\to x}\left[\frac{\gamma}{\alpha_2}\left(\frac{\sin(\nu_1|\xi|t)}{\nu_1|\xi|}\,\mathrm{e}^{-c_1|\xi|^2t}-\frac{\sin(\nu_2|\xi|t)}{\nu_2|\xi|}\,\mathrm{e}^{-c_2|\xi|^2t}\right)\frac{i\xi}{|\xi|^2}\right]P_{\theta_1},\\
\psi&:=\ml{F}^{-1}_{\xi\to x}\left[-\frac{1}{2\alpha_2}\left((\alpha_0-\alpha_2)\,\frac{\sin(\nu_1|\xi|t)}{\nu_1|\xi|}\,\mathrm{e}^{-c_1|\xi|^2t}-(\alpha_0+\alpha_2)\,\frac{\sin(\nu_2|\xi|t)}{\nu_2|\xi|}\,\mathrm{e}^{-c_2|\xi|^2t}\right)
\right]P_{\theta_1},
\end{align*}
where we denote some constants throughout this paper by
\begin{align*}
\alpha_0:=b^2-\kappa-\gamma^2,\ \ \alpha_1:=b^2+\kappa+\gamma^2,\ \ \alpha_2:=\sqrt{\alpha_1^2-4b^2\kappa},
\end{align*} and
\begin{align}\label{parameters}
	\nu_1:=\sqrt{\frac{1}{2}\,(\alpha_1+\alpha_2)},\ \ \nu_2:=\sqrt{\frac{1}{2}\,(\alpha_1-\alpha_2)},\ \ c_1:=\frac{\delta}{4}\left(-\frac{\alpha_0}{\alpha_2}+1\right),\ \ c_2:=\frac{\delta}{4}\left(\frac{\alpha_0}{\alpha_2}+1\right).
\end{align}
Notice that $\nu_1,\nu_2,c_1,c_2>0$. Indeed, from the relations $\nu_1^2+\nu_2^2=b^2+\kappa+\gamma^2$ and $\nu_1^2\nu_2^2=b^2\kappa$, one notices that the propagation speeds $\nu_j$ are related to the coupling coefficient $\gamma$ also due to the equivalent form of \eqref{Thermoelastic-Our-Problem} being
\begin{align*}
\begin{cases}
(\partial_t^2-b^2\Delta)u+\gamma\nabla\theta=0,&x\in\mb{R}^n,\ t>0,\\
(\partial_t^2-(\kappa+\gamma^2)\Delta-\delta\Delta\partial_t)\theta+\gamma b^2\Delta\divv u=0,&x\in\mb{R}^n,\ t>0,
\end{cases}
\end{align*} 
which implies similar behavior between $\divv u$ (in turn $u$) and $\theta$.

\begin{theorem}\label{Thm-III-U}
Let us assume $(u_0,u_1)\in (L^2\cap L^1)^n\times (L^2\cap L^1)^n$ and $(\theta_0,\theta_1)\in (L^2\cap L^1)\times (L^2\cap L^1)$ for any $n\geqslant 1$. The elastic displacement $u=u(t,x)$ to the thermoelastic system of type III, i.e. the Cauchy problem \eqref{Thermoelastic-Our-Problem} with $\delta>0$, globally in time exists such that $u\in(\ml{C}([0,+\infty),L^2))^n$ and satisfies the following upper bound estimates:
\begin{align*}
	\|u(t,\cdot)\|_{(L^2)^n}&\lesssim\ml{D}_n(1+t) \left(\|(u_0,u_1)\|_{(L^2\cap L^1)^n\times (L^2\cap L^1)^n}+\|(\theta_0,\theta_1)\|_{(L^2\cap L^1)\times (L^2\cap L^1)}\right),\\
	\|u(t,\cdot)-\varphi(t,\cdot)\|_{(L^2)^n}&=
	o\big(\ml{D}_n(1+t)\big)\ \ \mbox{for large time}\ \ t\gg1.
\end{align*}
Furthermore, let us assume $\alpha_1<3\alpha_2$ if $n=4$ (no any restriction in other dimensions). Then, it satisfies the following lower bound estimates:
\begin{align*}
\|u(t,\cdot)\|_{(L^2)^n}\gtrsim\ml{D}_n(1+t)|P_{\theta_1}|
\end{align*}
for large time $t\gg1$, provided that $P_{\theta_1}\neq0$.
\end{theorem}
\begin{remark}
Combining the derived estimates in Theorem \ref{Thm-III-U}, the solution $u$ satisfies the optimal large time  estimates
\begin{align*}
\|u(t,\cdot)\|_{(L^2)^n}\approx \ml{D}_n(1+t)
\end{align*}
provided that $P_{\theta_1}\neq0$. It implies the infinite time $L^2$-blowup of $u(t,\cdot)$ when $n\leqslant 4$ and the polynomially decay estimates when $n\geqslant 5$, namely, the critical dimension for $u$ is $n=4$. Particularly, the new growth rates of polynomial-logarithmic type are found. Actually, the decay structure of $u$ is completely provided by the viscous dissipation $-\delta\Delta\theta_t$ indirectly via the coupling terms $(\gamma\nabla\theta,\gamma\divv u_{tt})^{\mathrm{T}}$. Hence, the damping effect of $u$ is weak.
\end{remark}
\begin{remark}
By subtracting the function $\varphi(t,\cdot)$ in the $L^2$ norm, the error estimates in Theorem \ref{Thm-III-U} hold for any $n\geqslant 1$. Thus, the large time profile of $u$ can be explained by $\varphi$ which is the double diffusion waves with the strong singularity. To the best of authors' knowledge, the terminology ``double diffusion waves'' is the first time to be introduced (quite different from the double diffusion in the structurally damped waves model \cite{D'Abbicco-Ebert=2014} due to strong singularities in our waves parts). By denoting the classical diffusion waves (see, for example, \cite{Ikehata-Todorova-Yordanov=2013,Ikehata=2014,Ikehata-Ono=2017})
\begin{align*}
\ml{W}_{\nu_k,c_k}(t,x):=\ml{F}_{\xi\to x}^{-1}\left(\frac{\sin(\nu_k|\xi|t)}{\nu_k|\xi|}\,\mathrm{e}^{-c_k|\xi|^2t}\right)P_{\theta_1}
\end{align*}
with the propagation speed $\nu_k$ and the diffusion coefficient $c_k$, then the double diffusion waves $\varphi$ is the interaction between $\ml{W}_{\nu_1,c_1}(t,x)$ and $\ml{W}_{\nu_2,c_2}(t,x)$ with the strong singularity (for small frequencies). Due to two different propagation speed $\nu_1\neq\nu_2$, some new phenomena arise comparing with the single diffusion waves. More detailed explanations will be shown in Subsection \ref{Subsection-Double-Diffusion-Waves}.
\end{remark}

\begin{theorem}\label{Thm-III-theta}
 Let us assume $(u_0,u_1)\in (H^1\cap L^1)^n\times (L^2\cap L^1)^n$ and $(\theta_0,\theta_1)\in (L^2\cap L^1)\times (L^2\cap L^1)$ for any $n\geqslant 1$. The temperature difference $\theta=\theta(t,x)$ to the thermoelastic system of type III, i.e. the Cauchy problem \eqref{Thermoelastic-Our-Problem} with $\delta>0$, globally in time exists such that $\theta\in\ml{C}([0,+\infty),L^2)$ and satisfies the following upper bound estimates:
	\begin{align*}
	\|\theta(t,\cdot)\|_{L^2}&\lesssim \ml{E}_n(1+t)\left(\|(u_0,u_1)\|_{(H^1\cap L^1)^n\times (L^2\cap L^1)^n}+\|(\theta_0,\theta_1)\|_{(L^2\cap L^1)\times (L^2\cap L^1)}\right),\\
		\|\theta(t,\cdot)-\psi(t,\cdot)\|_{L^2}&=o\big(\ml{E}_n(1+t)\big)\ \ \mbox{for large time}\ \ t\gg1.
	\end{align*}
Furthermore, let us assume $(\alpha_0-\alpha_2)^2(\alpha_1-\alpha_2)>2(\alpha_0+\alpha_2)^2(\alpha_1+\alpha_2)$ or $(\alpha_0+\alpha_2)^2(\alpha_1+\alpha_2)>2(\alpha_0-\alpha_2)^2(\alpha_1-\alpha_2)$ if $n=2$ (no any restriction in other dimensions). Then, it satisfies the following lower bound estimates:
\begin{align*}
	\|\theta(t,\cdot)\|_{L^2}\gtrsim\ml{E}_n(1+t)|P_{\theta_1}|
\end{align*}
for large time $t\gg1$, provided that $P_{\theta_1}\neq0$. The last optimal rates are exactly the same as those for the strongly damped wave equation \cite{Ikehata=2014,Ikehata-Ono=2017} (i.e. the Cauchy problem \eqref{Thermoelastic-Our-Problem} with $\delta>0$ and $\gamma=0$ formally).
\end{theorem}
\begin{remark}
	Combining the derived estimates in Theorem \ref{Thm-III-theta}, the solution $\theta$ satisfies the optimal large time estimates
	\begin{align*}
		\|\theta(t,\cdot)\|_{L^2}\approx \ml{E}_n(1+t)
	\end{align*}
	provided that $P_{\theta_1}\neq0$. It implies the infinite time $L^2$-blowup of $\theta(t,\cdot)$ when $n\leqslant 2$ and the polynomially decay estimates when $n\geqslant 3$, namely, the critical dimension for $\theta$ is $n=2$. It is not surprising that the temperature difference $\theta$ decays faster [resp. grows slower] than the elastic displacement $u$ because of the viscous dissipation $-\delta\Delta\theta_t$ (directly on $\theta$ but indirectly on $u$) in the coupled system \eqref{Thermoelastic-Our-Problem}.
\end{remark}
\begin{remark}
	By subtracting the function $\psi(t,\cdot)$ in the $L^2$ norm, the error estimates in Theorem \ref{Thm-III-theta} hold for any $n\geqslant 1$. Thus, the large time profile of $\theta$ can be explained by $\psi$ which is the double diffusion waves with the weak singularity.
\end{remark}

The last two results have concluded sharp large time behavior for the solutions $u,\theta$ themselves with the usual $L^1$ integrable data (without the operator $|\nabla|^{-1}$), which answer Questions I and II from our introduction. Let us compare our results with those in the previous references.
\begin{itemize}
	\item We claim the existence of $(u,\theta)\in (\ml{C}([0,+\infty),L^2))^{n+1}$ to be a supplement of \cite[Theorem 4.1]{Reissig-Wang=2005} and \cite[Theorem 1.1]{Yang-Wang=2006}. Precisely, we do not require the higher regular data with $s\geqslant 4$ from \cite{Reissig-Wang=2005} in the 1d model, $s\geqslant 2$ from \cite{Yang-Wang=2006} in the 3d model, where the assumptions on $|\nabla|^{-\alpha}\theta_1$ with $\alpha\in(1/2,1]$ mainly from \cite[Lemma 4.1]{Reissig-Wang=2005}  have been dropped also. Our results show the existence of solutions themselves instead of the energy terms. Furthermore, the new optimal growth/decay estimates indicate different large time behavior for $u$ and $\theta$, particularly, the growth phenomena (infinite time $L^2$-blowup) in lower dimensions are of interest.
	\item The large time asymptotic profiles $\varphi$ and $\psi$ for the solutions $u$ and $\theta$, respectively, are described by the double diffusion waves with different degrees of singularities (for small frequencies). For this reason, we discover the important influence from the hyperbolic  (elastic) part of the thermoelastic system \eqref{Thermoelastic-Our-Problem}, instead of the pure parabolic decay effect stated in \cite{Reissig-Wang=2005,Yang-Wang=2006,Jachmann=2008}. These new results did not be found in the literature.
	\item It is worth pointing out that the studies on sharp large time behavior for energy terms are easier than those for solutions themselves due to the weaker singularities for small frequencies and the exponential decay for large frequencies. Concerning the energy terms with usual $H^s\cap L^1$ data (carrying suitable $s\geqslant 0$), by deriving some pointwise estimates for $|\xi|\widehat{u},\widehat{u}_t$ as generalizations of Propositions \ref{Prop-small}-\ref{Prop-bdd}, it is not difficult to get
	\begin{align*}
		\|\nabla u(t,\cdot)\|_{(L^2)^n}+\|u_t(t,\cdot)\|_{(L^2)^n}+\|\theta(t,\cdot)\|_{L^2}\approx \ml{E}_n(1+t)
	\end{align*}
	for large time $t\gg1$, provided that $P_{\theta_1}\neq0$. It drops the assumption on $|\nabla|^{-1}\theta_1$ from \cite{Zhang-Zuazua=2003,Reissig-Wang=2005,Yang-Wang=2006,Jachmann=2008}, e.g. \cite[Theorem 7.1]{Zhang-Zuazua=2003} required $\nabla(-\Delta)^{-1}(\theta_1-\Delta\theta_0+\divv u_1)\in (L^1)^n$. The major tools are singular estimates in Lemmas \ref{Lemma-Upper-Bound}-\ref{Lemma-Improve}, which allow us to treat the singularity $|\xi|^{-\sigma}$ with $\sigma\in\mb{N}_+$ for small frequencies sharply, rather than to transport the singularity $|\xi|^{-\sigma}$ to the initial data $\widehat{\theta}_1$ as references.
	If one considers additional assumptions on $|\nabla|^{-1}\theta_1$, the above energy terms decay polynomially with the optimal rates $t^{-\frac{n}{4}}$ for large time $t\gg1$. Note that these rates exactly coincide with the previous literature \cite[Theorem 7.1]{Zhang-Zuazua=2003} and \cite[Theorem 5.1 with $q=2,p=1$]{Reissig-Wang=2005} and \cite[Theorem 1.5 with $q=2,p=1$]{Yang-Wang=2006}.
\end{itemize}
\begin{remark}\label{Remark-Compare-with-Fourier}
By formally taking $\kappa=0$ in \eqref{Thermoelastic-Our-Problem} and integrating the resultant equation for $\theta$ over $[0,t]$, the type III model is reduced to the classical thermoelastic system with the Fourier law of heat conduction. If the compatibility condition on initial data does not hold, this is a singular limit process as $\kappa\downarrow 0$. In the recent results \cite[Theorem 2.2 and Remark 2.2]{Chen-Takeda=2023}, the authors found the optimal leading term (profile) for $u$ as $t\gg1$ by the single diffusion waves with the weak singularity $|\xi|^{-1}$, and for $\theta$ by the diffusion without singularity. However, replacing the Fourier law ($\delta=0$) by the hyperbolic thermal law from Green-Naghdi ($\delta>0$) in thermoelastic systems, we observe it produces stronger singularities for $u$ and $\theta$.
\end{remark}

\subsection{Results for the thermoelastic system of type II}
\hspace{5mm}Let us denote $\widetilde{u},\widetilde{\theta}$ as the solutions to the thermoelastic system of type II, i.e. the Cauchy problem \eqref{Thermoelastic-Our-Problem} with $\delta=0$, for clarity to avoid repetitions of symbol (with the type III model). Also note that the superscript ``$\widetilde{\quad}$'' is always used for the type II model.
\begin{theorem}\label{Thm-II-U}Let us assume $(\widetilde{u}_0,\widetilde{u}_1)\in(L^2)^n\times (L^2\cap L^1)^n$ and $(\widetilde{\theta}_0,\widetilde{\theta}_1)\in(L^2\cap L^1)\times (L^2\cap L^1)$ if $n\in\{1,2,3,4\}$. The elastic displacement $\widetilde{u}=\widetilde{u}(t,x)$ to the thermoelastic system of type II, i.e. the Cauchy problem \eqref{Thermoelastic-Our-Problem} with $\delta=0$, globally in time exists such that $\widetilde{u}\in(\ml{C}([0,+\infty),L^2))^n$ and satisfies the following upper bound estimates:
\begin{align*}
\|\widetilde{u}(t,\cdot)\|_{(L^2)^n}&\lesssim\ml{D}_n(1+t)\big(\|(\widetilde{u}_0,\widetilde{u}_1)\|_{(L^2)^n\times (L^2\cap L^1)^n}+\|(\widetilde{\theta}_0,\widetilde{\theta}_1)\|_{(L^2\cap L^1)\times (L^2\cap L^1)}\big).
\end{align*}
Furthermore, let us assume $\alpha_1<3\alpha_2$ if $n=4$ (no any restriction in other dimensions). Then, taking $\widetilde{\theta}_1\in L^{1,1}$ additionally, it satisfies the following lower bound estimates:
\begin{align*}
	\|\widetilde{u}(t,\cdot)\|_{(L^2)^n}\gtrsim\ml{D}_n(1+t)|P_{\widetilde{\theta}_1}|
\end{align*}
for large time $t\gg1$, provided that $P_{\widetilde{\theta}_1}\neq0$.
\end{theorem}

\begin{theorem}\label{Thm-II-theta} Let us assume $(\widetilde{u}_0,\widetilde{u}_1)\in ({H}^1)^n\times (L^2)^n$ and $(\widetilde{\theta}_0,\widetilde{\theta}_1)\in L^2\times (L^2\cap L^1)$ if $n\in\{1,2\}$. The temperature difference $\widetilde{\theta}=\widetilde{\theta}(t,x)$ to the thermoelastic system of type II, i.e. the Cauchy problem \eqref{Thermoelastic-Our-Problem} with $\delta=0$, globally in time exists such that $\widetilde{\theta}\in\ml{C}([0,+\infty),L^2)$ and satisfies the following upper bound estimates:
	\begin{align*}
		\|\widetilde{\theta}(t,\cdot)\|_{L^2}&\lesssim\ml{E}_n(1+t)\big(\|(\widetilde{u}_0,\widetilde{u}_1)\|_{({H}^1)^n\times (L^2)^n}+\|(\widetilde{\theta}_0,\widetilde{\theta}_1)\|_{L^2\times (L^2\cap L^1)}\big).
	\end{align*}
Furthermore, let us assume $(\alpha_0-\alpha_2)^2(\alpha_1-\alpha_2)>2(\alpha_0+\alpha_2)^2(\alpha_1+\alpha_2)$ or $(\alpha_0+\alpha_2)^2(\alpha_1+\alpha_2)>2(\alpha_0-\alpha_2)^2(\alpha_1-\alpha_2)$ if $n=2$ (no any restriction in $n=1$). Then, taking $\widetilde{\theta}_1\in L^{1,1}$ additionally, it satisfies the following lower bound estimates:
\begin{align*}
	\|\widetilde{\theta}(t,\cdot)\|_{L^2}\gtrsim\ml{E}_n(1+t)|P_{\widetilde{\theta}_1}|
\end{align*}
for large time $t\gg1$, provided that $P_{\widetilde{\theta}_1}\neq0$. The last optimal rates are exactly the same as those for the free wave equation \cite{Ikehata=2023} (i.e. the Cauchy problem \eqref{Thermoelastic-Our-Problem} with $\delta=0$ and $\gamma=0$ formally).
\end{theorem}
\begin{remark}
The assumption $\widetilde{\theta}_1\in L^{1,1}$ seems not essential. It may be improved by $\widetilde{\theta}_1\in L^1$ via the Lebesgue dominated convergence theorem and $\ml{F}(\widetilde{\theta}_1)(0)=P_{\widetilde{\theta}_1}$ by following \cite[Proposition A.1]{Chen-Takeda=2023}. But one needs to carefully analyze the $L^1$ integrability for control functions.
\end{remark}
\begin{remark}
By combining all derived estimates in Theorems \ref{Thm-II-U} and \ref{Thm-II-theta}, the solutions $\widetilde{u},\widetilde{\theta}$ in low dimensions satisfy the optimal large time estimates
\begin{align*}
\|\widetilde{u}(t,\cdot)\|_{(L^2)^n}\approx \ml{D}_n(1+t)\ \ \mbox{and}\ \ \|\widetilde{\theta}(t,\cdot)\|_{L^2}\approx \ml{E}_n(1+t)
\end{align*}
provided that $P_{\widetilde{\theta}_1}\neq0$. They imply the infinite time $L^2$-blowup of $\widetilde{u}(t,\cdot)$ when $n\leqslant 4$ and $\widetilde{\theta}(t,\cdot)$ when $n\leqslant 2$. In a sense,  the infinite time $L^2$-blowup may express quantitatively failure of the Huygens principle, which is surprising in 3d because $\widetilde{u}$ solves the strictly hyperbolic equation \eqref{Hyperbolic-Eq}.
\end{remark}
\begin{remark} Let us state some influences from the elastic waves part versus the thermal part (from the hyperbolic heat conduction) in the thermoelastic system \eqref{Thermoelastic-Our-Problem} with $\delta>0$ and $\delta=0$.
 In low dimensions (i.e. $n\leqslant 4$ for the elastic displacement; $n\leqslant 2$ for the temperature difference), the elastic waves part plays a decisive role due to the same optimal large time growth rates between $(u,\theta)$ and $(\widetilde{u},\widetilde{\theta})$, correspondingly. In high dimensions, the thermal effect works due to the polynomial decay estimates for the type III model, whereas the decay rates are weakened by the elastic waves effect (e.g. the slower decay rates $t^{1-\frac{n}{4}}$ for $u$ than those for the Gaussian kernel). In conclusion, the critical dimensions are $n=4$ for the elastic displacement and $n=2$ for the temperature difference, respectively, to distinguish their large time growth/decay behavior.
\end{remark}

Up to now there is only the well-posedness for the Cauchy problem \eqref{Thermoelastic-Our-Problem} with $\delta=0$ in 1d as far as we know. Precisely, \cite[Theorem 3.16]{Jachmann=2008} stated the existence result with the initial data $|\nabla|^{-1}\widetilde{\theta}_1\in L^2$ when $n=1$. In our theorems, this unnatural assumption is dropped and we discover the new optimal growth estimates.

To end this part, strongly motivated by $\widehat{\ml{G}}_{3}$ and $\widehat{\ml{G}}_4$ in Subsection \ref{Subsection-Pointwise}, we next introduce the profiles $\widetilde{\varphi}=\widetilde{\varphi}(t,x)\in\mb{R}^n$ and $\widetilde{\psi}=\widetilde{\psi}(t,x)\in\mb{R}$ via Fourier multipliers as follows:
\begin{align*}
	\widetilde{\varphi}&:=\ml{F}^{-1}_{\xi\to x}\left[\frac{\gamma}{\alpha_2}\left(\frac{\sin(\nu_1|\xi|t)}{\nu_1|\xi|}-\frac{\sin(\nu_2|\xi|t)}{\nu_2|\xi|}\right)\frac{i\xi}{|\xi|^2}\right]P_{\widetilde{\theta}_1},\\
	\widetilde{\psi}&:=\ml{F}^{-1}_{\xi\to x}\left[-\frac{1}{2\alpha_2}\left((\alpha_0-\alpha_2)\,\frac{\sin(\nu_1|\xi|t)}{\nu_1|\xi|}-(\alpha_0+\alpha_2)\,\frac{\sin(\nu_2|\xi|t)}{\nu_2|\xi|}\right)
	\right]P_{\widetilde{\theta}_1}.
\end{align*}
The asymptotic profiles satisfy the consistencies $\widetilde{\varphi}\equiv\varphi|_{\delta=0}$ and $\widetilde{\psi}\equiv\psi|_{\delta=0}$, which are quite reasonable since the type III model will turn to the type II model formally as $\delta=0$.
In order to investigate large time profiles, let us define the cut-off functions $\chi_{\mathrm{c},j}(|D|)$ in the Fourier space such that
\begin{align*}
\mathrm{supp}\,\chi_{\mathrm{c},1}(|\xi|)\subset\begin{cases}
[0,\mu_6/t]&\mbox{if}\ \ n=1,3,\\
[\rho_1/t,+\infty)&\mbox{if}\ \ n=2,\\
[1/t,1/\sqrt{t}\,]&\mbox{if}\ \ n=4,
\end{cases}
\ \  \ \ 
\mathrm{supp}\,\chi_{\mathrm{c},2}(|\xi|)\subset\begin{cases}
	[0,\mu_6/t]&\mbox{if}\ \ n=1,\\
[1/t,1/\sqrt{t}\,]&\mbox{if}\ \ n=2,
\end{cases}
\end{align*}
with a small constant $\mu_6>0$ (such that $\max\{\nu_1,\nu_2\}\mu_6\ll 1$) and a fixed constant $\rho_1>0$ (such that $\nu_2<3\pi/(4\rho_1)<\nu_1$).
 In general considering the localized $L^2$ space
 \begin{align*}
 L^2_{\chi,j}:=\left\{f\in\ml{S}':\ \|f\|_{L^2_{\chi,j}}:=\left\|\chi_{\mathrm{c},j}(|D|)\,\mathrm{e}^{\Delta}f\,\right\|_{L^2}\equiv\left\|\chi_{\mathrm{c},j}(|\xi|)\,\mathrm{e}^{-|\xi|^2}\widehat{f}(\xi)\right\|_{L^2}<+\infty  \right\},
 \end{align*}
 we have obtained the following error estimates in the proofs of lower bound estimates from Theorems \ref{Thm-II-U} and \ref{Thm-II-theta} (for example \eqref{Est-16}, \eqref{Est-17}, \eqref{Est-18} for the unknown $\widetilde{u}$):
\begin{align*}
\|\widetilde{u}(t,\cdot)-\widetilde{\varphi}(t,\cdot)\|_{(L^2_{\chi,1})^n}&=o\big(\ml{D}_n(1+t)\big)\ \ \mbox{if}\ \ n\in\{1,2,3,4\},\\
\|\widetilde{\theta}(t,\cdot)-\widetilde{\psi}(t,\cdot)\|_{L^2_{\chi,2}}&=o\big(\ml{E}_n(1+t)\big)\ \ \, \mbox{if}\ \ n\in\{1,2\},
\end{align*}
for large time $t\gg1$, whose time-dependent functions (i.e. their rates) are addressed explicitly in the demonstrations. By subtracting the functions $\widetilde{\varphi}$ and $\widetilde{\psi}$, respectively, the last error estimates with slower growth rates hold in comparison with those in Theorems \ref{Thm-II-U} and \ref{Thm-II-theta}. For this reason, the large time asymptotic profiles of $\widetilde{u}$ and $\widetilde{\theta}$ can be explained by $\widetilde{\varphi}$ and $\widetilde{\psi}$, which are the double waves with different degrees of singularities. To the best of authors' knowledge, the terminology ``double waves'' is the first time to be introduced. By denoting the classical waves (see, for example, \cite{Strichartz=1970,Brenner=1975,Peral=1980,Reissig-Yagdjian=2000,Ebert-Reissig=2018,Ikehata=2023})
\begin{align*}
\widetilde{\ml{W}}_{\nu_k}(t,x):=\ml{W}_{\nu_k,0}(t,x)=\ml{F}_{\xi\to x}^{-1}\left(\frac{\sin(\nu_k|\xi|t)}{\nu_k|\xi|}\right)P_{\widetilde{\theta}_1}
\end{align*}
with the propagation speed $\nu_k$, then the double waves $\widetilde{\varphi}$ is the interaction between $\widetilde{\ml{W}}_{\nu_1}(t,x)$ and $\widetilde{\ml{W}}_{\nu_2}(t,x)$ with the strong singularity, while $\widetilde{\psi}$ is the interaction with the weak singularity. The strong singularity causes the faster growth rates of $\widetilde{u}$ than those of $\widetilde{\theta}$.
\begin{remark}
	By combining Remark \ref{Rem-High-Dimensions} and \eqref{Est-12}, \eqref{Est-10}, one may conclude some bounded estimates  (and the global in time existence in $L^2$) for solutions to the thermoelastic system of type II in high dimensions as follows:
	\begin{align*}
	\|\widetilde{u}(t,\cdot)\|_{(L^2)^n}&\lesssim \|(\widetilde{u}_0,\widetilde{u}_1)\|_{(L^2)^n\times (L^2\cap L^1)^n}+\|(\widetilde{\theta}_0,\widetilde{\theta}_1)\|_{(L^2\cap L^1)\times (L^2\cap L^1)}\ \ \mbox{if}\ \ n\geqslant 5,\\
	\|\widetilde{\theta}(t,\cdot)\|_{L^2}&\lesssim\|(\widetilde{u}_0,\widetilde{u}_1)\|_{(H^1)^n\times (L^2)^n}+\|(\widetilde{\theta}_0,\widetilde{\theta}_1)\|_{L^2\times (L^2\cap L^1)}\ \  \qquad \ \ \ \mbox{if}\ \ n\geqslant 3,
	\end{align*}
for any $t>0$. However, their sharp upper/lower bounds are still not clear.
\end{remark}

\section{Large time behavior for the thermoelastic system of type III}\label{Section_type III}\setcounter{equation}{0}
\subsection{Pretreatment by reductions}
\hspace{5mm}In this section we consider $\delta>0$ in the thermoelastic system \eqref{Thermoelastic-Our-Problem}. To deeply study the solution itself, let us apply the strongly damped wave operator $(\partial_t^2-\kappa\Delta-\delta\Delta\partial_t)$ on \eqref{Thermoelastic-Our-Problem}$_1$ and combine \eqref{Thermoelastic-Our-Problem}$_2$, thanks to $\nabla \divv u=\Delta u$ from the rotational-free condition $\nabla\times u=0$, the vector unknown $u=u(t,x)$ satisfies
\begin{align*}
\begin{cases}
\ml{L}_{\mathrm{Type-III}}(\partial_t,\Delta)u=0,&x\in\mb{R}^n,\ t>0,\\
\partial_t^ju(0,x)=u_j(x)\ \ \mbox{for}\ \ j\in\{0,1,2,3\},&x\in\mb{R}^n,
\end{cases}
\end{align*}
where the scalar partial differential operator (not hyperbolic) is denoted by
\begin{align*}
\ml{L}_{\mathrm{Type-III}}(\partial_t,\Delta):=\partial_t^4-\delta\Delta\partial_t^3-(b^2+\kappa+\gamma^2)\Delta\partial_t^2+b^2\delta\Delta^2\partial_t+b^2\kappa\Delta^2
\end{align*}
and the initial data are $u_0(x),u_1(x)$ defined in the original model as well as
\begin{align*}
u_2(x):=b^2\Delta u_0(x)-\gamma\nabla\theta_0(x),\ \ u_3(x):=b^2\Delta u_1(x)-\gamma\nabla \theta_1(x).
\end{align*}
Similarly, the scalar unknown $\theta=\theta(t,x)$ solves
\begin{align*}
	\begin{cases}
		\ml{L}_{\mathrm{Type-III}}(\partial_t,\Delta)\theta=0,&x\in\mb{R}^n,\ t>0,\\
		\partial_t^j\theta(0,x)=\theta_j(x)\ \ \mbox{for}\ \ j\in\{0,1,2,3\},&x\in\mb{R}^n,
	\end{cases}
\end{align*}
where the initial data are $\theta_0(x),\theta_1(x)$ defined in the original model as well as
\begin{align*}
\theta_2(x)&:=-b^2\gamma\Delta\divv u_0(x)+(\kappa+\gamma^2)\Delta \theta_0(x)+\delta\Delta\theta_1(x),\\
\theta_3(x)&:=-b^2\delta\gamma\Delta^2\divv u_0(x)-b^2\gamma\Delta\divv u_1(x)+(\kappa+\gamma^2)\delta\Delta^2\theta_0(x)+(\kappa+\gamma^2)\Delta\theta_1(x)+\delta^2\Delta^2\theta_1(x).
\end{align*}
In other words, we may introduce the (scalar or vector) unknown $v=v(t,x)$ such that $v=u,\theta$, which satisfies
\begin{align}\label{TypeIII-v}
	\begin{cases}
		\ml{L}_{\mathrm{Type-III}}(\partial_t,\Delta)v=0,&x\in\mb{R}^n,\ t>0,\\
		\partial_t^jv(0,x)=v_j(x)\ \ \mbox{for}\ \ j\in\{0,1,2,3\},&x\in\mb{R}^n.
	\end{cases}
\end{align}
We do not separate our considerations by different dimensions like \cite{Quintanilla-Racke=2003,Reissig-Wang=2005,Yang-Wang=2006,Jachmann=2008}.

Although the solutions $u$ and $\theta$ satisfy the same partial differential equation with the operator $\ml{L}_{\mathrm{Type-III}}(\partial_t,\Delta)$, their qualitative properties are quite different because of the regularities of third and fourth initial data. Roughly speaking, one may expect faster decay (or slower growth) estimates but higher regular initial data for $\theta$ than those for $u$.

\subsection{Asymptotic behavior for the characteristic roots}
\hspace{5mm}We just need to analyze the equation for $v$ here. Let us apply the partial Fourier transform with respect to $x$ for the Cauchy problem \eqref{TypeIII-v}, namely,
\begin{align*}
	\begin{cases}
		\widehat{v}_{tttt}+\delta|\xi|^2\widehat{v}_{ttt}+(b^2+\kappa+\gamma^2)|\xi|^2\widehat{v}_{tt}+b^2\delta|\xi|^4\widehat{v}_t+b^2\kappa|\xi|^4\widehat{v}=0,&\xi\in\mb{R}^n,\ t>0,\\
		\partial_t^j\widehat{v}(0,\xi)=\widehat{v}_j(\xi)\ \ \mbox{for}\ \ j\in\{0,1,2,3\},&\xi\in\mb{R}^n,
	\end{cases}
\end{align*}
whose corresponding characteristic equation is given by (see also \cite[Equation (8.8)]{Zhang-Zuazua=2003})
\begin{align}\label{quartic}
	\lambda^4+\delta|\xi|^2\lambda^3+(b^2+\kappa+\gamma^2)|\xi|^2\lambda^2+b^2\delta|\xi|^4\lambda+b^2\kappa|\xi|^4=0.
\end{align}
The explicit $|\xi|$-dependent roots of this quartic is complicated, but we just need to understand their asymptotic behavior.
\begin{itemize}
	\item \textbf{Small Frequencies $\xi\in\ml{Z}_{\intt}(\varepsilon_0)$}: The discriminant of \eqref{quartic} is 
	\begin{align*}
	\triangle_{\mathrm{Disc}}(|\xi|)=16b^2\kappa\big((b^2-\kappa)^2+\gamma^4+2\kappa\gamma^2+2b^2\gamma^2 \big)^2|\xi|^{12}+O(|\xi|^{14})>0
	\end{align*}
and the additional factor is $P_{\mathrm{Disc}}(|\xi|)=8(b^2+\kappa+\gamma^2)|\xi|^2-3\delta^2|\xi|^4>0$ for $\xi\in\ml{Z}_{\intt}(\varepsilon_0)$. It means that there are two pairs of distinct (non-real) complex conjugate roots $\lambda_{1,2}$ and $\lambda_{3,4}$ to the quartic \eqref{quartic}. Via asymptotic expansions to these roots for small frequencies $|\xi|\leqslant \varepsilon_0\ll 1$, the pairwise distinct characteristic roots 
\begin{align*}
\lambda_{1,2}=\lambda_{\mathrm{R}}^{(1)}\pm i\lambda_{\mathrm{I}}^{(1)}\ \ \mbox{and}\ \ \lambda_{3,4}=\lambda_{\mathrm{R}}^{(2)}\pm i\lambda_{\mathrm{I}}^{(2)}
\end{align*}
can be expressed according to
\begin{align*}
\lambda_{\mathrm{R}}^{(1)}&=-c_2|\xi|^2+O(|\xi|^3),\ \  \lambda_{\mathrm{I}}^{(1)}=\nu_2|\xi|+O(|\xi|^3),\\
\lambda_{\mathrm{R}}^{(2)}&=-c_1|\xi|^2+O(|\xi|^3),\ \ \lambda_{\mathrm{I}}^{(2)}=\nu_1|\xi|+O(|\xi|^3),
\end{align*}
where we should recall the notations in \eqref{parameters} that $c_1,c_2,\nu_1,\nu_2>0$.

	\item \textbf{Large Frequencies $\xi\in\ml{Z}_{\extt}(N_0)$}: The discriminant of \eqref{quartic} is 
	\begin{align*}
		\triangle_{\mathrm{Disc}}(|\xi|)=-4b^6\delta^6|\xi|^{18}+O(|\xi|^{16})<0
	\end{align*}
for $\xi\in\ml{Z}_{\extt}(N_0)$. It means that there are two distinct real roots $\lambda_1,\lambda_2$ and a pair of (non-real) complex conjugate roots $\lambda_{3,4}$. Via asymptotic expansions to these roots for large frequencies $|\xi|\geqslant N_0\gg1$, the pairwise distinct characteristic roots can be expressed by
\begin{align*}
\lambda_1=-\delta|\xi|^2+O(|\xi|),\ \ \lambda_2=-\frac{\kappa}{\delta}+O(|\xi|^{-1}),\ \ \lambda_{3,4}=\lambda_{\mathrm{R}}\pm i\lambda_{\mathrm{I}}
\end{align*}
with
\begin{align*}
\lambda_{\mathrm{R}}=-\frac{\gamma^2}{2\delta}+O(|\xi|^{-1}),\ \ \lambda_{\mathrm{I}}=b|\xi|+O(|\xi|^{-1}).
\end{align*}
	\item \textbf{Bounded Frequencies $\xi\in\ml{Z}_{\bdd}(\varepsilon_0,N_0)$}: Let us set $\lambda=i\mu$ with $\mu\in\mb{R}\backslash\{0\}$ for a contradiction. Then, the non-trivial number $\mu=\mu(|\xi|)$ fulfills
	\begin{align*}
		\begin{cases}
		\mu^4-(b^2+\kappa+\gamma^2)|\xi|^2\mu^2+b^2\kappa|\xi|^4=0,\\
	i\delta|\xi|^2(\mu^2-b^2|\xi|^2)\mu=0,	
		\end{cases}\Rightarrow\ \ b^2\gamma^2|\xi|^4=0.
	\end{align*}
It yields a contradiction immediately due to $b^2\gamma^2|\xi|^4>0$ for $\xi\in\ml{Z}_{\mathrm{bdd}}(\varepsilon_0,N_0)$. Namely, there is no pure imaginary root to the quartic \eqref{quartic}. From $\mathrm{Re}\,\lambda_j<0$ for $\xi\in\ml{Z}_{\intt}(\varepsilon_0)\cup\ml{Z}_{\extt}(N_0)$ and any $j\in\{1,2,3,4\}$, we are able to claim $\mathrm{Re}\,\lambda_j<0$ for any $\xi\in\mb{R}^n$. A similar fact was proposed by \cite[Proposition 3.7]{Yang-Wang=2006} in a matrix sense.
\end{itemize}

\subsection{Pointwise estimates in the Fourier space}\label{Subsection-Pointwise-Estimates}
\hspace{5mm}For small frequencies, there are two pairs of distinct conjugate characteristic roots, which allows us to employ the representation \eqref{Rep-two-pairs} with $\widehat{v}=\widehat{u},\widehat{\theta}$. Before stating our refined estimates, let us introduce the crucial kernels
\begin{align*}
\widehat{\ml{G}}_0&:=\left(\cos(\nu_1|\xi|t)\,\mathrm{e}^{-c_1|\xi|^2t}-\cos(\nu_2|\xi|t)\,\mathrm{e}^{-c_2|\xi|^2t}\right)\frac{1}{|\xi|},\\
\widehat{\ml{G}}_1&:=\frac{\gamma}{\alpha_2}\left(\frac{\sin(\nu_1|\xi|t)}{\nu_1|\xi|}\,\mathrm{e}^{-c_1|\xi|^2t}-\frac{\sin(\nu_2|\xi|t)}{\nu_2|\xi|}\,\mathrm{e}^{-c_2|\xi|^2t}\right)\frac{i\xi}{|\xi|^2},\\
\widehat{\ml{G}}_2&:=-\frac{1}{2\alpha_2}\left((\alpha_0-\alpha_2)\,\frac{\sin(\nu_1|\xi|t)}{\nu_1|\xi|}\,\mathrm{e}^{-c_1|\xi|^2t}-(\alpha_0+\alpha_2)\,\frac{\sin(\nu_2|\xi|t)}{\nu_2|\xi|}\,\mathrm{e}^{-c_2|\xi|^2t}\right).
\end{align*} 
Note that $\widehat{\ml{G}}_1=(\widehat{\ml{G}}_1^1,\dots,\widehat{\ml{G}}_1^n)\in\mb{R}^n$ and $\widehat{\ml{G}}_0,\widehat{\ml{G}}_2\in\mb{R}$.

 Roughly speaking, they are the double diffusion waves kernels consisting of two diffusion waves with different propagation speeds $\nu_1\neq\nu_2$ and different diffusion coefficients $c_1\neq c_2$. It seems that $\widehat{\ml{G}}_1$ has a stronger singularity $|\xi|^{-2}$ than those with $|\xi|^{-1}$ for $\widehat{\ml{G}}_0$ and $\widehat{\ml{G}}_2$ as $\xi\in\ml{Z}_{\intt}(\varepsilon_0)$. For briefness, let us denote
\begin{align*}
\widehat{\Xi}_0:=|\widehat{u}_0|+|\widehat{u}_1|+|\widehat{\theta}_0|+|\widehat{\theta}_1|.
\end{align*}
\begin{prop}\label{Prop-small}
Let $\xi\in\ml{Z}_{\intt}(\varepsilon_0)$. The solutions $\widehat{u},\widehat{\theta}$ satisfy the following sharp pointwise estimates in the Fourier space:  
\begin{align}
\chi_{\intt}(\xi)|\widehat{u}|&\lesssim\chi_{\intt}(\xi) \left[\left(1+|\xi|t+\frac{|\sin(\widetilde{c}|\xi|t)|}{|\xi|}\right)\mathrm{e}^{-c|\xi|^2t}+|\widehat{\ml{G}}_0|+|\widehat{\ml{G}}_1|\right]\widehat{\Xi}_0,\label{Est-06}\\[0.5em]
\chi_{\intt}(\xi)|\widehat{u}-\widehat{\ml{G}}_1\widehat{\theta}_1|&\lesssim \chi_{\intt}(\xi)\left[\left(1+|\xi|t+\frac{|\sin(\widetilde{c}|\xi|t)|}{|\xi|}\right)\mathrm{e}^{-c|\xi|^2t}+|\widehat{\ml{G}}_0|\right]\widehat{\Xi}_0,\label{Est-05}
\end{align}
with a positive constant $\widetilde{c}$, and
\begin{align*}
\chi_{\intt}(\xi)|\widehat{\theta}|&\lesssim\chi_{\intt}(\xi)\left(|\widehat{\ml{G}}_2|+\mathrm{e}^{-c|\xi|^2t}\right)\widehat{\Xi}_0,\\
\chi_{\intt}(\xi)|\widehat{\theta}-\widehat{\ml{G}}_2\widehat{\theta}_1|&\lesssim \chi_{\intt}(\xi)\,\mathrm{e}^{-c|\xi|^2t}\,\widehat{\Xi}_0.
\end{align*}
\end{prop}
\begin{proof}
Let us firstly consider $\widehat{v}=\widehat{u}$ in the representation \eqref{Rep-two-pairs}, in which
\begin{align*}
\chi_{\intt}(\xi)\det(\mb{V})&= \chi_{\intt}(\xi)\left(-4\alpha_2^2\nu_1\nu_2|\xi|^6+O(|\xi|^7) \right),
\end{align*}
and
\begin{align*} 
\chi_{\intt}(\xi)[\det(\mb{V}_1^{(1)})+\det(\mb{V}_1^{(2)})]&=\chi_{\intt}(\xi)\left(-4i\gamma\xi\lambda_{\mathrm{I}}^{(1)}\lambda_{\mathrm{I}}^{(2)}\left[(\lambda_{\mathrm{I}}^{(1)})^2-(\lambda_{\mathrm{I}}^{(2)})^2\right]\widehat{\theta}_0+O(|\xi|^6;\widehat{\Xi}_0)\right),\\
\chi_{\intt}(\xi)[\det(\mb{V}_2^{(1)})+\det(\mb{V}_2^{(2)})]&=\chi_{\intt}(\xi)\left(4i\gamma\xi\lambda_{\mathrm{I}}^{(1)}\lambda_{\mathrm{I}}^{(2)}\left[(\lambda_{\mathrm{I}}^{(1)})^2-(\lambda_{\mathrm{I}}^{(2)})^2\right]\widehat{\theta}_0+O(|\xi|^6;\widehat{\Xi}_0)\right),
\end{align*}
whose dominant terms are of opposite sign, as well as
\begin{align*}
\chi_{\intt}(\xi)[\det(\mb{V}_1^{(1)})-\det(\mb{V}_1^{(2)})]&=\chi_{\intt}(\xi)\left(4\gamma\alpha_2\nu_1\xi|\xi|^3\widehat{\theta}_1+O(|\xi|^5;\widehat{\Xi}_0)\right),\\
\chi_{\intt}(\xi)[\det(\mb{V}_2^{(1)})-\det(\mb{V}_2^{(2)})]&=\chi_{\intt}(\xi)\left(-4\gamma\alpha_2\nu_2\xi|\xi|^3\widehat{\theta}_1+O(|\xi|^5;\widehat{\Xi}_0)\right).
\end{align*}
It should be emphasized that there are some cancellations in the sum $\det(\mb{V}_k^{(1)})+\det(\mb{V}_k^{(2)})$ for $k\in\{1,2\}$. Precisely, the $\xi|\xi|^3$-terms are canceled in the sum, which leads to $\det(\mb{V}_k^{(1)})+\det(\mb{V}_k^{(2)})\sim |\xi|^5$ but $\det(\mb{V}_k^{(1)})-\det(\mb{V}_k^{(2)})\sim \xi|\xi|^3$ as small frequencies. 

Recalling the representation \eqref{Rep-two-pairs} in our case
\begin{align*}
	\widehat{u}&=\sum\limits_{k=1,2}\left[\cos(\lambda_{\mathrm{I}}^{(k)}t)\,\frac{\det(\mb{V}_k^{(1)})+\det(\mb{V}_k^{(2)})}{\det(\mb{V})}\,\mathrm{e}^{\lambda_{\mathrm{R}}^{(k)}t}\right]+i\sum\limits_{k=1,2}\left[\sin(\lambda_{\mathrm{I}}^{(k)}t)\,\frac{\det(\mb{V}_k^{(1)})-\det(\mb{V}_k^{(2)})}{\det(\mb{V})}\,\mathrm{e}^{\lambda_{\mathrm{R}}^{(k)}t}\right],
\end{align*}
its dominant part for small frequencies seems to be generated by the second component of last representation (the coefficients of sine functions are of $|\xi|^{-2}$ order; but those of cosine functions are of $|\xi|^{-1}$ order) whose approximation is given by $\widehat{\ml{G}}_1\widehat{\theta}_1$. 

We next verify the last statement from the sine part, i.e. the second one of representation of $\widehat{u}$, firstly. The last asymptotic expressions of determinants imply
\begin{align*}
\chi_{\intt}(\xi)\left|\frac{\det(\mb{V}_1^{(1)})-\det(\mb{V}_1^{(2)})}{\det(\mb{V})}-\frac{4\gamma\alpha_2\nu_1\xi|\xi|^3}{-4\alpha_2^2\nu_1\nu_2|\xi|^6}\,\widehat{\theta}_1\right|&\lesssim\chi_{\intt}(\xi)|\xi|^{-1}\,\widehat{\Xi}_0,\\
\chi_{\intt}(\xi)\left|\frac{\det(\mb{V}_2^{(1)})-\det(\mb{V}_2^{(2)})}{\det(\mb{V})}-\frac{-4\gamma\alpha_2\nu_2\xi|\xi|^3}{-4\alpha_2^2\nu_1\nu_2|\xi|^6}\,\widehat{\theta}_1\right|&\lesssim\chi_{\intt}(\xi)|\xi|^{-1}\,\widehat{\Xi}_0,
\end{align*}
which address the leading terms, respectively,
\begin{align*}
-\frac{\gamma\xi}{\alpha_2\nu_2|\xi|^3}\,\widehat{\theta}_1\ \ \mbox{and}\ \ \frac{\gamma\xi}{\alpha_2\nu_1|\xi|^3}\,\widehat{\theta}_1.
\end{align*}
Moreover, from $\lambda_{\mathrm{R}}^{(1)}=-c_2|\xi|^2+O(|\xi|^3)$ and $\lambda_{\mathrm{R}}^{(2)}=-c_1|\xi|^2+O(|\xi|^3)$  we find that
\begin{align}\label{Est-09}
\chi_{\intt}(\xi)\left|\mathrm{e}^{\lambda_{\mathrm{R}}^{(1)}t}-\mathrm{e}^{-c_2|\xi|^2t}\right|&=\chi_{\intt}(\xi)\,\mathrm{e}^{-c_1|\xi|^2t}\left|\mathrm{e}^{O(|\xi|^3)\,t}-1\right|\notag\\
&=\chi_{\intt}(\xi)\,O(|\xi|^3)\,t\,\mathrm{e}^{-c_1|\xi|^2t}\int_0^1\mathrm{e}^{O(|\xi|^3)\,t\tau}\,\mathrm{d}\tau\notag\\
&\lesssim \chi_{\intt}(\xi)|\xi|^3t\,\mathrm{e}^{-c|\xi|^2t},
\end{align}
and similarly
\begin{align*}
\chi_{\intt}(\xi)\left|\mathrm{e}^{\lambda_{\mathrm{R}}^{(2)}t}-\mathrm{e}^{-c_1|\xi|^2t}\right|\lesssim\chi_{\intt}(\xi)|\xi|^3t\,\mathrm{e}^{-c|\xi|^2t},
\end{align*}
which address the leading terms, respectively, $\mathrm{e}^{-c_2|\xi|^2t}$ and $\mathrm{e}^{-c_1|\xi|^2t}$. Thanks to the mean value theorem and the boundedness of cosine function, one arrives at
\begin{align*}
\chi_{\intt}(\xi)\left(|\sin(\lambda_{\mathrm{I}}^{(1)}t)-\sin(\nu_2|\xi|t)|+|\sin(\lambda_{\mathrm{I}}^{(2)}t)-\sin(\nu_1|\xi|t)|\right)&\lesssim \chi_{\intt}(\xi)|\xi|^3t,
\end{align*}
which addresses the leading terms, respectively, $\sin(\nu_2|\xi|t)$ and $\sin(\nu_1|\xi|t)$.
Collecting the above-mentioned leading terms, the dominant term $\widehat{\ml{G}}_1\widehat{\theta}_1$ arises. Making use of last error estimates associated with the standard triangle inequality, one concludes that
\begin{align*}
\chi_{\intt}(\xi)\left|i\sum\limits_{k=1,2}\left[\sin(\lambda_{\mathrm{I}}^{(k)}t)\,\frac{\det(\mb{V}_k^{(1)})-\det(\mb{V}_k^{(2)})}{\det(\mb{V})}\,\mathrm{e}^{\lambda_{\mathrm{R}}^{(k)}t}\right]-\widehat{\ml{G}}_1\widehat{\theta}_1\right|\lesssim \chi_{\intt}(\xi)\left(|\xi|t+\frac{|\sin(\widetilde{c}|\xi|t)|}{|\xi|}\right)\mathrm{e}^{-c|\xi|^2t}\,\widehat{\Xi}_0
\end{align*}
with a suitable constant $\widetilde{c}=\widetilde{c}(\nu_1,\nu_2)>0$.

  For another, thanks to the opposite sign of dominant terms between $\det(\mb{V}_1^{(1)})+\det(\mb{V}_1^{(2)})$ and $\det(\mb{V}_2^{(1)})+\det(\mb{V}_2^{(2)})$, by the same approach as the above we discover the leading term for the first component of last representation of $\widehat{u}$, precisely,
  \begin{align*}
  	\frac{i\gamma}{\alpha_2}\left(\cos(\nu_1|\xi|t)\,\mathrm{e}^{-c_1|\xi|^2t}-\cos(\nu_2|\xi|t)\,\mathrm{e}^{-c_2|\xi|^2t}\right)\frac{\xi}{|\xi|^2}\,\widehat{\theta}_0=\frac{i\gamma\xi}{\alpha_2|\xi|}\,\widehat{\ml{G}}_0\widehat{\theta}_0.
  \end{align*}
As a consequence, its estimate can be decomposed into two parts
\begin{align*}
\chi_{\intt}(\xi)\left|\sum\limits_{k=1,2}\left[\cos(\lambda_{\mathrm{I}}^{(k)}t)\,\frac{\det(\mb{V}_k^{(1)})+\det(\mb{V}_k^{(2)})}{\det(\mb{V})}\,\mathrm{e}^{\lambda_{\mathrm{R}}^{(k)}t}\right]\right|\lesssim\chi_{\intt}(\xi)\left(|\widehat{\ml{G}}_0|\,|\widehat{\theta}_0|+\mathrm{e}^{-c|\xi|^2t}\,\widehat{\Xi}_0\right),
\end{align*}
whose second term on the right-hand side is benefited from the error estimates
\begin{align*}
	\chi_{\intt}(\xi)\left(|\cos(\lambda_{\mathrm{I}}^{(1)}t)-\cos(\nu_2|\xi|t)|+|\cos(\lambda_{\mathrm{I}}^{(2)}t)-\cos(\nu_1|\xi|t)|\right)&\lesssim \chi_{\intt}(\xi)|\xi|^3t
\end{align*}
and \eqref{Est-09} which produce the additional factor $|\xi|^3t$ (similarly to the sine part) to compensate the singularity $|\xi|^{-1}$ for small frequencies. Remark that
\begin{align*}
\chi_{\intt}(\xi)|\xi|^3t|\widehat{\ml{G}}_0|\lesssim \chi_{\intt}(\xi)\,\mathrm{e}^{-c|\xi|^2t}.
\end{align*}
Combining all derived estimates, the proof of \eqref{Est-05} is completed, which turns to \eqref{Est-06} immediately via the triangle inequality.

Secondly, by following the same way as the above, the representation \eqref{Rep-two-pairs} for $\widehat{v}=\widehat{\theta}$ holds with the following facts:
\begin{align*}
\chi_{\intt}(\xi)|\det(\mb{V}_k^{(1)})+\det(\mb{V}_k^{(2)})|&\lesssim\chi_{\intt}(\xi)|\xi|^6\,\widehat{\Xi}_0\ \ \mbox{for}\ \ k\in\{1,2\},
\end{align*}
and
\begin{align*}
\chi_{\intt}(\xi)[\det(\mb{V}_1^{(1)})-\det(\mb{V}_1^{(2)})]&=\chi_{\intt}(\xi)\left[2i\nu_1(\alpha_2+\alpha_0)\alpha_2|\xi|^5\widehat{\theta}_1+O(|\xi|^6;\widehat{\Xi}_0)\right],\\
\chi_{\intt}(\xi)[\det(\mb{V}_2^{(1)})-\det(\mb{V}_2^{(2)})]&=\chi_{\intt}(\xi)\left[2i\nu_2(\alpha_2-\alpha_0)\alpha_2|\xi|^5\widehat{\theta}_1+O(|\xi|^6;\widehat{\Xi}_0)\right].
\end{align*}
Hence,
\begin{align*}
	\chi_{\intt}(\xi)\left|\frac{\det(\mb{V}_1^{(1)})-\det(\mb{V}_1^{(2)})}{\det(\mb{V})}-\frac{2i\nu_1(\alpha_2+\alpha_0)\alpha_2|\xi|^5}{-4\alpha_2^2\nu_1\nu_2|\xi|^6}\,\widehat{\theta}_1\right|&\lesssim\chi_{\intt}(\xi)\,\widehat{\Xi}_0,\\
	\chi_{\intt}(\xi)\left|\frac{\det(\mb{V}_2^{(1)})-\det(\mb{V}_2^{(2)})}{\det(\mb{V})}-\frac{2i\nu_2(\alpha_2-\alpha_0)\alpha_2|\xi|^5}{-4\alpha_2^2\nu_1\nu_2|\xi|^6}\,\widehat{\theta}_1\right|&\lesssim\chi_{\intt}(\xi)\,\widehat{\Xi}_0.
\end{align*}
It implies that the leading term is understood by
\begin{align*}
i\sum\limits_{k=1,2}\left[\sin(\lambda_{\mathrm{I}}^{(k)}t)\,\frac{\det(\mb{V}_k^{(1)})-\det(\mb{V}_k^{(2)})}{\det(\mb{V})}\,\mathrm{e}^{\lambda_{\mathrm{R}}^{(k)}t}\right]\sim\widehat{\ml{G}}_2\widehat{\theta}_1.
\end{align*}
Then, we are able to get the desired estimates for $\widehat{\theta}$ by following the same philosophy as those for $\widehat{u}$. Our proof is completed.
\end{proof}

For large frequencies, there are a pair of conjugate roots and two distinct real roots, which allows us to employ the representation \eqref{Rep-one-pair} with $\widehat{v}=\widehat{u},\widehat{\theta}$. Different from the refined estimates for small frequencies, there is no any singularity for $\xi\in\ml{Z}_{\extt}(N_0)$. Therefore, we just focus on some exponential decay estimates with the regular Cauchy data.

\begin{prop}\label{Prop-large}
	Let $\xi\in\ml{Z}_{\extt}(N_0)$. The solutions $\widehat{u},\widehat{\theta}$ satisfy the following pointwise estimates in the Fourier space: 
	\begin{align*}
		\chi_{\extt}(\xi)|\widehat{u}|&\lesssim\chi_{\extt}(\xi)\,\mathrm{e}^{-ct}\left(|\widehat{u}_0|+|\xi|^{-1}|\widehat{u}_1|+|\xi|^{-1}|\widehat{\theta}_0|+|\xi|^{-3}|\widehat{\theta}_1|\right),
	\end{align*}
	and
	\begin{align*}
		\chi_{\extt}(\xi)|\widehat{\theta}|&\lesssim\chi_{\extt}(\xi)\,\mathrm{e}^{-ct}\left(|\xi|\,|\widehat{u}_0|+|\xi|^{-1}|\widehat{u}_1|+|\widehat{\theta}_0|+|\widehat{\theta}_1|\right).
	\end{align*}
\end{prop}
\begin{proof}
Thanks to the representation \eqref{Rep-one-pair}, we compute the determinant of Vandermonde matrix
\begin{align*}
\chi_{\extt}(\xi)\det(\mb{V})=\chi_{\extt}(\xi)\left(-2ib^3\delta^3|\xi|^9+O(|\xi|^8)\right),
\end{align*}
and the determinant of $\mb{V}_j$ in different situations that
\begin{itemize}
	\item if $\widehat{v}=\widehat{u}$, then
	\begin{align*}
	\chi_{\extt}(\xi)|\det(\mb{V}_1)|&\lesssim \chi_{\extt}(\xi)\left(|\xi|^5|\widehat{u}_0|+|\xi|^5|\widehat{u}_1|+|\xi|^4|\widehat{\theta}_0|+|\xi|^4|\widehat{\theta}_1|\right),\\
		\chi_{\extt}(\xi)|\det(\mb{V}_{2,3,4})|&\lesssim \chi_{\extt}(\xi)\left(|\xi|^9|\widehat{u}_0|+|\xi|^8|\widehat{u}_1|+|\xi|^8|\widehat{\theta}_0|+|\xi|^6|\widehat{\theta}_1|\right);
	\end{align*}
	\item if $\widehat{v}=\widehat{\theta}$, then
\begin{align*}
	\chi_{\extt}(\xi)|\det(\mb{V}_1)|&\lesssim \chi_{\extt}(\xi)\left(|\xi|^8|\widehat{u}_0|+|\xi|^6|\widehat{u}_1|+|\xi|^7|\widehat{\theta}_0|+|\xi|^7|\widehat{\theta}_1|\right),\\
		\chi_{\extt}(\xi)|\det(\mb{V}_{2,3,4})|&\lesssim \chi_{\extt}(\xi)\left(|\xi|^{10}|\widehat{u}_0|+|\xi|^8|\widehat{u}_1|+|\xi|^9|\widehat{\theta}_0|+|\xi|^9|\widehat{\theta}_1|\right).
\end{align*}
\end{itemize}
Combining the real parts of characteristic roots that
\begin{align*}
\chi_{\extt}(\xi)\left(\mathrm{e}^{\lambda_1t}+\mathrm{e}^{\lambda_2t}+\mathrm{e}^{\lambda_{\mathrm{R}}t}\right)\lesssim\chi_{\extt}(\xi)\,\mathrm{e}^{-ct},
\end{align*}
 the desired exponential decay estimates for $\widehat{u}$ as well as $\widehat{\theta}$ are deduced.
\end{proof}

Finally, thanks to $\mathrm{Re}\,\lambda_j<0$ for bounded frequencies, it is easy to conclude the next result.
\begin{prop}\label{Prop-bdd}
Let $\xi\in\ml{Z}_{\bdd}(\varepsilon_0,N_0)$. The solutions $\widehat{u},\widehat{\theta}$ satisfy the following exponential decay estimates in the Fourier space: 
\begin{align*}
\chi_{\bdd}(\xi)(|\widehat{u}|+|\widehat{\theta}|)\lesssim \chi_{\bdd}(\xi)\,\mathrm{e}^{-ct}\left(|\widehat{u}_0|+|\widehat{u}_1|+|\widehat{\theta}_0|+|\widehat{\theta}_1|\right).
\end{align*}
\end{prop}
\subsection{Optimal large time estimates for double diffusion waves kernels}\label{Subsection-Double-Diffusion-Waves}
\hspace{5mm}As our preparations in order to estimate $u(t,\cdot)$ and $\theta(t,\cdot)$ in the $L^2$ norm optimally, this subsection  contributes to study sharp large time behavior for the following time-dependent function (their leading terms $\widehat{\ml{G}}_1$ and $\widehat{\ml{G}}_2$ are related to the cases with $\sigma=1$ and $\sigma=0$, respectively):
\begin{align*}
	\ml{I}(t;n):=\int_0^{\varepsilon_0}\left|\ell_1\frac{\sin(\beta_1rt)}{\beta_1r}\,\mathrm{e}^{-c_1r^2t}-\ell_2\frac{\sin(\beta_2rt)}{\beta_2r}\,\mathrm{e}^{-c_2r^2t}\right|^2\,r^{n-1-2\sigma}\,\mathrm{d}r
\end{align*}
with $\ell_1,\ell_2\neq0$ generally and $\beta_1,\beta_2,c_1,c_2>0$ such that $c_1\neq c_2$ due to our consideration \eqref{parameters}. Again, the positive constant $\varepsilon_0\ll 1$ denotes the size of small frequencies.

Formally taking $\ell_2=0$ and $\sigma=0$, it will reduce to the classical diffusion waves kernel (see \cite{Shibata=2000,Ikehata-Todorova-Yordanov=2013,Ikehata=2014,Ikehata-Ono=2017} and references therein) in the $L^2$ norm, i.e.
\begin{align}\label{Optimal-Ikehata}
\ml{I}(t;n)|^{\ell_2=0}_{\sigma=0}=\frac{1}{|\mb{S}^{n-1}|}\left\|\ell_1\frac{\sin(\beta_1|\xi|t)}{\beta_1|\xi|}\,\mathrm{e}^{-c_1|\xi|^2t}\right\|_{L^2(|\xi|\leqslant\varepsilon_0)}^2\approx[\ml{E}_n(1+t)]^2
\end{align}
 for large time $t\gg1$, whose optimality has been studied by \cite{Ikehata=2014} for $n\geqslant 3$ and \cite{Ikehata-Ono=2017} for $n\leqslant 2$. Surprisingly, carrying a stronger singularity for small frequencies, the properties for $\ml{I}(t;n)|_{\sigma=1}$ are quite different from those for $\ml{I}(t;n)|_{\sigma=1}^{\ell_2=0}$ because of the blow-up phenomenon
 \begin{align}\label{Blow-up-01}
 \ml{I}(t_0;n)|_{\sigma=1}^{\ell_2=0}=+\infty,\ \ \mbox{namely},\ \ \frac{\sin(\beta_1rt_0)}{\beta_1r}\,\mathrm{e}^{-\frac{c_1}{2}r^2t_0}\,r^{\frac{n-3}{2}}\not\in L^2([0,\varepsilon_0])
 \end{align}
in low dimensions $n\in\{1,2\}$ even at finite time $t_0>0$. This situation will be completely changed, particularly, when $\ell_1=\ell_2\neq0$, due to some cancellations in the $(t,r)$-subzone $\{0\leqslant rt\ll 1\}$. Thus, the study for $\ml{I}(t;n)$ is not simply a generalization of the one for the classical diffusion waves kernel $\ml{I}(t;n)|^{\ell_2=0}$. 
 
 For one thing, motivated by \cite{Ikehata=2014,Ikehata-Ono=2017}, we in this part firstly estimate $\ml{I}(t;n)$ when $n\geqslant 1+2\sigma$ with any $\sigma\in\mb{N}_0$ sharply in the sense of same large time behavior for its upper and lower bounds. To overcome new difficulties from the subtraction ($\ell_2\neq0$) in sharp lower bound estimates,  some suitable parameters $\mu_j>0$ will be introduced as shrinking the domain of integral for different dimensions. The parameters $\epsilon_j>0$ are devoted to relax the restrictions on $\beta_1,\beta_2,\ell_1,\ell_2$.
 
Concerning the remaining case $n<1+2\sigma$, due to the attractive aim in this part being $\ml{I}(t;n)|_{\sigma=0,1}$ (recalling $\widehat{\ml{G}}_1$ and $\widehat{\ml{G}}_2$), we will estimate $\ml{I}(t;n)|_{\sigma=1}$ in low dimensions $n\in\{1,2\}$ only, where we remark that $\{n<1+2\sigma\}\cap \{\sigma=0,n\in\mb{N}_+\}=\emptyset$.
The threshold condition (with respect to the parameters $\ell_1,\ell_2$) of finite time blow-up for $\ml{I}(t;n)|_{\sigma=1}$ is given. To be specific,  $\ml{I}(t;n)|_{\sigma=1}^{\ell_1=\ell_2}$ globally in time exists, where we additionally provide its sharp upper and lower bound estimates. Oppositely, $\ml{I}(t;n)|_{\sigma=1}^{\ell_1\neq\ell_2}$ blows up in finite time at $0<t\leqslant t_0$ with any $t_0>0$.

\subsubsection{Optimal large time estimates for high dimensions}
\hspace{5mm}Let us focus on upper bound estimates for high dimensions $n\geqslant 1+2\sigma$ at first, in which the singular part $r^{n-1-2\sigma}$ of $\ml{I}(t;r)$ becomes regular. In other words, it just needs to deal with the weaker singular parts $(\beta_1r)^{-2}$ and $(\beta_2r)^{-2}$ for $r\leqslant \varepsilon_0$.
\begin{lemma}\label{Lemma-Upper-Bound}
Let $n\geqslant 1+2\sigma$ with any $\sigma\in\mb{N}_0$. Let $\ell_1,\ell_2$ be non-trivial constants, and $\beta_1,\beta_2,c_1,c_2$ be positive constants. Then, the following sharp upper bound estimates:
\begin{align*}
\ml{I}(t;n)\lesssim\begin{cases}
(1+t)&\mbox{if}\ \ n=1+2\sigma,\\
\ln (\mathrm{e}+t)&\mbox{if}\ \ n=2+2\sigma,\\
(1+t)^{1+\sigma-\frac{n}{2}}&\mbox{if}\ \ n\geqslant 3+2\sigma,
\end{cases}
\end{align*}
hold for any $t>0$.
\end{lemma}

\begin{proof}
To begin with the demonstration, concerning bounded time $0<t\leqslant t_0$ it is trivial that
\begin{align*}
\ml{I}(t;n)&\lesssim t^2\sum\limits_{j=1,2}\int_0^{\varepsilon_0}\left|\frac{\sin(\beta_jrt)}{\beta_jrt}\right|^2\,\mathrm{e}^{-2c_jr^2t}\,r^{n-1-2\sigma}\,\mathrm{d}r\\
&\lesssim t^2_0\int_0^{\varepsilon_0}r^{n-1-2\sigma}\,\mathrm{d}r\lesssim 1,
\end{align*}
since $n-1-2\sigma\geqslant0$. Let us next prove our lemma for large time $t\geqslant t_0\gg1$ in three cases (different dimensions $n$)  due to the non-summable singularity $r^{n-3-2\sigma}$ in $r=0$.\\

\noindent\underline{Lower Dimensional Case: $n= 1+2\sigma$.} Clearly, it holds that
\begin{align*}
\ml{I}(t;1+2\sigma)&\lesssim\sum\limits_{j=1,2}\int_0^{\varepsilon_0}|\sin(\beta_jrt)|^2\,\mathrm{e}^{-2c_jr^2t}\,r^{-2}\,\mathrm{d}r\\
&\lesssim\sqrt{t}\sum\limits_{j=1,2}\left(\int_0^{1/\sqrt{t}}+\int_{1/\sqrt{t}}^{+\infty}\right)|\sin(\beta_jw\sqrt{t})|^2\,\mathrm{e}^{-2c_jw^2}\,w^{-2}\,\mathrm{d}w
\end{align*}
with the ansatz $w:=r\sqrt{t}$ (always used throughout this subsection) by separating the domain of integral into two parts suitably via the threshold $1/\sqrt{t}$. For one thing, by $|\sin(\beta_jw\sqrt{t})|\leqslant \beta_jw\sqrt{t}$ one estimates
\begin{align*}
\int_0^{1/\sqrt{t}}|\sin(\beta_jw\sqrt{t})|^2\,\mathrm{e}^{-2c_jw^2}\,w^{-2}\,\mathrm{d}w&\lesssim t\int_0^{1/\sqrt{t}}\mathrm{e}^{-2c_jw^2}\,\mathrm{d}w\lesssim \sqrt{t}.
\end{align*}
For another, an application of integration by parts shows
\begin{align}\label{Est-02}
\int_{1/\sqrt{t}}^{+\infty}|\sin(\beta_jw\sqrt{t})|^2\,\mathrm{e}^{-2c_jw^2}\,w^{-2}\,\mathrm{d}w&\lesssim\int_{1/\sqrt{t}}^{+\infty}\mathrm{e}^{-2c_jw^2}\,w^{-2}\,\mathrm{d}w=-\int_{1/\sqrt{t}}^{+\infty}\mathrm{e}^{-2c_jw^2}\,\mathrm{d}w^{-1}\notag\\
&\lesssim -\left(\mathrm{e}^{-2c_jw^2}\,w^{-1}\right)\big|_{1/\sqrt{t}}^{+\infty}+\int_{1/\sqrt{t}}^{+\infty}w^{-1}\,\mathrm{d}\mathrm{e}^{-2c_jw^2}\notag\\
&\lesssim \sqrt{t}\,\mathrm{e}^{-2c_j/t}-4c_j\int_{1/\sqrt{t}}^{+\infty}\mathrm{e}^{-2c_jw^2}\,\mathrm{d}w\lesssim \sqrt{t},
\end{align}
which says $\ml{I}(t;1+2\sigma)\lesssim t$ by combining the last two estimates.\\

\noindent\underline{Critical Dimensional Case: $n=2+2\sigma$.} By the similar way to the one of $n=1+2\sigma$, the threshold $1/t$ separates the domain of integral such that
\begin{align*}
	\ml{I}(t;2+2\sigma)&\lesssim\sum\limits_{j=1,2}\int_0^{\varepsilon_0}|\sin(\beta_jrt)|^2\,\mathrm{e}^{-2c_jr^2t}\,r^{-1}\,\mathrm{d}r\\
	&\lesssim\sum\limits_{j=1,2}\left(\int_0^{1/t}+\int_{1/t}^{+\infty}\right)|\sin(\beta_jw\sqrt{t})|^2\,\mathrm{e}^{-2c_jw^2}\,w^{-1}\,\mathrm{d}w.
\end{align*}
For the first integral, the inequality $|\sin(\beta_jw\sqrt{t})|\leqslant \beta_jw\sqrt{t}$ implies
\begin{align*}
\int_0^{1/t}|\sin(\beta_jw\sqrt{t})|^2\,\mathrm{e}^{-2c_jw^2}\,w^{-1}\,\mathrm{d}w\lesssim t\int_0^{1/t}\mathrm{e}^{-2c_jw^2}\,w\,\mathrm{d}w\lesssim t^{-1}.
\end{align*}
For the second integral, one may employ an integration by parts further and the boundedness of sine function to deduce
\begin{align}\label{Est-03}
\int_{1/t}^{+\infty}|\sin(\beta_jw\sqrt{t})|^2\,\mathrm{e}^{-2c_jw^2}\,w^{-1}\,\mathrm{d}w&\lesssim\int_{1/t}^{+\infty}\mathrm{e}^{-2c_jw^2}\,w^{-1}\,\mathrm{d}w=\int_{1/t}^{+\infty}\mathrm{e}^{-2c_jw^2}\mathrm{d}\ln w\notag\\
&\lesssim \mathrm{e}^{-2c_j/t^2}\,\ln t+4c_j\int_{0}^{+\infty}\mathrm{e}^{-2c_jw^2}\,w\,|\ln w|\,\mathrm{d}w\notag\\
&\lesssim \ln t+1
\end{align}
because of $\mathrm{e}^{-2c_jw^2}\,w\,|\ln w|\in L^1([0,+\infty))$. We thus conclude
\begin{align*}
\ml{I}(t;2+2\sigma)\lesssim t^{-1}+\ln t+1\lesssim \ln t
\end{align*}
 for large time $t\gg1$.\\

\noindent\underline{Higher Dimensional Case: $n\geqslant 3+2\sigma$.} Thanks to the boundedness of $|\sin(\beta_jrt)|$, one derives
\begin{align*}
\ml{I}(t;n)&\lesssim\sum\limits_{j=1,2}\int_0^{\varepsilon_0}|\sin(\beta_jrt)|^2\,\mathrm{e}^{-2c_jr^2t}\,r^{n-3-2\sigma}\,\mathrm{d}r\\
&\lesssim\sum\limits_{j=1,2}t^{-\frac{n-2-2\sigma}{2}}\int_0^{+\infty}\mathrm{e}^{-2c_jw^2}\,w^{n-3-2\sigma}\,\mathrm{d}w\lesssim t^{1+\sigma-\frac{n}{2}},
\end{align*}
in which we used $n-3-2\sigma\geqslant0$ so that $\mathrm{e}^{-2c_jw^2}\,w^{n-3-2\sigma}\in L^1([0,+\infty))$.
\end{proof}

We are going to justify the optimality of derived estimates in Lemma \ref{Lemma-Upper-Bound} for large time $t\gg1$ via estimating its lower bounds.

\begin{lemma}\label{Lemma-Lower-Bound}
Let $n\geqslant 1+2\sigma$ with any $\sigma\in\mb{N}_0$. Let $\ell_1,\ell_2$ be non-trivial constants, and $\beta_1,\beta_2,c_1,c_2$ be positive constants
such that $c_1\neq c_2$ and
\begin{align}\label{Condition-Lower-Bound}
	\begin{cases}
	\ell_1^2\neq\ell_2^2	&\mbox{if}\ \ n=1+2\sigma,\\[0.5em]
\displaystyle{\frac{\ell_1^2\beta_2^2}{\ell_2^2\beta_1^2}>2}\ \ \mbox{or}\ \ \displaystyle{\frac{\ell_2^2\beta_1^2}{\ell_1^2\beta_2^2}>2}	&\mbox{if}\ \ n=2+2\sigma,\\[0.5em]
\mbox{no restriction}	&\mbox{if}\ \ n\geqslant 3+2\sigma.
	\end{cases}
\end{align}
 Then, the following sharp lower bound estimates:
\begin{align*}
	\ml{I}(t;n)\gtrsim\begin{cases}
		t&\mbox{if}\ \ n=1+2\sigma,\\
		\ln t&\mbox{if}\ \ n=2+2\sigma,\\
		t^{1+\sigma-\frac{n}{2}}&\mbox{if}\ \ n\geqslant 3+2\sigma,
	\end{cases}
\end{align*}
 hold for large time $t\gg1$.
\end{lemma}
\begin{proof}
Our main philosophy is to introduce time-dependent subzones carrying suitable parameters associated with the $\epsilon_0$-dependent triangle inequality
\begin{align}\label{Triangle-Ineq}
|f-g|^2=(1-\epsilon_0)|f|^2-\frac{1-\epsilon_0}{\epsilon_0}\,|g|^2+\epsilon_0\left|f-\frac{1}{\epsilon_0}\,g\right|^2\geqslant (1-\epsilon_0)|f|^2-\frac{1-\epsilon_0}{\epsilon_0}\,|g|^2,
\end{align}
with a constant $\epsilon_0\in(0,1)$ to be fixed in each situation.
 We then may control lower bounds  quantitatively due to the fact that two diffusion waves kernels of the difference in $\ml{I}(t;n)$ have the similar large time qualitative property. We still divide our proof into three cases.\\

	\noindent\underline{Lower Dimensional Case: $n= 1+2\sigma$.}  We discuss the case $\ell_1^2>\ell_2^2$ in the first place. Let us consider large time $t\gg1$ such that 
	\begin{align*}
	[\mu_1/t,(1+\epsilon_1)\mu_1/t]\subset [0,\varepsilon_0]
	\end{align*}
with constants $\mu_1>0$ and $\epsilon_1>0$. By using the triangle inequality \eqref{Triangle-Ineq} and shrinking the domain of integral from $[0,\varepsilon_0]$ to $[\mu_1/t,(1+\epsilon_1)\mu_1/t]$, one arrives at
\begin{align*}
\ml{I}(t;1+2\sigma)&\geqslant \int_{\mu_1/t}^{(1+\epsilon_1)\mu_1/t}\left|\ell_1\frac{\sin(\beta_1rt)}{\beta_1r}\,\mathrm{e}^{-c_1r^2t}-\ell_2\frac{\sin(\beta_2rt)}{\beta_2r}\,\mathrm{e}^{-c_2r^2t}\right|^2\,\mathrm{d}r\\
&\geqslant \frac{(1-\epsilon_0)\ell_1^2}{\beta_1^2}\int_{\mu_1/t}^{(1+\epsilon_1)\mu_1/t}|\sin(\beta_1rt)|^2\,\mathrm{e}^{-2c_1r^2t}\,r^{-2}\,\mathrm{d}r\\
&\quad-\frac{(1-\epsilon_0)\ell_2^2}{\epsilon_0\beta_2^2}\int_{\mu_1/t}^{(1+\epsilon_1)\mu_1/t}|\sin(\beta_2rt)|^2\,\mathrm{e}^{-2c_2r^2t}\,r^{-2}\,\mathrm{d}r\\
&\geqslant \frac{(1-\epsilon_0)\ell_1^2}{\beta_1^2}\,\mathrm{e}^{-2c_1(1+\epsilon_1)^2\mu_1^2/t}\int_{\mu_1/t}^{(1+\epsilon_1)\mu_1/t}|\sin(\beta_1rt)|^2\,r^{-2}\,\mathrm{d}r-\frac{(1-\epsilon_0)\epsilon_1\ell_2^2\mu_1}{\epsilon_0}\,\mathrm{e}^{-2c_2\mu_1^2/t}\,t,
\end{align*}
where we used $|\sin(\beta_2rt)|\leqslant \beta_2rt$ in the last line.
We now choose $\mu_1>0$ to be a small parameter such that $(1+\epsilon_1)\beta_1\mu_1\ll1$ leading to $|\sin(\beta_1rt)|\geqslant \sin(\beta_1\mu_1)$ for any $r\in[\mu_1/t,(1+\epsilon_1)\mu_1/t]$. As a consequence,
\begin{align*}
\ml{I}(t;1+2\sigma)&\geqslant\frac{(1-\epsilon_0)\epsilon_1\ell_1^2}{(1+\epsilon_1)\mu_1\beta_1^2}\,\mathrm{e}^{-2c_1(1+\epsilon_1)^2\mu_1^2/t}\,t\,[\sin(\beta_1\mu_1)]^2-\frac{(1-\epsilon_0)\epsilon_1\ell_2^2\mu_1}{\epsilon_0}\,\mathrm{e}^{-2c_2\mu_1^2/t}\,t\\
&\geqslant(1-\epsilon_0)\epsilon_1\mu_1\left(\frac{\ell_1^2}{1+\epsilon_1}\,\mathrm{e}^{-2c_1(1+\epsilon_1)^2\mu_1^2/t}\left|\frac{\sin(\beta_1\mu_1)}{\beta_1\mu_1}\right|^2-\frac{\ell_2^2}{\epsilon_0}\,\mathrm{e}^{-2c_2\mu_1^2/t}\right)t\\
&\gtrsim \mathrm{e}^{-c/t}\,t,
\end{align*}
provided that
\begin{align*}
\frac{\ell_1^2}{1+\epsilon_1}\,\mathrm{e}^{-2c_1(1+\epsilon_1)^2\mu_1^2/t}\left|\frac{\sin(\beta_1\mu_1)}{\beta_1\mu_1}\right|^2>\frac{\ell_2^2}{\epsilon_0}\,\mathrm{e}^{-2c_2\mu_1^2/t}\ \ \Leftrightarrow\ \ \frac{\epsilon_0\ell_1^2}{(1+\epsilon_1)\ell_2^2}\left|\frac{\sin(\beta_1\mu_1)}{\beta_1\mu_1}\right|^2>\mathrm{e}^{2[c_1(1+\epsilon_1)^2-c_2]\mu_1^2/t}.
\end{align*}
Because of our consideration $\ell_1^2>\ell_2^2$, there exists a positive constant $\epsilon$ such that $\ell_1^2>(1+\epsilon)\ell_2^2$, e.g. $\epsilon=(\ell_1^2-\ell_2^2)/(2\ell_2^2)$. Thanks to $\epsilon_1,\mu_1>0$ and $\epsilon_0\in(0,1)$, there exist some positive constants $\epsilon_0,\epsilon_1,\mu_1$ such that
\begin{align*}
	\epsilon=\frac{\ell_1^2-\ell_2^2}{2\ell_2^2}=\frac{1+\epsilon_1}{\epsilon_0}\left|\frac{\beta_1\mu_1}{\sin(\beta_1\mu_1)}\right|^2-1>0,
\end{align*}
which means
\begin{align*}
 \ell_1^2>\frac{1+\epsilon_1}{\epsilon_0}\left|\frac{\beta_1\mu_1}{\sin(\beta_1\mu_1)}\right|^2\ell_2^2\ \ \Leftrightarrow\ \ \frac{\epsilon_0\ell_1^2}{(1+\epsilon_1)\ell_2^2}\left|\frac{\sin(\beta_1\mu_1)}{\beta_1\mu_1}\right|^2>1.
\end{align*}
 Note that $\epsilon_0,\epsilon_1,\mu_1$ are fixed constants dependent on $\beta_1,\ell_1,\ell_2$. So, one derives $\ml{I}(t;1+2\sigma)\gtrsim t$ for large time $t\gg1$. Due to the symmetric between $(\ell_1,\beta_1,c_1)$ and $(\ell_2,\beta_2,c_2)$ in $\ml{I}(t;n)$,  one may easily derive the desired lower bound estimate for $\ell_1^2<\ell_2^2$ by changing the roles of $f$ and $g$ in \eqref{Triangle-Ineq}.\\

	\noindent\underline{Critical Dimensional Case: $n=2+2\sigma$.} We just discuss the case $\ell_1^2\beta_2^2>2\ell_2^2\beta_1^2$, and the symmetric case $\ell_2^2\beta_1^2>2\ell_1^2\beta_2^2$ can be treated similarly. Let us consider large time $t\gg1$ such that 
	\begin{align*}
	[1/t,\mu_2/\sqrt{t}\, ]\subset[0,\varepsilon_0]
	\end{align*}
	with a suitable parameter $\mu_2>0$ to be determined later. Making use of the triangle inequality \eqref{Triangle-Ineq} and shrinking the domain of integral from $[0,\varepsilon_0]$ to $[1/t,\mu_2/\sqrt{t}\,]$, one may deduce
	\begin{align*}
	\ml{I}(t;2+2\sigma)&\geqslant\int_{1/t}^{\mu_2/\sqrt{t}}\left|\ell_1\frac{\sin(\beta_1rt)}{\beta_1r}\,\mathrm{e}^{-c_1r^2t}-\ell_2\frac{\sin(\beta_2rt)}{\beta_2r}\,\mathrm{e}^{-c_2r^2t}\right|^2\,r\,\mathrm{d}r\\
	&\geqslant \frac{(1-\epsilon_0)\ell_1^2}{\beta_1^2}\int_{1/t}^{\mu_2/\sqrt{t}}|\sin(\beta_1rt)|^2\,\mathrm{e}^{-2c_1r^2t}\,r^{-1}\,\mathrm{d}r\\
	&\quad-\frac{(1-\epsilon_0)\ell_2^2}{\epsilon_0\beta_2^2}\int_{1/t}^{\mu_2/\sqrt{t}}|\sin(\beta_2rt)|^2\,\mathrm{e}^{-2c_2r^2t}\,r^{-1}\,\mathrm{d}r\\
	&\geqslant \frac{(1-\epsilon_0)\ell_1^2}{2\beta_1^2}\int_{1/t}^{\mu_2/\sqrt{t}}\mathrm{e}^{-2c_1r^2t}\,r^{-1}\,\mathrm{d}r-\frac{(1-\epsilon_0)\ell_1^2}{2\beta_1^2}\int_{1/t}^{\mu_2/\sqrt{t}}\cos(2\beta_1rt)\,\mathrm{e}^{-2c_1r^2t}\,r^{-1}\,\mathrm{d}r\\
	&\quad-\frac{(1-\epsilon_0)\ell_2^2}{\epsilon_0\beta_2^2}\int_{1/t}^{\mu_2/\sqrt{t}}\mathrm{e}^{-2c_2r^2t}\,r^{-1}\,\mathrm{d}r\\
	&=:\ml{J}_{2,1}(t)-\ml{J}_{2,2}(t)-\ml{J}_{2,3}(t),
	\end{align*}
	where we used $2|\sin(\beta_1rt)|^2=1-\cos(2\beta_1rt)$ and $|\sin(\beta_2rt)|^2\leqslant 1$, respectively, in the previous line.
First of all, carrying out an integration by parts one gets
\begin{align}\label{Est-11}
-\frac{4\beta_1^3}{(1-\epsilon_0)\ell_1^2}\,\ml{J}_{2,2}(t)&=-\frac{1}{t}\int_{1/t}^{\mu_2/\sqrt{t}}\mathrm{e}^{-2c_1r^2t}\,r^{-1}\,\mathrm{d}\sin(2\beta_1rt)\notag\\
&=-\frac{1}{t}\left(\frac{\sin(2\beta_1rt)}{r}\,\mathrm{e}^{-2c_1r^2t}\right)\Big|_{1/t}^{\mu_2/\sqrt{t}}-\frac{1}{t}\int_{1/t}^{\mu_2/\sqrt{t}}\sin(2\beta_1rt)\,\mathrm{e}^{-2c_1r^2t}\,(r^{-2}+4c_1t)\,\mathrm{d}r\notag\\
&=-\frac{1}{\mu_2\sqrt{t}}\,\sin(2\beta_1\mu_2\sqrt{t})\,\mathrm{e}^{-2c_1\mu_2^2}+\sin(2\beta_1)\,\mathrm{e}^{-2c_1/t}\notag\\
&\quad-\frac{1}{t}\int_{1/t}^{\mu_2/\sqrt{t}}\sin(2\beta_1rt)\,\mathrm{e}^{-2c_1r^2t}\,(r^{-2}+4c_1t)\,\mathrm{d}r.
\end{align}
Concerning large time $t\gg1$, we are able to obtain
\begin{align*}
|\ml{J}_{2,2}(t)|\lesssim t^{-\frac{1}{2}}+1+t^{-1}\int_{1/t}^{\mu_2/\sqrt{t}}r^{-2}\,\mathrm{d}r+\int_{1/t}^{\mu_2/\sqrt{t}}\,\mathrm{d}r\lesssim 1.
\end{align*}
Secondly, as $t\gg 1$ it is clear that
\begin{align*}
\ml{J}_{2,1}(t)-\ml{J}_{2,3}(t)&\geqslant(1-\epsilon_0) \left(\frac{\ell_1^2}{2\beta_1^2}\,\mathrm{e}^{-2c_1\mu_2^2}-\frac{\ell_2^2}{\epsilon_0\beta_2^2}\,\mathrm{e}^{-2c_2/t}\right)\int_{1/t}^{\mu_2/\sqrt{t}}r^{-1}\,\mathrm{d}r\gtrsim\ln t,
\end{align*}
	provided that the following inequality:
	\begin{align}\label{Ineq-critical}
		\frac{\ell_1^2}{2\beta_1^2}\,\mathrm{e}^{-2c_1\mu_2^2}>\frac{\ell_2^2}{\epsilon_0\beta_2^2}\,\mathrm{e}^{-2c_2/t}\ \ \Leftrightarrow\ \ \frac{\epsilon_0\ell_1^2\beta_2^2}{2\ell_2^2\beta_1^2}> \mathrm{e}^{2c_1\mu_2^2-2c_2/t}
	\end{align}
	holds for a small but fixed parameter $\mu_2>0$. Due to $\ell_1^2\beta_2^2>2\ell_2^2\beta_1^2$, there exist $\mu_2>0$ and $\epsilon_0\in(0,1)$ such that
	\begin{align*}
	1<\mathrm{e}^{2c_1\mu_2^2}<\frac{\ell_1^2\beta_2^2}{2\ell_2^2\beta_1^2}\ \ \Rightarrow\ \ \mathrm{e}^{2c_1\mu_2^2}<\frac{\epsilon_0\ell_1^2\beta_2^2}{2\ell_2^2\beta_1^2}<\frac{\ell_1^2\beta_2^2}{2\ell_2^2\beta_1^2},
	\end{align*}
according to the density argument in $\mb{R}$, namely, the condition \eqref{Ineq-critical} is guaranteed. It shows the lower bound estimate $\ml{I}(t;2+2\sigma)\gtrsim \ln t-1\gtrsim \ln t$ for large time $t\gg1$.
	 \\
	 
	\noindent\underline{Higher Dimensional Case: $n\geqslant 3+2\sigma$.} We just discuss the case $c_1<c_2$ here, and the symmetric case $c_2<c_1$ can be done analogously. Let us consider large time $t\gg1$ such that
	\begin{align*}
	[\mu_3/\sqrt{t},(1+\epsilon_2)\mu_3/\sqrt{t}\,]\subset[0,\varepsilon_0]
	\end{align*} 
with suitable parameters $\mu_3>0$ and $\epsilon_2>0$ to be determined later. By using the $\epsilon_0$-dependent triangle inequality \eqref{Triangle-Ineq} again and shrinking the domain of integral from $[0,\varepsilon_0]$ to $[\mu_3/\sqrt{t},(1+\epsilon_2)\mu_3/\sqrt{t}\,]$, one derives
\begin{align*}
\ml{I}(t;n)&\geqslant\int_{\mu_3/\sqrt{t}}^{(1+\epsilon_2)\mu_3/\sqrt{t}}\left|\ell_1\frac{\sin(\beta_1rt)}{\beta_1r}\,\mathrm{e}^{-c_1r^2t}-\ell_2\frac{\sin(\beta_2rt)}{\beta_2r}\,\mathrm{e}^{-c_2r^2t}\right|^2\,r^{n-1-2\sigma}\,\mathrm{d}r,
\end{align*}
furthermore,
\begin{align*}
\ml{I}(t;n)&\geqslant (1-\epsilon_0)\ell_1^2\int_{\mu_3/\sqrt{t}}^{(1+\epsilon_2)\mu_3/\sqrt{t}}\frac{|\sin(\beta_1rt)|^2}{\beta_1^2r^2}\,\mathrm{e}^{-2c_1r^2t}\,r^{n-1-2\sigma}\,\mathrm{d}r\\
&\quad-\frac{(1-\epsilon_0)\ell_2^2}{\epsilon_0}\,\int_{\mu_3/\sqrt{t}}^{(1+\epsilon_2)\mu_3/\sqrt{t}}\frac{|\sin(\beta_2rt)|^2}{\beta_2^2r^2}\,\mathrm{e}^{-2c_2r^2t}\,r^{n-1-2\sigma}\,\mathrm{d}r\\
&\geqslant \frac{(1-\epsilon_0)\ell_1^2}{2\beta_1^2}\int_{\mu_3/\sqrt{t}}^{(1+\epsilon_2)\mu_3/\sqrt{t}}\mathrm{e}^{-2c_1r^2t}\,r^{n-3-2\sigma}\,\mathrm{d}r\\
&\quad-\frac{(1-\epsilon_0)\ell_1^2}{2\beta_1^2}\int_{\mu_3/\sqrt{t}}^{(1+\epsilon_2)\mu_3/\sqrt{t}}\cos(2\beta_1rt)\,\mathrm{e}^{-2c_1r^2t}\,r^{n-3-2\sigma}\,\mathrm{d}r\\
&\quad-\frac{(1-\epsilon_0)\ell_2^2}{\epsilon_0\beta_2^2}\int_{\mu_3/\sqrt{t}}^{(1+\epsilon_2)\mu_3/\sqrt{t}}\mathrm{e}^{-2c_2r^2t}\,r^{n-3-2\sigma}\,\mathrm{d}r\\
&=:\ml{J}_{3,1}(t)-\ml{J}_{3,2}(t)-\ml{J}_{3,3}(t),
\end{align*}
where we applied $2|\sin(\beta_1rt)|^2=1-\cos(2\beta_1rt)$ and $|\sin(\beta_2rt)|\leqslant 1$, respectively, in the last line. Again, by the change of variable $r\sqrt{t}=w$ one arrives at
\begin{align*}
\ml{J}_{3,2}(t)=\frac{(1-\epsilon_0)\ell_1^2}{2\beta_1^2}\,t^{1+\sigma-\frac{n}{2}}\int_{\mu_3}^{(1+\epsilon_2)\mu_3}\cos(2\beta_1w\sqrt{t})\,\mathrm{e}^{-2c_1w^2}\,w^{n-3-2\sigma}\,\mathrm{d}w=o(t^{1+\sigma-\frac{n}{2}})
\end{align*}
for large time $t\gg1$, in which we employed the Riemann-Lebesgue lemma due to
\begin{align*}
\mathrm{e}^{-2c_1w^2}\,w^{n-3-2\sigma}\in L^1([\mu_3,(1+\epsilon_2)\mu_3]) \ \ \mbox{for}\ \ n-3-2\sigma\geqslant 0.
\end{align*}
For another, some direct computations yield
\begin{align*}
\ml{J}_{3,1}(t)-\ml{J}_{3,3}(t)&\geqslant (1-\epsilon_0)\left(\frac{\ell_1^2}{2\beta_1^2}\,\mathrm{e}^{-2c_1(1+\epsilon_2)^2\mu_3^2}-\frac{\ell_2^2}{\epsilon_0\beta_2^2}\,\mathrm{e}^{-2c_2\mu_3^2}\right)\int_{\mu_3/\sqrt{t}}^{(1+\epsilon_2)\mu_3/\sqrt{t}}r^{n-3-2\sigma}\,\mathrm{d}r\gtrsim t^{1+\sigma-\frac{n}{2}}
\end{align*}
provided that the following inequality:
\begin{align}\label{Ineq-higher}
\frac{\ell_1^2}{2\beta_1^2}\,\mathrm{e}^{-2c_1(1+\epsilon_2)^2\mu_3^2}>\frac{\ell_2^2}{\epsilon_0\beta_2^2}\,\mathrm{e}^{-2c_2\mu_3^2}\ \ \Leftrightarrow\ \ \frac{\epsilon_0\ell_1^2\beta_2^2}{2\ell_2^2\beta_1^2}> \mathrm{e}^{2[c_1(1+\epsilon_2)^2-c_2]\mu_3^2}
\end{align}
holds for a parameter $\mu_3>0$. From our consideration $c_1<c_2$, there exists $\epsilon_2>0$ such that
\begin{align*}
\epsilon_2=\sqrt{\frac{c_1+c_2}{2c_1}}-1>0\ \ \Rightarrow\ \ c_1(1+\epsilon_2)^2-c_2=-\frac{c_2-c_1}{2}<0.
\end{align*}
By restricting $\epsilon_0\in(0,1)$ such that
\begin{align*}
\epsilon_0>\frac{2\ell_2^2\beta_1^2}{\ell_1^2\beta_2^2}\,\mathrm{e}^{-(c_2-c_1)\mu_3^2}\ \ \mbox{for}\ \ \mu_3\gg1,
\end{align*}
the condition \eqref{Ineq-higher} can be guaranteed. Finally, it concludes $\ml{I}(t;n)\gtrsim t^{1+\sigma-\frac{n}{2}}+o(t^{1+\sigma-\frac{n}{2}})\gtrsim t^{1+\sigma-\frac{n}{2}}$ for large time $t\gg1$.
\end{proof}
\begin{remark}
It seems challenging to remove all restrictions in \eqref{Condition-Lower-Bound}. But when $n=1+2\sigma$ for some classes of $\{\ell_1,\ell_2,\sigma\}$, we are able to remove the restriction on parameters in Lemma \ref{Lemma-Improve}.
\end{remark}

Summarizing the derived estimates in Lemmas \ref{Lemma-Upper-Bound} and \ref{Lemma-Lower-Bound}, one may immediately conclude the following optimal large time estimates for high dimensions $n\geqslant 1+2\sigma$:
\begin{align*}
	\ml{I}(t;n)\approx\begin{cases}
	t&\mbox{if}\ \ n=1+2\sigma,\\
	\ln t&\mbox{if}\ \ n=2+2\sigma,\\
	t^{1+\sigma-\frac{n}{2}}&\mbox{if}\ \ n\geqslant 3+2\sigma,
\end{cases}
\end{align*}
provided that the conditions \eqref{Condition-Lower-Bound} holds. Especially, $n=2+2\sigma$ is the critical dimension to distinguish optimal growth estimates (when $n\leqslant 2+2\sigma$) and optimal decay estimates (when $n\geqslant 3+2\sigma$). We remark that the last optimal estimates can be generalized to any $\sigma\geqslant0$ easily.

\subsubsection{Optimal large time estimates for low dimensions}
\hspace{5mm}We next turn to the low dimensional case $n<1+2\sigma$, which is more difficult due to the additional singular part $r^{n-1-2\sigma}$ in $r=0$. Recalling in Proposition \ref{Prop-small} that the determined asymptotic profile of $\widehat{u}$ in the $L^2$ norm is related to $\ml{I}(t;n)$ when $\sigma=1$, to end this final case we are interested in behavior of $\ml{I}(t;n)|_{\sigma=1}$ when $n<(1+2\sigma)|_{\sigma=1}$, i.e. $n\in\{1,2\}$. Surprisingly, asymptotic behavior of it heavily depends on the  parameters $\ell_1,\ell_2,\beta_1,\beta_2$.
\begin{lemma}\label{Lemma-Very-Low-D}
Let $n\in\{1,2\}$. Let $\ell_1,\ell_2$ be non-trivial constants, and $\beta_1,\beta_2,c_1,c_2$ be positive constants such that $c_1\neq c_2$.
\begin{itemize}
	\item If $\ell_1=\ell_2$ with $\beta_1=\beta_2$, then the following sharp upper bound estimates:
	\begin{align*}
		\ml{I}(t;n)|_{\sigma=1}^{\ell_1=\ell_2,\,\beta_1=\beta_2}\lesssim (1+t)^{2-\frac{n}{2}}
	\end{align*}
hold for any $t>0$, and the following sharp lower bound estimates:
	\begin{align*}
		\ml{I}(t;n)|_{\sigma=1}^{\ell_1=\ell_2,\,\beta_1=\beta_2}\gtrsim t^{2-\frac{n}{2}}
	\end{align*}
	hold for large time $t\gg1$.
		\item If  $\ell_1=\ell_2$ with $\beta_1\neq\beta_2$, then  the following  sharp upper bound estimates:
	\begin{align}\label{Est-04}
		\ml{I}(t;n)|_{\sigma=1}^{\ell_1=\ell_2,\,\beta_1\neq\beta_2}\lesssim \begin{cases}
			(1+t)^3&\mbox{if}\ \ n=1,\\
			(1+t)^2\ln (\mathrm{e}+t)&\mbox{if}\ \ n=2,
		\end{cases}
	\end{align}
	hold for any $t>0$, and the following  sharp lower bound estimates:
	\begin{align*}
		\ml{I}(t;n)|_{\sigma=1}^{\ell_1=\ell_2,\,\beta_1\neq\beta_2}\gtrsim \begin{cases}
			t^3&\mbox{if}\ \ n=1,\\
			t^2\ln t&\mbox{if}\ \ n=2,
		\end{cases}
	\end{align*}
	hold for large time $t\gg1$.
	\item If  $\ell_1\neq\ell_2$, then it blows up in finite time, precisely,
	\begin{align*}
		\ml{I}(t;n)|_{\sigma=1}^{\ell_1\neq \ell_2}=+\infty
	\end{align*}
	 for $0<t\leqslant t_0$ with any $t_0>0$. Notice that it includes the cases $\beta_1=\beta_2$ and $\beta_1\neq\beta_2$.
\end{itemize}
\end{lemma}
\begin{proof}To understand the delicate relations for parameters and the strength of non-summable singularity $r^{n-3}$ in $r=0$, we provide an appropriate representation of double diffusion waves kernel for $n\in\{1,2\}$.
According to the mean value theorem, there exist positive constants $c_3$ between $c_1$ and $c_2$, moreover, $\beta_3$ between $\beta_1$ and $\beta_2$ such that
\begin{align*}
\mathrm{e}^{-c_1r^2t}-\mathrm{e}^{-c_2r^2t}=(c_2-c_1)\,t\,\mathrm{e}^{-c_3r^2t}\,r^2\ \ \mbox{and}\ \
\sin(\beta_1rt)-\sin(\beta_2rt)=(\beta_1-\beta_2)\,t\cos(\beta_3rt)\,r.
\end{align*}
The double diffusion waves kernels can be rewritten by
\begin{align*}
	&\ell_1\frac{\sin(\beta_1rt)}{\beta_1r}\,\mathrm{e}^{-c_1r^2t}-\ell_2\frac{\sin(\beta_2rt)}{\beta_2r}\,\mathrm{e}^{-c_2r^2t}\\
	&=\frac{\ell_1}{\beta_1}\frac{\sin(\beta_1rt)}{r}\left(\mathrm{e}^{-c_1r^2t}-\mathrm{e}^{-c_2r^2t}\right)+\frac{\ell_1}{\beta_1}\frac{\sin(\beta_1rt)-\sin(\beta_2rt)}{r}\,\mathrm{e}^{-c_2r^2t}+\left(\frac{\ell_1}{\beta_1}-\frac{\ell_2}{\beta_2}\right)\frac{\sin(\beta_2rt)}{r}\,\mathrm{e}^{-c_2r^2t}\\
	&=\frac{\ell_1(c_2-c_1)}{\beta_1}\,t\sin(\beta_1rt)\,\mathrm{e}^{-c_3r^2t}\,r+\frac{\ell_1(\beta_1-\beta_2)}{\beta_1}\,t\cos(\beta_3rt)\,\mathrm{e}^{-c_2r^2t}+\frac{\ell_1\beta_2-\ell_2\beta_1}{\beta_1}\frac{\sin(\beta_2rt)}{\beta_2r}\,\mathrm{e}^{-c_2r^2t}\\
	&=:I_{0,1}(t,r)+I_{0,2}(t,r)+I_{0,3}(t,r).
\end{align*}
 Here, one may observe that the parameters  $\beta_1,\beta_2,\ell_1,\ell_2$ play important roles in the value (vanishing or not) of $I_{0,j}(t,r)$, which needs carefully analysis in the below four cases.\\

\noindent\underline{Case 1: $\beta_1=\beta_2$ with $\ell_1\beta_2=\ell_2\beta_1$.} It is obviously that $\ell_1=\ell_2$ and
\begin{align}
\ml{I}(t;n)|_{\sigma=1}^{\ell_1=\ell_2,\,\beta_1=\beta_2}&=\int_0^{\varepsilon_0}|I_{0,1}(t,r)|^2\,r^{n-3}\,\mathrm{d}r\notag\\
&\lesssim t^2\int_0^{\varepsilon_0}\mathrm{e}^{-2c_3r^2t}\,r^{n-1}\,\mathrm{d}r\lesssim (1+t)^{2-\frac{n}{2}}.\label{Est-01}
\end{align}
due to $I_{0,2}(t,r)=0=I_{0,3}(t,r)$.
For another, considering large time such that $2\mu_4/\sqrt{t}\leqslant \varepsilon_0$ with $\mu_4>0$, one may follow the same philosophy of Lemma \ref{Lemma-Lower-Bound} in higher dimensions to derive
\begin{align*}
\ml{I}(t;n)|_{\sigma=1}^{\ell_1=\ell_2,\,\beta_1=\beta_2}
&\gtrsim t^2\int_{\mu_4/\sqrt{t}}^{2\mu_4/\sqrt{t}}\big(1-\cos(2\beta_1rt)\big)\,\mathrm{e}^{-2c_3r^2t}\,r^{n-1}\,\mathrm{d}r\\
&\gtrsim t^{2-\frac{n}{2}}\int_{\mu_4}^{2\mu_4}\big(1-\cos(2\beta_1w\sqrt{t})\big)\,\mathrm{e}^{-2c_3w^2}\,w^{n-1}\,\mathrm{d}w\\
&\gtrsim t^{2-\frac{n}{2}}+o(t^{2-\frac{n}{2}})\gtrsim t^{2-\frac{n}{2}}
\end{align*}
for large time $t\gg1$, where we employed the Riemann-Lebesgue lemma.\\

\noindent\underline{Case 2: $\beta_1\neq\beta_2$ with $\ell_1\beta_2=\ell_2\beta_1$.} This implies $\ell_1\neq\ell_2$ and $I_{0,3}(t,r)=0$. From the triangle inequality \eqref{Triangle-Ineq} with $\epsilon_0=1/2$, the lower bound can be shown with the aid of a sufficiently small parameter $\varepsilon_1$ satisfying $0<\varepsilon_1\ll \varepsilon_0$. Furthermore, the parameter $\varepsilon_0$ is considered as $\beta_3\varepsilon_0t_0\leqslant\pi/4$ for a given $t_0>0$. One immediately concludes
\begin{align*}
\ml{I}(t;n)|_{\sigma=1}^{\ell_1\neq\ell_2,\,\beta_1\neq\beta_2}&\geqslant-\int_0^{\varepsilon_0}|I_{0,1}(t,r)|^2\,r^{n-3}\,\mathrm{d}r+\frac{1}{2}\int_0^{\varepsilon_0}|I_{0,2}(t,r)|^2\,r^{n-3}\,\mathrm{d}r\\
&\gtrsim -(1+t)^{2-\frac{n}{2}}+t^2\int_{\varepsilon_1}^{\varepsilon_0}|\cos(\beta_3rt)|^2\,\mathrm{e}^{-2c_2r^2t}\,r^{n-3}\,\mathrm{d}r\\
&\gtrsim -(1+t)^{2-\frac{n}{2}}+t^2\,\mathrm{e}^{-2c_2\varepsilon_0^2t}\times\begin{cases}
\varepsilon_1^{-1}-\varepsilon_0^{-1}&\mbox{if}\ \ n=1,\\
\ln\varepsilon_0+\ln\varepsilon_1^{-1}&\mbox{if}\ \ n=2,
\end{cases}
\end{align*}
where we used \eqref{Est-01} and the fact
\begin{align*}
|\cos(\beta_3rt)|^2\geqslant\frac{1}{2}\ \ \mbox{for}\ \ \beta_3rt\leqslant \beta_3\varepsilon_0t_0\leqslant\frac{\pi}{4}.
\end{align*}
Note that the unexpressed multiplicative constant in the previous inequality is independent of $\varepsilon_1$. Letting $\varepsilon_1\downarrow0$, i.e. $\varepsilon_1^{-1}\to+\infty$ and $\ln\varepsilon_1^{-1}\to+\infty$, but $\varepsilon_0$ is fixed, it immediately shows that in finite time at $0<t\leqslant t_0$ with any $t_0>0$ the function blows up, i.e. $\ml{I}(t;n)|_{\sigma=1}^{\ell_1\neq\ell_2,\,\beta_1\neq\beta_2}=+\infty$.\\

\noindent\underline{Case 3: $\beta_1=\beta_2$ with $\ell_1\beta_2\neq\ell_2\beta_1$.} This implies $\ell_1\neq\ell_2$ and $I_{0,2}(t,r)=0$. By the same way as the above, we take the small parameter $\varepsilon_0$ such that $\beta_2\varepsilon_0t_0\leqslant 1$ and arrive at
\begin{align*}
\ml{I}(t;n)|_{\sigma=1}^{\ell_1\neq\ell_2,\, \beta_1=\beta_2}&\geqslant-\int_0^{\varepsilon_0}|I_{0,1}(t,r)|^2\,r^{n-3}\,\mathrm{d}r+\frac{1}{2}\int_0^{\varepsilon_0}|I_{0,3}(t,r)|^2\,r^{n-3}\,\mathrm{d}r\\
&\gtrsim -(1+t)^{2-\frac{n}{2}}+t^2\int_{\varepsilon_1}^{\varepsilon_0}\left|\frac{\sin(\beta_2rt)}{\beta_2rt}\right|^2\,\mathrm{e}^{-2c_2r^2t}\,r^{n-3}\,\mathrm{d}r\\
&\gtrsim -(1+t)^{2-\frac{n}{2}}+t^2\,\mathrm{e}^{-2c_2\varepsilon_0^2t}\times\begin{cases}
	\varepsilon_1^{-1}-\varepsilon_0^{-1}&\mbox{if}\ \ n=1,\\
	\ln\varepsilon_0+\ln\varepsilon_1^{-1}&\mbox{if}\ \ n=2,
\end{cases}
\end{align*}
in which we applied \eqref{Est-01} and the fact
\begin{align*}
	\left|\frac{\sin(\beta_2rt)}{\beta_2rt}\right|^2\geqslant|\sin 1|^2\ \ \mbox{for}\ \ \beta_2rt\leqslant \beta_2\varepsilon_0t_0\leqslant1.
\end{align*}
Letting $\varepsilon_1\downarrow0$ it immediately shows that in finite time at $0<t\leqslant t_0$ with any $t_0>0$ the function blows up, i.e. $\ml{I}(t;n)|^{\ell_1\neq\ell_2,\,\beta_1=\beta_2}_{\sigma=1}=+\infty$. It also can be regarded as the proof of \eqref{Blow-up-01}.\\

\noindent\underline{Case 4: $\beta_1\neq\beta_2$ with $\ell_1\beta_2\neq\ell_2\beta_1$.} The relation of $\ell_1$ and $\ell_2$ is uncertain. Under this situation, the non-trivial functions $I_{0,2}(t,r)$ and $I_{0,3}(t,r)$ play essential roles, where there may exist a cancellation in the subzone $\{0\leqslant rt\ll 1\}$ to compensate the singularity $r^{-2}$ in $r=0$. Without loss of generality, the case $\beta_1>\beta_2$ will be discussed. The next re-arrangement is valid:
\begin{align*}
I_{0,2}(t,r)+I_{0,3}(t,r)&=\frac{\ell_1(\beta_1-\beta_2)}{\beta_1}\,t\,\big(\cos(\beta_3rt)-\cos(\beta_2rt)\big)\,\mathrm{e}^{-c_2r^2t}\\
&\quad+\frac{t}{\beta_1}\,\mathrm{e}^{-c_2r^2t}\left((\ell_1\beta_2-\ell_2\beta_1)\,\frac{\sin(\beta_2rt)}{\beta_2rt}+(\ell_1\beta_1-\ell_1\beta_2)\cos(\beta_2rt)\right),
\end{align*}
whose first component on the right-hand side can be replaced by
\begin{align*}
\frac{\ell_1(\beta_1-\beta_2)(\beta_2-\beta_3)}{\beta_1}\,t^2\sin(\beta_4rt)\,\mathrm{e}^{-c_2r^2t}\,r
\end{align*}
via the mean value theorem with $\beta_2<\beta_4<\beta_3$.

Firstly, for any fixed time $0<t\leqslant t_0$ the maximal size of small frequencies $\varepsilon_0$ can be chosen by $\beta_j\varepsilon_0t_0\ll 1$ for $j\in\{2,3\}$. For any $r\leqslant \varepsilon_0$ the well-known Taylor expansions of sine and cosine functions yield
	\begin{align*}
	\cos(\beta_2rt)=1-\frac{\beta_2^2}{2}\,r^2t^2+O(r^4t^4)\ \ \mbox{and} \ \  
	\frac{\sin(\beta_2rt)}{\beta_2rt}=1-\frac{\beta_2^2}{6}\,r^2t^2+O(r^4t^4),
	\end{align*}
	leading to
	\begin{align}\label{Taylor-subtraction}
		\cos(\beta_2rt)-\frac{\sin(\beta_2rt)}{\beta_2rt}=-\frac{\beta_2^2}{3}\,t^2r^2+O(r^4t^4).
	\end{align} According to the last expansions, the key function in this situation can be rewritten by
	\begin{align*}
	I_{0,2}(t,r)+I_{0,3}(t,r)&=\frac{\ell_1(\beta_1-\beta_2)(\beta_2-\beta_3)}{\beta_1}\,t^2\sin(\beta_4rt)\,\mathrm{e}^{-c_2r^2t}\,r\notag\\
	&\quad+\begin{cases}
		\displaystyle{\frac{t}{\beta_1}\,\mathrm{e}^{-c_2r^2t}\left[(\ell_1-\ell_2)\beta_1+O(r^2t^2)\right]}&\mbox{if}\ \ \ell_1\neq\ell_2,\\[0.7em]
		\displaystyle{\frac{\ell_1(\beta_1-\beta_2)}{\beta_1}\,t\,\mathrm{e}^{-c_2r^2t}\left[-\frac{\beta_2^2}{3}\,t^2r^2+O(r^4t^4)\right]}&\mbox{if}\ \ \ell_1=\ell_2.
	\end{cases}
	\end{align*}
\begin{itemize}
	\item If $\ell_1\neq\ell_2$, then the lower bound is estimated by the triangle inequality \eqref{Triangle-Ineq} as follows:
	\begin{align*}
		\ml{I}(t;n)|_{\sigma=1}^{\ell_1\neq\,\ell_2,\,\beta_1\neq\beta_2}&\gtrsim t^2\int_{0}^{\varepsilon_0}\mathrm{e}^{-2c_2r^2t}\,r^{n-3}\,\mathrm{d}r-t^2\int_0^{\varepsilon_0}O(r^4t^4)\,\mathrm{e}^{-2c_2r^2t}\,r^{n-3}\,\mathrm{d}r\\
		&\quad-t^4\int_0^{\varepsilon_0}|\sin(\beta_4rt)|^2\,\mathrm{e}^{-2c_2r^2t}\,r^{n-1}\,\mathrm{d}r-\int_0^{\varepsilon_0}|I_{0,1}(t,r)|^2\,r^{n-3}\,\mathrm{d}r\\
		&\gtrsim t^2\,\mathrm{e}^{-2c_2\varepsilon_0^2t}\int_{\varepsilon_1}^{\varepsilon_0}r^{n-3}\,\mathrm{d}r-(1+t)^{5-\frac{n}{2}}-(1+t)^{4-\frac{n}{2}}-(1+t)^{2-\frac{n}{2}}
	\end{align*}
	with a sufficiently small constant $0<\varepsilon_1\ll \varepsilon_0$.
	Analogously to the last blow-up proofs, one may claim $\ml{I}(t;n)|_{\sigma=1}^{\ell_1\neq\,\ell_2,\,\beta_1\neq\beta_2}=+\infty$ as $\varepsilon_1\downarrow 0$ in finite time at $0<t\leqslant t_0$ with any $t_0>0$ via the blow-up rates $\varepsilon_1^{-1}$ if $n=1$ and $\ln\varepsilon_1^{-1}$ if $n=2$.
	\item If $\ell_1=\ell_2$, then the upper bound is controlled by
\begin{align*}
\ml{I}(t;n)|_{\sigma=1}^{\ell_1=\ell_2,\,\beta_1\neq\beta_2}&\lesssim t^{2-\frac{n}{2}}+t^4\int_0^{\varepsilon_0}|\sin(\beta_4rt)|^2\,\mathrm{e}^{-2c_2r^2t}\,r^{n-1}\,\mathrm{d}r+t^6\int_0^{\varepsilon_0}\mathrm{e}^{-2c_2r^2t}\,r^{n+1}\,\mathrm{d}r\\
&\quad+t^2\int_0^{\varepsilon_0}\mathrm{e}^{-2c_2r^2t}\,O(r^8t^8)\,r^{n-3}\,\mathrm{d}r\\
&\lesssim t^{2-\frac{n}{2}}+t^{4-\frac{n}{2}}+t^{5-\frac{n}{2}}+t^{7-\frac{n}{2}}<+\infty 
\end{align*}
in finite time $0<t\leqslant t_0$ with any $t_0>0$, since the singularity in $r=0$ has been compensated so that the power of $t$ is strictly positive in our consideration $n\in\{1,2\}$.
\end{itemize}

In the last discussion, we found that when $\ell_1=\ell_2$ the function $\ml{I}(t;n)|_{\sigma=1}$ is locally in time defined even for $n\in\{1,2\}$, where the simple situation $\beta_1=\beta_2$ was studied in Case 1. For this reason, the next discussion concentrates on its large time behavior if $\ell_1=\ell_2$ and $\beta_1\neq \beta_2$.

 Let us study upper bounds for it here. Recall that
\begin{align*}
\int_0^{\varepsilon_0}|I_{0,2}(t,r)+I_{0,3}(t,r)|^2\,r^{n-3}\,\mathrm{d}r
&\lesssim t^2\int_0^{\varepsilon_0}\left|\cos(\beta_3rt)-\frac{\sin(\beta_2rt)}{\beta_2rt}\right|^2\,\mathrm{e}^{-2c_2r^2t}\,r^{n-3}\,\mathrm{d}r\\
&\lesssim t^{3-\frac{n}{2}}\int_0^{+\infty}\left|\cos(\beta_3w\sqrt{t})-\frac{\sin(\beta_2w\sqrt{t})}{\beta_2w\sqrt{t}}\right|^2\,\mathrm{e}^{-2c_2w^2}\,w^{n-3}\,\mathrm{d}w.
\end{align*}
Let us denote 
\begin{align*}
\ml{J}_{4,1}(t)+\ml{J}_{4,2}(t):=t^{3-\frac{n}{2}}\left(\int_0^{\mu_5/\sqrt{t}}+\int_{\mu_5/\sqrt{t}}^{+\infty}\right)\left|\cos(\beta_3w\sqrt{t})-\frac{\sin(\beta_2w\sqrt{t})}{\beta_2w\sqrt{t}}\right|^2\,\mathrm{e}^{-2c_2w^2}\,w^{n-3}\,\mathrm{d}w
\end{align*}
with a small constant $\mu_5>0$ such that $\beta_j\mu_5\ll 1$ with $j\in\{2,3\}$. For the first integral, thanks to $\beta_jw\sqrt{t}\leqslant \beta_j\mu_5\ll 1$, similarly to \eqref{Taylor-subtraction} we obtain
\begin{align*}
\left|\cos(\beta_3w\sqrt{t})-\frac{\sin(\beta_2w\sqrt{t})}{\beta_2w\sqrt{t}}\right|\lesssim w^2t.
\end{align*}
One may obtain the first estimate
\begin{align*}
\ml{J}_{4,1}(t)\lesssim t^{5-\frac{n}{2}}\int_0^{\mu_5/\sqrt{t}}\mathrm{e}^{-2c_2w^2}\,w^{n+1}\,\mathrm{d}w\lesssim (1+t)^{4-n}
\end{align*}
for $n\in\{1,2\}$. Repeating \eqref{Est-02} when $n=1$ and \eqref{Est-03} when $n=2$, the second estimate is given by
\begin{align*}
\ml{J}_{4,2}(t)\lesssim t^{3-\frac{n}{2}}\int_{\mu_5/\sqrt{t}}^{+\infty}\mathrm{e}^{-2c_2w^2}\,w^{n-3}\,\mathrm{d}w\lesssim \begin{cases}
(1+t)^3&\mbox{if}\ \ n=1,\\
(1+t)^2\ln (\mathrm{e}+t)&\mbox{if}\ \ n=2.
\end{cases}
\end{align*}
The combination of the last estimates and \eqref{Est-01} implies
\begin{align*}
\ml{I}(t;n)|_{\sigma=1}^{\ell_1=\ell_2,\,\beta_1\neq\beta_2}\lesssim (1+t)^{2-\frac{n}{2}}+\ml{J}_{4,1}(t)+\ml{J}_{4,2}(t),
\end{align*}
which concludes \eqref{Est-04} immediately.

 To verify its large time optimality via its lower bounds, we in the first place set $\mu_6/t\leqslant \varepsilon_0\ll 1$ for large time $t\gg1$ with a small parameter $\mu_6>0$ such that $\beta_j\mu_6\ll 1$ with $j\in\{2,3\}$. For $\beta_jrt\leqslant \beta_j\mu_6\ll 1$, the Taylor expansions imply
 \begin{align}
 |I_{0,2}(t,r)+I_{0,3}(t,r)|&=\left|\frac{\ell_1(\beta_1-\beta_2)}{\beta_1}\,t\,\mathrm{e}^{-c_2r^2t}\left(\frac{\beta_2^2-3\beta_3^2}{6}\,r^2t^2+O(r^4t^4)\right)\right|\notag\\
 &\gtrsim t^3\,\mathrm{e}^{-c_2r^2t}\,r^2,\label{Est-07}
 \end{align}
where we used the coefficient of $r^2t^2$ being strictly negative due to $\beta_2^2<\beta_3^2$.
Applying the expansion \eqref{Est-07} and the triangle inequality \eqref{Triangle-Ineq} with $\epsilon_0=1/2$, one derives
\begin{align*}
\ml{I}(t;n)|_{\sigma=1}^{\ell_1=\ell_2,\,\beta_1\neq\beta_2}&\geqslant\frac{1}{2}\int_0^{\mu_6/t}|I_{0,2}(t,r)+I_{0,3}(t,r)|^2\,r^{n-3}\,\mathrm{d}r-\int_0^{\varepsilon_0}|I_{0,1}(t,r)|^2\,r^{n-3}\,\mathrm{d}r\\
&\gtrsim t^6\int_0^{\mu_6/t}\mathrm{e}^{-2c_2r^2t}\,r^{n+1}\,\mathrm{d}r-t^{2-\frac{n}{2}}\\
&\gtrsim t^{4-n}\,\mathrm{e}^{-2c_2\mu_6^2/t}-t^{2-\frac{n}{2}}\gtrsim t^{4-n}
\end{align*}
for large time $t\gg1$ and $n\in\{1,2\}$. It addresses the sharpness for $n=1$, but the lower bound for $n=2$ is not sharp due to $t^2\lesssim\ml{I}(t;2)|_{\sigma=1}^{\ell_1=\ell_2,\,\beta_1\neq\beta_2}\lesssim t^2\ln t$ for large time according to the last results. To improve this lower bound estimate, we propose another approach. Let us recall the dominant term of $\ml{I}(t;2)|_{\sigma=1}^{\ell_1=\ell_2,\,\beta_1\neq\beta_2}$ via
\begin{align*}
\ml{J}_5(t)&:=\int_0^{\varepsilon_0}|I_{0,2}(t,r)+I_{0,3}(t,r)|^2\,r^{-1}\,\mathrm{d}r\gtrsim t^2\int_0^{\varepsilon_0\sqrt{t}}\left|\cos(\beta_3w\sqrt{t})-\frac{\sin(\beta_2w\sqrt{t})}{\beta_2w\sqrt{t}}\right|^2\,\mathrm{e}^{-2c_2w^2}\,w^{-1}\,\mathrm{d}w
\end{align*}
due to the ansatz $w=r\sqrt{t}$ again. Let us shrink the domain of integral from $[0,\varepsilon_0\sqrt{t}\,]$ to $[\rho_1/\sqrt{t},\varepsilon_0\sqrt{t}\,]$ with $\rho_1=3\pi/(4\beta_3)>0$, and then apply the triangle inequality \eqref{Triangle-Ineq} to deduce
\begin{align*}
\ml{J}_5(t)&\gtrsim \underbrace{t^2\int_{\rho_1/\sqrt{t}}^{\varepsilon_0\sqrt{t}}|\cos(\beta_3w\sqrt{t})|^2\,\mathrm{e}^{-2c_2w^2}\,w^{-1}\,\mathrm{d}w}_{=:\ml{J}_{5,1}(t)}-\underbrace{t^2\int_{\rho_1/\sqrt{t}}^{\varepsilon_0\sqrt{t}}\left|\frac{\sin(\beta_2w\sqrt{t})}{\beta_2w\sqrt{t}}\right|^2\,\mathrm{e}^{-2c_2w^2}\,w^{-1}\,\mathrm{d}w}_{=:\ml{J}_{5,2}(t)}
\end{align*}
for large time $t\gg1$ such that $t>\rho_1\varepsilon_0^{-1}$.
For the second integral, it is easy to get
\begin{align*}
\ml{J}_{5,2}(t)\lesssim t\int_{\rho_1/\sqrt{t}}^{\varepsilon_0\sqrt{t}}\mathrm{e}^{-2c_2w^2}\,w^{-3}\,\mathrm{d}w\lesssim t\,\mathrm{e}^{-2c_2\rho_1^2/t}\int_{\rho_1/\sqrt{t}}^{\varepsilon_0\sqrt{t}}w^{-3}\,\mathrm{d}w\lesssim t^2.
\end{align*}
For another, motivated by \cite{Ikehata-Ono=2017,Ikehata=2023} we introduce two positive valued sequences
\begin{align}\label{sequences-2}
\rho_k:=\frac{1}{\beta_3}\left(-\frac{\pi}{4}+k\pi\right)\ \ \mbox{and}\ \ \varsigma_k:=\frac{1}{\beta_3}\left(\frac{\pi}{4}+k\pi\right)
\end{align}
for any $k\in\mb{N}_+$ satisfying
\begin{align}\label{sequences}
	\varsigma_k-\rho_k=\rho_{k+1}-\varsigma_k=\frac{\pi}{2\beta_3},
\end{align}
namely, the length of the interval $\Omega_k:=[\rho_k,\varsigma_k]$ is fixed and $\varsigma_k$ is the midpoint of $[\rho_k,\rho_{k+1}]$. Hence, for any $w\in\Omega_k/\sqrt{t}$ it yields the uniform bound $|\cos(\beta_3w\sqrt{t})|\geqslant 1/\sqrt{2}$.
Let us shrink the domain of integral further to arrive at
\begin{align}\label{Est-14}
\ml{J}_{5,1}(t)&\gtrsim t^2\sum\limits_{k=1}^{+\infty}\int_{\{\Omega_k/\sqrt{t}\}\cap[\rho_1/\sqrt{t},\,\varepsilon_0\sqrt{t}\,]}\left(\frac{1}{\sqrt{2}}\right)^2\,\mathrm{e}^{-2c_2w^2}\,w^{-1}\,\mathrm{d}w\gtrsim t^2\int_{\rho_1/\sqrt{t}}^{\varepsilon_0\sqrt{t}}\mathrm{e}^{-2c_2w^2}\,w^{-1}\,\mathrm{d}w,
\end{align} 
where we used the integrand $\mathrm{e}^{-2c_2w^2}\,w^{-1}>0$ being a monotonously decreasing function and the fixed length \eqref{sequences} so that
\begin{align*}
\int_{\rho_k/\sqrt{t}}^{\rho_{k+1}/\sqrt{t}}\mathrm{e}^{-2c_2w^2}\,w^{-1}\,\mathrm{d}w\leqslant 2\int_{\rho_k/\sqrt{t}}^{\varsigma_k/\sqrt{t}}\mathrm{e}^{-2c_2w^2}\,w^{-1}\,\mathrm{d}w.
\end{align*}
In the above statement, to get the final estimate in the interval $[\rho_1/\sqrt{t},\varepsilon_0\sqrt{t}\,]$, we know that there exists a large number $k_0\gg1$ such that $\varepsilon_0\sqrt{t}\in[\rho_{k_0}/\sqrt{t},\rho_{k_0+1}/\sqrt{t}\,]$, consequently,
\begin{align*}
\Big(\bigcup\limits_{k=1}^{k_0-1}[\rho_k/\sqrt{t},\rho_{k+1}/\sqrt{t}\,] \Big)\bigcup\,[\rho_{k_0}/\sqrt{t},\varepsilon_0\sqrt{t}\,]=[\rho_1/\sqrt{t},\varepsilon_0\sqrt{t}\,]\ \ &\mbox{if}\ \ \varepsilon_0\sqrt{t}\in[\rho_{k_0}/\sqrt{t},\varsigma_{k_0}/\sqrt{t}\,],\\
\bigcup\limits_{k=1}^{k_0}[\rho_k/\sqrt{t},\rho_{k+1}/\sqrt{t}\,] \supset[\rho_1/\sqrt{t},\varepsilon_0\sqrt{t}\,]
\ \ &\mbox{if}\ \ \varepsilon_0\sqrt{t}\in[\varsigma_{k_0}/\sqrt{t},\rho_{k_0+1}/\sqrt{t}\,].
\end{align*}
It leads to the large time lower bound
\begin{align*}
\ml{J}_{5,1}(t)\gtrsim t^2\int_{\rho_1/\sqrt{t}}^1\mathrm{e}^{-2c_2w^2}\,w^{-1}\,\mathrm{d}w\gtrsim t^2\ln t.
\end{align*}
Summarizing the derived estimates, by \eqref{Est-01} and the triangle inequality \eqref{Triangle-Ineq}, we are able to gain the sharp estimate
\begin{align*}
\ml{I}(t;2)|_{\sigma=1}^{\ell_1=\ell_2,\,\beta_1\neq\beta_2}&\gtrsim \ml{J}_{5,1}(t)-\ml{J}_{5,2}(t)-\int_0^{\varepsilon_0}|I_{0,1}(t,r)|^2\,r^{-1}\,\mathrm{d}r\\
&\gtrsim t^2\ln t-t^2-t\gtrsim t^2\ln t
\end{align*}
for large time $t\gg1$. Our proof is totally finished.
\end{proof}
\begin{remark}\label{Rem-parameter-blowup}
In Lemma \ref{Lemma-Very-Low-D} we notice two different degrees thresholds  for behavior of $\ml{I}(t;n)|_{\sigma=1}$ if $n\in\{1,2\}$ as follows:
\begin{align*}
\ml{I}(t;n)|_{\sigma=1}\begin{cases}
\ell_1\neq\ell_2:\mbox{blow-up in finite time }\\
\ell_1=\ell_2:\mbox{globally in time defined }\begin{cases}
\beta_1=\beta_2:\mbox{optimal rates }\begin{cases}
t^{\frac{3}{2}}&\mbox{if}\ \ n=1\\
t&\mbox{if}\ \ n=2
\end{cases}\\[1.5em]
\beta_1\neq\beta_2:\mbox{optimal rates }\begin{cases}
	t^{3}&\mbox{if}\ \ n=1\\
	t^2\ln t&\mbox{if}\ \ n=2
\end{cases}
\end{cases}
\end{cases}
\end{align*}
whose optimality of growth rates is guaranteed by the sharp upper and lower bounds for large time $t\gg1$. This parameter-dependent phenomenon  is caused by the possibility of compensation of strong non-summable singularity $r^{-2}$ in $r=0$.
\end{remark}

\subsubsection{A remark of improvement in 3d with $\sigma=1$}
\hspace{5mm}Finally, strongly motivated by the proof of Lemma \ref{Lemma-Very-Low-D}, we can remove the assumption $\ell_1^2\neq\ell_2^2$ in Lemma \ref{Lemma-Lower-Bound} in lower dimensional case $n=(1+2\sigma)|_{\sigma=1}=3$ if $\beta_1\neq\beta_2$.
\begin{lemma}\label{Lemma-Improve}
Let $\ell_1=\ell_2$ be non-trivial constants, and $\beta_1,\beta_2,c_1,c_2$ be positive constants such that $\beta_1\neq\beta_2$ and $c_1\neq c_2$. Then,  the following sharp lower bound estimate:
\begin{align*}
\ml{I}(t;3)|_{\sigma=1}^{\ell_1=\ell_2,\,\beta_1\neq\beta_2}\gtrsim t
\end{align*}
holds for large time $t\gg1$.
\end{lemma}
\begin{proof} Considering $\mu_6/t\leqslant \varepsilon_0\ll 1$ for large time $t\gg1$ with a small parameter $\mu_6>0$ such that $\beta_j\mu_6\ll 1$ with $j\in\{2,3\}$, the lower bound estimate \eqref{Est-07} still holds thanks to $\ell_1=\ell_2$ and $\beta_1\neq\beta_2$ for $\beta_jrt\leqslant \beta_j\mu_6\ll 1$. Therefore,
\begin{align*}
	\ml{I}(t;3)|_{\sigma=1}^{\ell_1=\ell_2,\,\beta_1\neq\beta_2} &\geqslant\frac{1}{2}\int_0^{\mu_6/t}|I_{0,2}(t,r)+I_{0,3}(t,r)|^2\,\mathrm{d}r-\int_0^{\varepsilon_0}|I_{0,1}(t,r)|^2\,\mathrm{d}r\\
	&\gtrsim t^6\int_{0}^{\mu_6/t}\mathrm{e}^{-2c_2r^2t}\,r^4\,\mathrm{d}r-t^{\frac{1}{2}}\\
	&\gtrsim t\,\mathrm{e}^{-2c_2\mu_6^2/t}-t^{\frac{1}{2}}\gtrsim t
\end{align*}
for large time $t\gg1$, which completes our proof directly.
\end{proof}

\subsection{Optimal estimates of solutions with the $L^1$ integrable data}
\subsubsection{Preliminaries: Some estimates for the crucial kernels}
\hspace{5mm}According to the polar coordinates $\xi\in\ml{Z}_{\intt}(\varepsilon_0)\to(\sigma_{\omega},r)\in\mb{S}^{n-1}\times[0,\varepsilon_0]$ we are going to derive some estimates for $\widehat{\ml{G}}_0=\widehat{\ml{G}}_0(t,|\xi|)$, $\widehat{\ml{G}}_1^k=\widehat{\ml{G}}_1^k(t,\xi)$ and $\widehat{\ml{G}}_2=\widehat{\ml{G}}_2(t,|\xi|)$ defined at the beginning of Subsection \ref{Subsection-Pointwise-Estimates} by using optimal estimates from the last lemmas.
\begin{prop}\label{Prop-G0}
	The double diffusion waves kernel (in the cosine form) with  the weak singularity $|\xi|^{-1}$ satisfies the following asymptotic behavior:
	\begin{align*}
	\|\chi_{\intt}(\xi)\widehat{\ml{G}}_0(t,|\xi|)\|_{L^2}=o\big(\ml{D}_n(1+t)\big)
	\end{align*}
for large time $t\gg1$.
\end{prop}
\begin{proof}
We introduce a suitable decomposition of $\widehat{\ml{G}}_0$ by
\begin{align*}
\widehat{\ml{G}}_0&=\left[\big(1-\cos(\nu_2|\xi|t)\big)\,\mathrm{e}^{-c_2|\xi|^2t}-\big(1-\cos(\nu_1|\xi|t)\big)\,\mathrm{e}^{-c_1|\xi|^2t}\right]\frac{1}{|\xi|}+\big(\mathrm{e}^{-c_1|\xi|^2t}-\mathrm{e}^{-c_2|\xi|^2t}\big)\frac{1}{|\xi|}\\
&\leqslant\sum\limits_{k=1,2}\frac{2|\sin(\nu_k|\xi|t/2)|}{|\xi|}\,\mathrm{e}^{-c_k|\xi|^2t}+(c_2-c_1)\,t\,\mathrm{e}^{-c_3|\xi|^2t}\,|\xi|
\end{align*}
in which we applied $1-\cos(\nu_k|\xi|t)=2|\sin(\nu_k|\xi|t/2)|^2\leqslant2|\sin(\nu_k|\xi|t/2)|$ and the mean value theorem.
Thus, from \eqref{Optimal-Ikehata} one derives 
\begin{align*}
\|\chi_{\intt}(\xi)\widehat{\ml{G}}_0(t,|\xi|)\|_{L^2}&\lesssim\sum\limits_{k=1,2}\left\|\chi_{\intt}(\xi)\frac{|\sin(\nu_k|\xi|t/2)|}{|\xi|}\,\mathrm{e}^{-c_k|\xi|^2t}\right\|_{L^2}+t\left\|\chi_{\intt}(\xi)\,\mathrm{e}^{-c_3|\xi|^2t}\,|\xi|\right\|_{L^2}\\
&\lesssim\ml{E}_n(1+t)+t^{\frac{1}{2}-\frac{n}{4}}=o\big(\ml{D}_n(1+t)\big)
\end{align*}
for large time $t\gg1$.
\end{proof}

\begin{prop}\label{Prop-Kernel-Optimal-Type-III}
The double diffusion waves kernel (in the sine form) with the strong singularity $\xi_k|\xi|^{-3}$ satisfies the following sharp upper bound estimates:
\begin{align*}
	\|\chi_{\intt}(\xi)\widehat{\ml{G}}_1^k(t,\xi)\|_{L^2}\lesssim
	\ml{D}_n(1+t)
\end{align*}
for any $t>0$, and the following sharp lower bound estimates:
\begin{align*}
\|\chi_{\intt}(\xi)\widehat{\ml{G}}_1^k(t,\xi)\|_{L^2}\gtrsim\ml{D}_n(1+t)
\end{align*}
for large time $t\gg1$, provided that $\alpha_1<3\alpha_2$ if $n=4$.
\end{prop}
\begin{proof}
The polar coordinates with the $(n-1)$-dimensional measure of the unit sphere $\mb{S}^{n-1}$ show
\begin{align*}
\|\chi_{\intt}(\xi)\widehat{\ml{G}}_1^k(t,\xi)\|_{L^2}^2&=\frac{\gamma^2}{\alpha_2^2}\int_{|\xi|\leqslant\varepsilon_0}\left(\frac{\sin(\nu_1|\xi|t)}{\nu_1|\xi|}\,\mathrm{e}^{-c_1|\xi|^2t}-\frac{\sin(\nu_2|\xi|t)}{\nu_2|\xi|}\,\mathrm{e}^{-c_2|\xi|^2t}\right)^2\frac{\xi_k^2}{|\xi|^4}\,\mathrm{d}\xi\\
&=\frac{\gamma^2}{\alpha_2^2}\int_{\mb{S}^{n-1}}\frac{\omega_k^2}{|\omega|^4}\,\mathrm{d}\sigma_{\omega}\,\int_0^{\varepsilon_0}\left|\frac{\sin(\nu_1rt)}{\nu_1r}\,\mathrm{e}^{-c_1r^2t}-\frac{\sin(\nu_2rt)}{\nu_2r}\,\mathrm{e}^{-c_2r^2t}\right|^2r^{n-3}\,\mathrm{d}r\\
&=\frac{|\mb{S}^{n-1}|\gamma^2}{n\alpha_2^2}\,\ml{I}(t;n)|_{\sigma=1}^{\ell_1=\ell_2=1,\,\beta_j=\nu_j}.
\end{align*}
By using Lemmas \ref{Lemma-Upper-Bound}-\ref{Lemma-Improve} with $\sigma=1$, $\ell_1=\ell_2=1$, $\beta_j=\nu_j$, $\beta_1\neq\beta_2$, we conclude our desired estimates. Particularly, when we applied Lemma \ref{Lemma-Lower-Bound} for $n=4$ to get our lower bound, it needs to take the assumption
\begin{align}\label{Condition-01}
\left\{(\nu_1,\nu_2):\,\nu_2^2>2\nu_1^2\ \ \mbox{or}\ \ \nu_1^2>2\nu_2^2 \right\}\ \ \Rightarrow\ \ \alpha_1<3\alpha_2,
\end{align}
because the set $\alpha_1+3\alpha_2<0$ is empty.
\end{proof}

\begin{prop}\label{Prop-Kernel-Optimal-Type-III-2}
	The double diffusion waves kernel (in the sine form) with the weak singularity $|\xi|^{-1}$ satisfies the following sharp upper bound estimates:
	\begin{align*}
		\|\chi_{\intt}(\xi)\widehat{\ml{G}}_2(t,|\xi|)\|_{L^2}\lesssim \ml{E}_n(1+t)
	\end{align*}
	for any $t>0$, and the following sharp lower bound estimates:
	\begin{align*}
		\|\chi_{\intt}(\xi)\widehat{\ml{G}}_2(t,|\xi|)\|_{L^2}\gtrsim\ml{E}_n(1+t)
	\end{align*}
	for large time $t\gg1$, provided that $(\alpha_0-\alpha_2)^2(\alpha_1-\alpha_2)>2(\alpha_0+\alpha_2)^2(\alpha_1+\alpha_2)$ or $(\alpha_0+\alpha_2)^2(\alpha_1+\alpha_2)>2(\alpha_0-\alpha_2)^2(\alpha_1-\alpha_2)$ if $n=2$.
\end{prop}
\begin{proof}
Analogously, the polar coordinates allow us to express
	\begin{align*}
		\|\chi_{\intt}(\xi)\widehat{\ml{G}}_2(t,|\xi|)\|_{L^2}^2&=\frac{1}{4\alpha_2^2}\int_{|\xi|\leqslant\varepsilon_0}\left((\alpha_0-\alpha_2)\,\frac{\sin(\nu_1|\xi|t)}{\nu_1|\xi|}\,\mathrm{e}^{-c_1|\xi|^2t}-(\alpha_0+\alpha_2)\,\frac{\sin(\nu_2|\xi|t)}{\nu_2|\xi|}\,\mathrm{e}^{-c_2|\xi|^2t}\right)^2\,\mathrm{d}\xi\\
		&=\frac{|\mb{S}^{n-1}|}{4\alpha_2^2}\int_0^{\varepsilon_0}\left|(\alpha_0-\alpha_2)\,\frac{\sin(\nu_1rt)}{\nu_1r}\,\mathrm{e}^{-c_1r^2t}-(\alpha_0+\alpha_2)\,\frac{\sin(\nu_2rt)}{\nu_2r}\,\mathrm{e}^{-c_2r^2t}\right|^2r^{n-1}\,\mathrm{d}r\\
		&=\frac{|\mb{S}^{n-1}|}{4\alpha_2^2}\,\ml{I}(t;n)|_{\sigma=0}^{\ell_{1,2}=\alpha_0\mp\alpha_2,\,\beta_j=\nu_j}.
	\end{align*}
By using Lemmas \ref{Lemma-Upper-Bound}-\ref{Lemma-Lower-Bound} with $\sigma=0$, $\ell_1=\alpha_0-\alpha_2$, $\ell_2=\alpha_0+\alpha_2$, $\beta_j=\nu_j$, we conclude our desired estimates. Particularly, the assumption $\ell_1^2\neq\ell_2^2$ if $n=1$ always holds due to $\alpha_2>0$, and to get our lower bound for $n=2$, it needs to assume
\begin{align*}
\left\{(\nu_1,\nu_2):\,(\alpha_0-\alpha_2)^2\nu_2^2>2(\alpha_0+\alpha_2)^2\nu_1^2\ \ \mbox{or}\ \ (\alpha_0+\alpha_2)^2\nu_1^2>2(\alpha_0-\alpha_2)^2\nu_2^2 \right\},
\end{align*}
which is equivalent to our setting.
\end{proof}
\subsubsection{Upper bound estimates for the solutions}
\hspace{5mm}We first employ the Plancherel theorem for the small frequency part associated with Proposition \ref{Prop-small} to derive
\begin{align*}
\|\chi_{\intt}(D)u(t,\cdot)\|_{(L^2)^n}&\lesssim \left\|\chi_{\intt}(\xi)\left[\left(1+|\xi|t+\frac{|\sin(\widetilde{c}|\xi|t)|}{|\xi|}\right)\mathrm{e}^{-c|\xi|^2t}+|\widehat{\ml{G}}_0(t,|\xi|)|+|\widehat{\ml{G}}_1(t,\xi)|\right]\right\|_{L^2}\\
&\ \quad\times \left(\|(\widehat{u}_0,\widehat{u}_1)\|_{(L^{\infty})^n\times (L^{\infty})^n}+\|(\widehat{\theta}_0,\widehat{\theta}_1)\|_{L^{\infty}\times L^{\infty}}\right)\\
&\lesssim \left((1+t)^{-\frac{n}{4}}+(1+t)^{\frac{1}{2}-\frac{n}{4}}+\ml{E}_n(1+t)+\ml{D}_n(1+t)\right)\\
&\ \quad\times\left(\|(u_0,u_1)\|_{(L^{1})^n\times (L^{1})^n}+\|(\theta_0,\theta_1)\|_{L^{1}\times L^{1}}\right)\\
&\lesssim \ml{D}_n(1+t)\left(\|(u_0,u_1)\|_{(L^{1})^n\times (L^{1})^n}+\|(\theta_0,\theta_1)\|_{L^{1}\times L^{1}}\right),
\end{align*}
where we applied Propositions \ref{Prop-G0}-\ref{Prop-Kernel-Optimal-Type-III} and the Hausdorff-Young inequality $\|\widehat{f}\|_{L^{\infty}}\leqslant \|f\|_{L^1}$. Similarly, from Propositions \ref{Prop-small} and \ref{Prop-Kernel-Optimal-Type-III-2}, it is clear that
\begin{align*}
\|\chi_{\intt}(D)\theta(t,\cdot)\|_{L^2}&\lesssim\left\|\chi_{\intt}(\xi)\left(|\widehat{\ml{G}}_2(t,|\xi|)|+\mathrm{e}^{-c|\xi|^2t}\right)\right\|_{L^2}\left(\|(u_0,u_1)\|_{(L^{1})^n\times (L^{1})^n}+\|(\theta_0,\theta_1)\|_{L^{1}\times L^{1}}\right)\\
&\lesssim \ml{E}_n(1+t)\left(\|(u_0,u_1)\|_{(L^{1})^n\times (L^{1})^n}+\|(\theta_0,\theta_1)\|_{L^{1}\times L^{1}}\right).
\end{align*}
 Moreover, Propositions \ref{Prop-large} and \ref{Prop-bdd} lead to  
 \begin{align*}
 \left\|\big(\chi_{\bdd}(D)+\chi_{\extt}(D)\big)u(t,\cdot)\right\|_{(L^2)^n}&\lesssim \mathrm{e}^{-ct}\left(\|(u_0,u_1)\|_{(L^2)^n\times (L^2)^n}+\|(\theta_0,\theta_1)\|_{L^2\times L^2}\right),\\
 \left\|\big(\chi_{\bdd}(D)+\chi_{\extt}(D)\big)\theta(t,\cdot)\right\|_{L^2}&\lesssim \mathrm{e}^{-ct}\left(\|(u_0,u_1)\|_{(H^1)^n\times (L^2)^n}+\|(\theta_0,\theta_1)\|_{L^2\times L^2}\right).
 \end{align*}
They demonstrate the upper bound estimates in Theorems \ref{Thm-III-U} and \ref{Thm-III-theta}.

\subsubsection{Lower bound estimates for the solutions}
\hspace{5mm}By the same way as the last part, from Proposition \ref{Prop-small} we derive
\begin{align*}
\left\|\chi_{\intt}(\xi)\big(\widehat{u}(t,\xi)-\widehat{\ml{G}}_1(t,\xi)\widehat{\theta}_1(\xi)\big)\right\|_{(L^2)^n}= o\big(\ml{D}_n(1+t)\big)
\end{align*}
for large time $t\gg1$. Via the mean value theorem so that
\begin{align*}
|\ml{G}_1(t,x-y)-\ml{G}_1(t,x)|\lesssim|y|\,|\nabla \ml{G}_1(t,x-\eta_1y)|\ \ \mbox{with}\ \ \eta_1\in(0,1),
\end{align*}
one may separate the integral into two parts with the threshold $|y|=t^{1/8}$ as follows:
\begin{align*}
&\|\ml{G}_1(t,D)\theta_1(\cdot)-\ml{G}_1(t,\cdot)P_{\theta_1}\|_{(L^2)^n}\\
&\lesssim\left\|\int_{|y|\leqslant t^{1/8}}\big|\ml{G}_1(t,\cdot-y)-\ml{G}_1(t,\cdot)\big| \,|\theta_1(y)|\,\mathrm{d}y\,\right\|_{L^2}+\left\|\int_{|y|\geqslant t^{1/8}}\big(|\ml{G}_1(t,\cdot-y)|+|\ml{G}_1(t,\cdot)|\big)|\theta_1(y)|\,\mathrm{d}y\,\right\|_{L^2}\\
&\lesssim t^{\frac{1}{8}}\|\,|\xi\,\widehat{\ml{G}}_1(t,\xi)|\,\|_{L^2}\|\theta_1\|_{L^1}+\|\,|\widehat{\ml{G}}_1(t,\xi)|\,\|_{L^2}\|\theta_1\|_{L^1(|x|\geqslant t^{1/8})}\\
&\lesssim t^{\frac{1}{8}}\ml{E}_n(1+t)\|\theta_1\|_{L^1}+
o\big(\ml{D}_n(1+t)\big)
\end{align*}
for large time $t\gg1$, where we used Lemma \ref{Lemma-Upper-Bound} with $\sigma=0$, Proposition \ref{Prop-Kernel-Optimal-Type-III}, as well as
\begin{align*}
\left\|\big(1-\chi_{\intt}(\xi)\big)|\xi\,\widehat{\ml{G}}_1(t,\xi)|\right\|_{L^2}+\left\|\big(1-\chi_{\intt}(\xi)\big)|\widehat{\ml{G}}_1(t,\xi)|\right\|_{L^2}\lesssim\mathrm{e}^{-ct}.
\end{align*}
The standard triangle inequality gives
\begin{align*}
\|u(t,\cdot)-\varphi(t,\cdot)\|_{(L^2)^n}&=\|\widehat{u}(t,\xi)-\widehat{\ml{G}}_1(t,\xi)P_{\theta_1}\|_{(L^2)^n}\\
&\leqslant
	\left\|\widehat{u}(t,\xi)-\widehat{\ml{G}}_1(t,\xi)\widehat{\theta}_1(\xi)\right\|_{(L^2)^n}+\left\|\widehat{\ml{G}}_1(t,\xi)\widehat{\theta}_1(\xi)-\widehat{\ml{G}}_1(t,\xi)P_{\theta_1}\right\|_{(L^2)^n}\\
	&=o\big(\ml{D}_n(1+t)\big)
\end{align*}
for large time $t\gg1$, and analogously,
\begin{align*}
\|\theta(t,\cdot)-\psi(t,\cdot)\|_{L^2}=o\big(\ml{E}_n(1+t)\big).
\end{align*}

Eventually, employing the triangle inequality of inverse side, one easily derives
\begin{align*}
\|u(t,\cdot)\|_{(L^2)^n}&\gtrsim \|\widehat{\ml{G}}_1(t,\xi)\|_{(L^2)^n}|P_{\theta_1}|-\|u(t,\cdot)-\varphi(t,\cdot)\|_{(L^2)^n}\\
&\gtrsim	\ml{D}_n(1+t)|P_{\theta_1}|-o\big(\ml{D}_n(1+t)\big)\gtrsim \ml{D}_n(1+t)|P_{\theta_1}|
\end{align*}
moreover
\begin{align*}
\|\theta(t,\cdot)\|_{L^2}&\gtrsim \|\widehat{\ml{G}}_2(t,|\xi|)\|_{L^2}|P_{\theta_1}|-\|\theta(t,\cdot)-\psi(t,\cdot)\|_{L^2}\\
&\gtrsim \ml{E}_n(1+t)|P_{\theta_1}|-o\big(\ml{E}_n(1+t)\big)\gtrsim \ml{E}_n(1+t)|P_{\theta_1}|
\end{align*}
for large time $t\gg1$, provided that $P_{\theta_1}\neq0$. All in all, the desired estimates in Theorems \ref{Thm-III-U} and \ref{Thm-III-theta} are totally completed.

\section{Large time behavior for the thermoelastic system of type II}\label{Section_type II}\setcounter{equation}{0}
\subsection{Pretreatment by reductions}
\hspace{5mm}In this section we consider $\delta=0$ in the thermoelastic system \eqref{Thermoelastic-Our-Problem} whose solutions are denoted by $\widetilde{u},\widetilde{\theta}$ for clarity to avoid repetitions of symbol. Following the similar reductions as those in the type III model \eqref{TypeIII-v},  the (scalar or vector) unknown $\widetilde{v}=\widetilde{v}(t,x)$ such that $\widetilde{v}=\widetilde{u},\widetilde{\theta}$ satisfies
\begin{align}\label{Hyperbolic-Eq}
	\begin{cases}
		\ml{L}_{\mathrm{Type-II}}(\partial_t,\Delta)\widetilde{v}=0,&x\in\mb{R}^n,\ t>0,\\
		\partial_t^j\widetilde{v}(0,x)=\widetilde{v}_j(x)\ \ \mbox{for}\ \ j\in\{0,1,2,3\},&x\in\mb{R}^n,
	\end{cases}
\end{align}
in which the strictly hyperbolic differential operator is denoted by
\begin{align*}
	\ml{L}_{\mathrm{Type-II}}(\partial_t,\Delta):=\partial_t^4-(b^2+\kappa+\gamma^2)\Delta\partial_t^2+b^2\kappa\Delta^2
\end{align*}
and the additional initial data $\widetilde{v}_{2,3}(x)$ are calculated (or taken $\delta=0$ in those of the type III model in the last section) by
\begin{align*}
	\widetilde{u}_2(x)&:=b^2\Delta \widetilde{u}_0(x)-\gamma\nabla\widetilde{\theta}_0(x),\ \ \ \ \ \quad \qquad \qquad 
	\widetilde{u}_3(x):=b^2\Delta \widetilde{u}_1(x)-\gamma\nabla \widetilde{\theta}_1(x),\\
	\widetilde{\theta}_2(x)&:=-b^2\gamma\Delta\divv \widetilde{u}_0(x)+(\kappa+\gamma^2)\Delta \widetilde{\theta}_0(x),\ \
	\widetilde{\theta}_3(x):=-b^2\gamma\Delta\divv \widetilde{u}_1(x)+(\kappa+\gamma^2)\Delta\widetilde{\theta}_1(x).
\end{align*}

Analogously, the solutions $\widetilde{u}$ and $\widetilde{\theta}$ satisfy the same fourth order hyperbolic differential equation but their qualitative properties are different due to the regularities of initial data $\widetilde{u}_{2,3}(x)$ and $\widetilde{\theta}_{2,3}(x)$. Comparing with $\ml{L}_{\mathrm{Type-III}}(\partial_t,\Delta)$, the strictly hyperbolic operator $\ml{L}_{\mathrm{Type-II}}(\partial_t,\Delta)$ owns two pairs of conjugate characteristic roots, which are pairwise distinct and pure imaginary. Thanks to the strictly hyperbolic theory, the well-posedness was well-established by assuming suitably regularities for $\widetilde{v}_{2,3}(x)$, for example, in \cite[Chapter 3.4]{Ebert-Reissig=2018} it claims
\begin{align*}
\widetilde{v}\in \ml{C}([0,+\infty),H^3)\cap \ml{C}^1([0,+\infty),H^2)\cap \ml{C}^2([0,+\infty),H^1)\cap \ml{C}^3([0,+\infty),L^2)
\end{align*}
by taking $\widetilde{v}_j\in H^{3-j}$ for $j\in\{0,1,2,3\}$, i.e. higher regularities for the original initial data. Thus, our interest is large time behavior of solutions with lower regular data.

\subsection{Pointwise estimates in the Fourier space}\label{Subsection-Pointwise}
\hspace{5mm}Following the same approach as the type III model, we may directly apply the Fourier transform and deduce the corresponding Cauchy problem for $\widehat{\widetilde{v}}=\widehat{\widetilde{v}}(t,\xi)$. Then, its characteristic roots to the quartic \eqref{quartic} carrying $\delta=0$ are explicitly given by
\begin{align*}
	\widetilde{\lambda}_{1,2}=\pm i\nu_2|\xi|\ \ \mbox{and}\ \ \widetilde{\lambda}_{3,4}=\pm i\nu_1|\xi|,
\end{align*}
which are pairwise distinct and conjugate. Hence, it allows us to apply the representation \eqref{Rep-two-pairs} with lengthy but straightforward computations
\begin{align*}
	(\nu_2^2-\nu_1^2)|\xi|^2\,\widehat{\widetilde{v}}&=\cos(\nu_1|\xi|t)\big(\nu_2^2|\xi|^2\,\widehat{\widetilde{v}}_0+\widehat{\widetilde{v}}_2\big)+\frac{\sin(\nu_1|\xi|t)}{\nu_1|\xi|}\,\big(\nu_2^2|\xi|^2\,\widehat{\widetilde{v}}_1+\widehat{\widetilde{v}}_3\big)\\
	&\quad-\cos(\nu_2|\xi|t)\big(\nu_1^2|\xi|^2\,\widehat{\widetilde{v}}_0+\widehat{\widetilde{v}}_2\big)-\frac{\sin(\nu_2|\xi|t)}{\nu_2|\xi|}\,\big(\nu_1^2|\xi|^2\,\widehat{\widetilde{v}}_1+\widehat{\widetilde{v}}_3\big).
\end{align*}
According to the expressions of initial data in the Fourier space, the solutions can be rewritten by
\begin{align*}
	(\nu_2^2-\nu_1^2)\widehat{\widetilde{u}}&=\big((\nu_2^2-b^2)\cos(\nu_1|\xi|t)-(\nu_1^2-b^2)\cos(\nu_2|\xi|t)\big)\widehat{\widetilde{u}}_0\\
	&\quad+\left((\nu_2^2-b^2)\,\frac{\sin(\nu_1|\xi|t)}{\nu_1|\xi|}-(\nu_1^2-b^2)\,\frac{\sin(\nu_2|\xi|t)}{\nu_2|\xi|}\right)\widehat{\widetilde{u}}_1\\
	&\quad-\frac{i\gamma\xi}{|\xi|^2}\big(\cos(\nu_1|\xi|t)-\cos(\nu_2|\xi|t)\big)\widehat{\widetilde{\theta}}_0-\frac{i\gamma\xi}{|\xi|^2}\left(\frac{\sin(\nu_1|\xi|t)}{\nu_1|\xi|}-\frac{\sin(\nu_2|\xi|t)}{\nu_2|\xi|}\right)\widehat{\widetilde{\theta}}_1
\end{align*}
and
\begin{align*}
	(\nu_2^2-\nu_1^2)\widehat{\widetilde{\theta}}&=ib^2\gamma\big(\cos(\nu_1|\xi|t)-\cos(\nu_2|\xi|t)\big)\xi\cdot\widehat{\widetilde{u}}_0+ib^2\gamma\left(\frac{\sin(\nu_1|\xi|t)}{\nu_1|\xi|}-\frac{\sin(\nu_2|\xi|t)}{\nu_2|\xi|}\right)\xi\cdot\widehat{\widetilde{u}}_1\\
	&\quad+\big((\nu_2^2-\kappa-\gamma^2)\cos(\nu_1|\xi|t)-(\nu_1^2-\kappa-\gamma^2)\cos(\nu_2|\xi|t)\big)\widehat{\widetilde{\theta}}_0\\
	&\quad+\left((\nu_2^2-\kappa-\gamma^2)\,\frac{\sin(\nu_1|\xi|t)}{\nu_1|\xi|}-(\nu_1^2-\kappa-\gamma^2)\,\frac{\sin(\nu_2|\xi|t)}{\nu_2|\xi|}\right)\widehat{\widetilde{\theta}}_1.
\end{align*}

Let us now introduce the crucial kernels
\begin{align*}
	\widehat{\ml{G}}_3&:=-\frac{\gamma}{\nu_2^2-\nu_1^2}\left(\frac{\sin(\nu_1|\xi|t)}{\nu_1|\xi|}-\frac{\sin(\nu_2|\xi|t)}{\nu_2|\xi|}\right)\frac{i\xi}{|\xi|^2}\\ &\ =\frac{\gamma}{\alpha_2}\left(\frac{\sin(\nu_1|\xi|t)}{\nu_1|\xi|}-\frac{\sin(\nu_2|\xi|t)}{\nu_2|\xi|}\right)\frac{i\xi}{|\xi|^2},
\end{align*}	
moreover,
\begin{align*}
	\widehat{\ml{G}}_4&:=\frac{\nu_2^2-\kappa-\gamma^2}{\nu_2^2-\nu_1^2}\,\frac{\sin(\nu_1|\xi|t)}{\nu_1|\xi|}-\frac{\nu_1^2-\kappa-\gamma^2}{\nu_2^2-\nu_1^2}\,\frac{\sin(\nu_2|\xi|t)}{\nu_2|\xi|}\\
	&\ =-\frac{1}{2\alpha_2}\left((\alpha_0-\alpha_2)\,\frac{\sin(\nu_1|\xi|t)}{\nu_1|\xi|}-(\alpha_0+\alpha_2)\,\frac{\sin(\nu_2|\xi|t)}{\nu_2|\xi|}\right)
\end{align*}
according to the representations of parameters in \eqref{parameters}.
Note that $\widehat{\ml{G}}_3\in\mb{R}^n$ and $\widehat{\ml{G}}_4\in\mb{R}$. They are the double waves kernels consisting of two waves with different propagation speeds $\nu_1\neq \nu_2$. It seems that $\widehat{\ml{G}}_3$ has a stronger singularity $|\xi|^{-2}$ than the one with $|\xi|^{-1}$ for $\widehat{\ml{G}}_4$ if $|\xi|\ll 1$.
\begin{prop}\label{Prop-type-II-Error}
The solutions $\widehat{\widetilde{u}},\widehat{\widetilde{\theta}}$ satisfy the following pointwise estimates in the Fourier space:
\begin{align*}
	|\widehat{\widetilde{u}}-\widehat{\ml{G}}_3\widehat{\widetilde{\theta}}_1|\lesssim |\widehat{\widetilde{u}}_0|+\frac{|\sin(\tilde{c}|\xi|t)|}{|\xi|}\big(|\widehat{\widetilde{u}}_1|+|\widehat{\widetilde{\theta}}_0|\big)
\end{align*}
with a positive constant $\tilde{c}$, and
\begin{align*}
	|\widehat{\widetilde{\theta}}-\widehat{\ml{G}}_4\widehat{\widetilde{\theta}}_1|\lesssim |\xi|\,|\widehat{\widetilde{u}}_0|+|\widehat{\widetilde{u}}_1|+|\widehat{\widetilde{\theta}}_0|
\end{align*}
for any $\xi\in\mb{R}^n$.
\end{prop}
\begin{proof}
By using $1-\cos(\nu_j|\xi|t)|\leqslant 2|\sin(\nu_j|\xi|t/2)|$ we are able to get
\begin{align*}
	|\widehat{\widetilde{u}}-\widehat{\ml{G}}_3\widehat{\widetilde{\theta}}_1|\lesssim|\widehat{\widetilde{u}}_0|+\frac{|\sin(\tilde{c}|\xi|t)|}{|\xi|}\,|\widehat{\widetilde{u}}_1|+\frac{1}{|\xi|}\big(|\sin(\nu_1|\xi|t/2)|+|\sin(\nu_2|\xi|t/2)|\big)|\widehat{\widetilde{\theta}}_0|,
\end{align*}
which leads to our desired estimate with $\tilde{c}>0$ depending on $\nu_1,\nu_2$.
The error estimate for $\widehat{\widetilde{\theta}}$ is a direct consequence of subtraction and the boundedness of sine, cosine functions.
\end{proof}

\subsection{Sharp upper bound estimates in the $L^2$ norm}
\hspace{5mm}Before stating the sharp upper bound estimates for $\widetilde{u}$ and $\widetilde{\theta}$, we first deduce the next proposition for the crucial (dominant) kernel $\widehat{\ml{G}}_3=\widehat{\ml{G}}_3(t,\xi)$ acting on the initial data.
\begin{prop}\label{Prop-G3}
Let us assume $\widetilde{\theta}_1\in L^2\cap L^1$ for $1\leqslant n\leqslant 4$. Then, the following sharp upper bound estimates:
\begin{align*}
\|\ml{G}_3(t,D)\widetilde{\theta}_1(\cdot)\|_{(L^2)^n}\lesssim\ml{D}_n(1+t)\|\widetilde{\theta}_1\|_{L^2\cap L^1}
\end{align*}
hold for any $t>0$.
\end{prop}
\begin{proof} With the same setting on $\mu_6$ in the proof of Lemma \ref{Lemma-Very-Low-D}, let us take a small parameter $\mu_6>0$ such that $\nu_j\mu_6\ll 1$ with $j\in\{2,3\}$. The targeted integral is separated by the threshold $|\xi|=\mu_6/t$ into two parts
\begin{align*}
\|\ml{G}_3(t,D)\widetilde{\theta}_1(\cdot)\|_{(L^2)^n}^2&\lesssim\left(\int_{|\xi|\leqslant\mu_6/t}+\int_{|\xi|\geqslant\mu_6/t}\right)\left|\frac{\sin(\nu_1|\xi|t)}{\nu_1|\xi|}-\frac{\sin(\nu_2|\xi|t)}{\nu_2|\xi|}\right|^2|\xi|^{-2}\,|\widehat{\widetilde{\theta}}_1(\xi)|^2\,\mathrm{d}\xi\\
&=:\ml{J}_{6,1}(t;n)+\ml{J}_{6,2}(t;n).
\end{align*}

Recalling $\nu_1>\nu_2$, for $\nu_j|\xi|t\leqslant\nu_j\mu_6\ll 1$ the Taylor expansion of sine function says
\begin{align}\label{Est-13}
\left|\frac{\sin(\nu_1|\xi|t)}{\nu_1|\xi|t}-\frac{\sin(\nu_2|\xi|t)}{\nu_2|\xi|t}\right|=\frac{\nu_1^2-\nu_2^2}{6}\,|\xi|^2t^2+o(|\xi|^3t^3)\approx |\xi|^2t^2,
\end{align}
which leads to the estimate for any $t>0$ as follows:
\begin{align*}
\ml{J}_{6,1}(t;n)\lesssim t^6\int_{|\xi|\leqslant \mu_6/t}|\xi|^2\,|\widehat{\widetilde{\theta}}_1(\xi)|^2\,\mathrm{d}\xi\lesssim t^6\int_0^{\mu_6/t}r^{n+1}\,\mathrm{d}r\,\|\widehat{\widetilde{\theta}}_1\|_{L^{\infty}}^2\lesssim t^{4-n}\|\widetilde{\theta}_1\|_{L^1}^2.
\end{align*}
Notice that $\ml{J}_{6,1}(t;n)$ is bounded for small time due to $1\leqslant n\leqslant 4$.

For another, strongly motivated by the free wave equation studied in \cite{Ikehata=2023}, we are going to propose some refined zones in the plane $\{|\xi|\geqslant \mu_6/t\}$ in different dimensions for large time $t\gg1$. The basic philosophy is to apply the $L^1$ integrability of initial data and the polar coordinates for ``small frequencies''; to employ the $L^2$ regularity of initial data for ``large frequencies''.
\begin{itemize}
	\item For $n=1$ one uses $|\sin(\nu_j|\xi|t)|\leqslant 1$ to derive
\begin{align*}
	\ml{J}_{6,2}(t;1)&\lesssim \sum\limits_{j=1,2}\left(\int_{\mu_6/t\leqslant|\xi|\leqslant \mu_6/\sqrt{t}}+\int_{|\xi|\geqslant \mu_6/\sqrt{t}}\right)|\sin(\nu_j|\xi|t)|^2\,|\xi|^{-4}\,|\widehat{\widetilde{\theta}}_1(\xi)|^2\,\mathrm{d}\xi\\
	&\lesssim\int_{\mu_6/t}^{\mu_6/\sqrt{t}}r^{-4}\,\mathrm{d}r\,\|\widetilde{\theta}_1\|_{L^1}^2+t^2\int_{|\xi|\geqslant \mu_6/\sqrt{t}}|\widehat{\widetilde{\theta}}_1(\xi)|^2\,\mathrm{d}\xi\\
	&\lesssim (t^3-t^{\frac{3}{2}})\|\widetilde{\theta}_1\|_{L^1}^2+t^2\|\widetilde{\theta}_1\|_{L^2}^2\lesssim t^3\|\widetilde{\theta}_1\|_{L^2\cap L^1}^2.
\end{align*}
\item For $n=2$ one uses $|\sin(\nu_j|\xi|t)|\leqslant \nu_j|\xi|t$ to derive
\begin{align*}
\ml{J}_{6,2}(t;2)&\lesssim t^2\sum\limits_{j=1,2}\left(\int_{\mu_6/t\leqslant|\xi|\leqslant \mu_6/\sqrt{\ln t}}+\int_{|\xi|\geqslant \mu_6/\sqrt{\ln t}}\right)\left|\frac{\sin(\nu_j|\xi|t)}{\nu_j|\xi|t}\right|^2|\xi|^{-2}\,|\widehat{\widetilde{\theta}}_1(\xi)|^2\,\mathrm{d}\xi\\
&\lesssim t^2\int_{\mu_6/t}^{\mu_6/\sqrt{\ln t}}r^{-1}\,\mathrm{d}r\,\|\widetilde{\theta}_1\|_{L^1}^2+t^2\int_{|\xi|\geqslant \mu_6/\sqrt{\ln t}}|\xi|^{-2}\,|\widehat{\widetilde{\theta}}_1(\xi)|^2\,\mathrm{d}\xi\\
&\lesssim t^2(\ln t-\ln\ln t)\|\widetilde{\theta}_1\|_{L^1}^2+t^2\ln t\,\|\widetilde{\theta}_1\|_{L^2}^2\lesssim t^2\ln t\,\|\widetilde{\theta}_1\|_{L^2\cap L^1}^2.
\end{align*}
\item For $n=3$ one uses $|\sin(\nu_j|\xi|t)|\leqslant 1$ to derive
\begin{align*}
	\ml{J}_{6,2}(t;3)&\lesssim \sum\limits_{j=1,2}\left(\int_{\mu_6/t\leqslant|\xi|\leqslant \mu_6/\sqrt[4]{t}}+\int_{|\xi|\geqslant \mu_6/\sqrt[4]{t}}\right)|\sin(\nu_j|\xi|t)|^2\,|\xi|^{-4}\,|\widehat{\widetilde{\theta}}_1(\xi)|^2\,\mathrm{d}\xi\\
	&\lesssim\int_{\mu_6/t}^{\mu_6/\sqrt[4]{t}}r^{-2}\,\mathrm{d}r\,\|\widetilde{\theta}_1\|_{L^1}^2+t\int_{|\xi|\geqslant \mu_6/\sqrt[4]{t}}|\widehat{\widetilde{\theta}}_1(\xi)|^2\,\mathrm{d}\xi\\
	&\lesssim (t-t^{\frac{1}{4}})\|\widetilde{\theta}_1\|_{L^1}^2+t\,\|\widetilde{\theta}_1\|_{L^2}^2\lesssim t\,\|\widetilde{\theta}_1\|_{L^2\cap L^1}^2.
\end{align*}
\item For $n=4$ one uses $|\sin(\nu_j|\xi|t)|\leqslant 1$ to derive
\begin{align*}
\ml{J}_{6,2}(t;4)&\lesssim \sum\limits_{j=1,2}\left(\int_{\mu_6/t\leqslant|\xi|\leqslant \mu_6/\sqrt[4]{\ln t}}+\int_{|\xi|\geqslant \mu_6/\sqrt[4]{\ln t}}\right)|\sin(\nu_j|\xi|t)|^2\,|\xi|^{-4}\,|\widehat{\widetilde{\theta}}_1(\xi)|^2\,\mathrm{d}\xi\\
&\lesssim \int_{\mu_6/t}^{\mu_6/\sqrt[4]{\ln t}}r^{-1}\,\mathrm{d}r\,\|\widetilde{\theta}_1\|_{L^1}^2+\int_{|\xi|\geqslant\mu_6/\sqrt[4]{\ln t}}|\xi|^{-4}\,|\widehat{\widetilde{\theta}}_1(\xi)|^2\,\mathrm{d}\xi\\
	&\lesssim (\ln t-\ln\ln t)\|\widetilde{\theta}_1\|_{L^1}^2+\ln t\,\|\widetilde{\theta}_1\|_{L^2}^2\lesssim \ln t\,\|\widetilde{\theta}_1\|_{L^2\cap L^1}^2.
\end{align*}
\end{itemize}
Concerning large time $t\gg1$, the derived estimates in the above claim
\begin{align*}
\ml{J}_{6,2}(t;n)\lesssim[\ml{D}_n(1+t)]^2\|\widetilde{\theta}_1\|_{L^2\cap L^1}^2
\end{align*}
for $1\leqslant n\leqslant 4$. Concerning small time, it is trivial from $|\sin(\nu_j|\xi|t)|\leqslant \nu_j|\xi|t$ that
\begin{align*}
\ml{J}_{6,2}(t;n)\lesssim t^2\int_{|\xi|\geqslant \mu_6/t}|\xi|^{-2}\,|\widehat{\widetilde{\theta}}_1(\xi)|^2\,\mathrm{d}\xi\lesssim \|\widetilde{\theta}_1\|_{L^2}^2.
\end{align*}
From the definition of $\ml{D}_n(1+t)$, our proof is complete.
\end{proof}

Because of the fact that
\begin{align*}
	|\widehat{\ml{G}}_4(t,|\xi|)\widehat{\widetilde{\theta}}_1(\xi)|\lesssim\frac{|\sin(\tilde{c}|\xi|t)|}{\tilde{c}|\xi|}\,|\widehat{\widetilde{\theta}}_1(\xi)|,
\end{align*}
by following the proof of last result or using \cite[Section 3 with $w_0(\xi)\equiv0$, i.e. Lemmas 3.1 and 3.2 with $u_0(x)\equiv0$]{Ikehata=2023}, we are able to conclude the next proposition immediately.
\begin{prop}\label{Prop-G4}
	Let us assume $\widetilde{\theta}_1\in L^2\cap L^1$ for $1\leqslant n\leqslant 2$. Then, the following sharp upper bound estimates:
	\begin{align*}
		\|\ml{G}_4(t,|D|)\widetilde{\theta}_1(\cdot)\|_{L^2}\lesssim\ml{E}_n(1+t)\|\widetilde{\theta}_1\|_{L^2\cap L^1}
	\end{align*}
hold for any $t>0$.
\end{prop}

\begin{remark}\label{Rem-High-Dimensions}
Let us assume $\widetilde{\theta}_1\in L^2\cap L^1$ as before. For high dimensions $n\geqslant 5$, due to $\int_0^1r^{n-5}\,\mathrm{d}r\lesssim 1$ we have the bounded estimates
\begin{align*}
\|\ml{G}_3(t,D)\widetilde{\theta}_1(\cdot)\|_{(L^2)^n}^2&\lesssim\int_{\mb{R}^n}|\sin(\tilde{c}|\xi|t)|^2\,|\xi|^{-4}\,|\widehat{\widetilde{\theta}}_1(\xi)|^2\,\mathrm{d}\xi\\
&\lesssim \int_{|\xi|\leqslant 1}|\xi|^{-4}\,\mathrm{d}\xi\,\|\widetilde{\theta}_1\|_{L^1}^2+\int_{|\xi|\geqslant 1}|\widehat{\widetilde{\theta}}_1(\xi)|^2\,\mathrm{d}\xi\lesssim \|\widetilde{\theta}_1\|_{L^2\cap L^1}^2.
\end{align*}
Similarly, it holds
\begin{align*}
\|\ml{G}_4(t,|D|)\widetilde{\theta}_1(\cdot)\|_{L^2}^2\lesssim \|\widetilde{\theta}_1\|_{L^2\cap L^1}^2
\end{align*}
for $n\geqslant 3$. It means only the bounded estimates of crucial kernels in high dimensions.
\end{remark}

Analogously, as a corollary of last results with other initial data, it is clear that
\begin{align*}
\left\|\frac{\sin(\tilde{c}|\xi|t)}{\tilde{c}|\xi|}\,\big(|\widehat{\widetilde{u}}_1(\xi)|+|\widehat{\widetilde{\theta}}_0(\xi)|\big)\right\|_{L^2}\lesssim\ml{E}_n(1+t)\big(\|\widetilde{u}_1\|_{(L^2\cap L^1)^n}+\|\widetilde{\theta}_0\|_{L^2\cap L^1}\big)
\end{align*}
for $n\in\{1,2\}$. For another,  when $n\geqslant 3$ by Remark \ref{Rem-High-Dimensions}, we obtain the bounded estimates
\begin{align*}
	\left\|\frac{\sin(\tilde{c}|\xi|t)}{\tilde{c}|\xi|}\,\big(|\widehat{\widetilde{u}}_1(\xi)|+|\widehat{\widetilde{\theta}}_0(\xi)|\big)\right\|_{L^2}\lesssim\|\widetilde{u}_1\|_{(L^2\cap L^1)^n}+\|\widetilde{\theta}_0\|_{L^2\cap L^1}.
\end{align*}
All in all, from Proposition \ref{Prop-type-II-Error}, one may conclude the following two estimates:
\begin{align}\label{Est-12}
&\|\widetilde{u}(t,\cdot)-\ml{G}_3(t,D)\widetilde{\theta}_1(\cdot)\|_{(L^2)^n}\notag\\
&\qquad\lesssim\big(\|(\widetilde{u}_0,\widetilde{u}_1)\|_{(L^2)^n\times (L^2\cap L^1)^n}+\|\widetilde{\theta}_0\|_{L^2\cap L^1}\big)\times \begin{cases}
		\ml{E}_n(1+t)&\mbox{if}\ \ n\leqslant 2,\\
	1&\mbox{if}\ \ n\geqslant 3,
\end{cases}
\end{align}
moreover, for any $n\geqslant 1$
\begin{align}\label{Est-10}
\|\widetilde{\theta}(t,\cdot)-\ml{G}_4(t,|D|)\widetilde{\theta}_1(\cdot)\|_{L^2}\lesssim\|(\widetilde{u}_0,\widetilde{u}_1)\|_{(H^1)^n\times (L^2)^n}+\|\widetilde{\theta}_0\|_{L^2}.
\end{align}
Combining to the last two estimates and Propositions \ref{Prop-G3}-\ref{Prop-G4} associated with the usual triangle inequality, the upper bound estimates in Theorems \ref{Thm-III-U} and \ref{Thm-III-theta} are demonstrated.

\subsection{Sharp lower bound estimates in the $L^2$ norm}

\subsubsection{Lower bounds for the elastic displacement in low dimensions}
\hspace{5mm}As our preparation, we recall the well-established estimate from \cite[Lemma 3.1]{Ikehata=2004} that
\begin{align}\label{Est-08}
	|\widehat{f}-P_f|\lesssim |\xi|\,\|f\|_{L^{1,1}}.
\end{align}
We next will estimate the lower bound for $\widetilde{u}$  in different dimensions $n\in\{1,2,3,4\}$.\\

\noindent\underline{The 1d and 3d Cases.} By shrinking the domain into $\Omega_t^0:=\{|\xi|\leqslant \mu_6/t\}$ and employing the triangle inequality, the lower bound is shown by
\begin{align*}
	\|\widetilde{u}(t,\cdot)\|_{(L^2)^n}&\geqslant\|\widehat{\widetilde{u}}(t,\xi)\|_{(L^2(\Omega_t^0 ))^n}\geqslant \|\widehat{\ml{G}}_3(t,\xi)\|_{(L^2(\Omega_t^0 ))^n}|P_{\widetilde{\theta}_1}|-\left\|\widehat{\widetilde{u}}(t,\xi)-\widehat{\ml{G}}_3(t,\xi)P_{\widetilde{\theta}_1}\right\|_{(L^2(\Omega_t^0 ))^n}
\end{align*}
carrying a small parameter $\mu_6>0$ such that $\nu_j\mu_6\ll 1$ with $j\in\{2,3\}$. According to \eqref{Est-08} associated with $|\sin(\tilde{c}rt)|\leqslant \tilde{c}rt$, one estimates
\begin{align*}
	\big\|\widehat{\ml{G}}_3(t,\xi)\big(\widehat{\widetilde{\theta}}_1(\xi)-P_{\widetilde{\theta}_1}\big)\big\|_{(L^2(\Omega_t^0 ))^n}^2
	&\lesssim\|\widehat{\ml{G}}_3(t,\xi)|\xi|\,\|_{(L^2(\Omega_t^0 ))^n}^2\|\widetilde{\theta}_1\|_{L^{1,1}}^2\lesssim\|\sin(\tilde{c}|\xi|t)|\xi|^{-1}\|_{(L^2(\Omega_t^0 ))^n}^2\|\widetilde{\theta}_1\|_{L^{1,1}}^2\\
	&\lesssim\int_0^{\mu_6/t}|\sin(\tilde{c}rt)|^2\,r^{n-3}\,\mathrm{d}r\,\|\widetilde{\theta}_1\|_{L^{1,1}}^2\lesssim t^2\int_0^{\mu_6/t} r^{n-1}\,\mathrm{d}r\,\|\widetilde{\theta}_1\|_{L^{1,1}}^2\\
	&\lesssim t^{2-n}\|\widetilde{\theta}_1\|_{L^{1,1}}^2.
\end{align*}
Combining the last inequality as well as \eqref{Est-12}, the next error estimates are valid:
\begin{align}\label{Est-16}
	&\left\|\widehat{\widetilde{u}}(t,\xi)-\widehat{\ml{G}}_3(t,\xi)P_{\widetilde{\theta}_1}\right\|_{(L^2(\Omega_t^0 ))^n}\notag\\
	&\qquad\leqslant \left\|\widehat{\widetilde{u}}(t,\xi)-\widehat{\ml{G}}_3(t,\xi)\widehat{\widetilde{\theta}}_1(\xi)\right\|_{(L^2)^n}+\big\|\widehat{\ml{G}}_3(t,\xi)\big(\widehat{\widetilde{\theta}}_1(\xi)-P_{\widetilde{\theta}_1}\big)\big\|_{(L^2(\Omega_t^0 ))^n}\notag\\
	&\qquad\lesssim\big(\|(\widetilde{u}_0,\widetilde{u}_1)\|_{(L^2)^n\times (L^2\cap L^1)^n}+\|(\widetilde{\theta}_0,\widetilde{\theta}_1)\|_{(L^2\cap L^1)\times  L^{1,1}}\big)\times \begin{cases}
		\ml{E}_1(1+t)&\mbox{if}\ \ n=1,\\
		1&\mbox{if}\ \ n=3,
	\end{cases}
\end{align}
as $t\gg1$. For another, thanks to the Taylor expansion \eqref{Est-13} for $\nu_j|\xi|t\leqslant\nu_j\mu_6\ll 1$, it yields
\begin{align*}
	\|\widehat{\ml{G}}_3(t,\xi)\|_{(L^2(\Omega_t^0 ))^n}^2&\gtrsim t^2\int_{\Omega_t^0}\left|\frac{\sin(\nu_1|\xi|t)}{\nu_1|\xi|t}-\frac{\sin(\nu_2|\xi|t)}{\nu_2|\xi|t}\right|^2\frac{1}{|\xi|^2}\,\mathrm{d}\xi\gtrsim t^6\int_{\Omega_t^0}|\xi|^{2}\,\mathrm{d}\xi\gtrsim t^{4-n},
\end{align*}
which is sharp for $n\in\{1,3\}$. Consequently, we arrive at the optimal lower bound estimates $\|\widetilde{u}(t,\cdot)\|_{(L^2)^n}\gtrsim \ml{D}_n(1+t)|P_{\widetilde{\theta}_1}|$ if $n\in\{1,3\}$ for large time $t\gg1$ provided that $P_{\widetilde{\theta}_1}\neq0$.\\

\noindent\underline{The 2d Case.} Again, we first know the lower bound
\begin{align*}
	\|\widetilde{u}(t,\cdot)\|_{(L^2)^2}&\geqslant\|\widehat{\widetilde{u}}(t,\xi)\,\mathrm{e}^{-|\xi|^2}\|_{(L^2(|\xi|\geqslant \rho_1/t ))^2}\\
	&\geqslant \|\widehat{\ml{G}}_3(t,\xi)\,\mathrm{e}^{-|\xi|^2}\|_{(L^2(|\xi|\geqslant \rho_1/t ))^2}|P_{\widetilde{\theta}_1}|-\left\|\big(\widehat{\widetilde{u}}(t,\xi)-\widehat{\ml{G}}_3(t,\xi)P_{\widetilde{\theta}_1}\big)\,\mathrm{e}^{-|\xi|^2}\right\|_{(L^2(|\xi|\geqslant \rho_1/t ))^2},
\end{align*}
where we used $1\geqslant \mathrm{e}^{-|\xi|^2}$ and shrank the domain into $\{|\xi|\geqslant \rho_1/t\}$ with $\rho_1>0$ defined in \eqref{sequences-2}. Similarly to the last proof, one deduces
\begin{align}
	&\left\|\big(\widehat{\widetilde{u}}(t,\xi)-\widehat{\ml{G}}_3(t,\xi)P_{\widetilde{\theta}_1}\big)\,\mathrm{e}^{-|\xi|^2}\right\|_{(L^2(|\xi|\geqslant \rho_1/t ))^2}\notag\\
	&\qquad\lesssim \ml{E}_2(1+t)\big(\|(\widetilde{u}_0,\widetilde{u}_1)\|_{(L^2)^2\times (L^2\cap L^1)^2}+\|(\widetilde{\theta}_0,\widetilde{\theta}_1)\|_{(L^2\cap L^1)\times  L^{1,1}}\big)\label{Est-17}
\end{align}
as $t\gg1$, due to \eqref{Est-12} and the calculation that
\begin{align*}
	\left\|\widehat{\ml{G}}_3(t,\xi)\big(\widehat{\widetilde{\theta}}_1(\xi)-P_{\widetilde{\theta}_1}\big)\,\mathrm{e}^{-|\xi|^2}\right\|_{(L^2(|\xi|\geqslant\rho_1/t))^2}^2&\lesssim\left\|\sin(\tilde{c}|\xi|t)\,\mathrm{e}^{-|\xi|^2}\,|\xi|^{-1}\right\|_{(L^2(|\xi|\geqslant \rho_1/t))^2}^2\|\widetilde{\theta}_1 \|_{L^{1,1}}^2\\
	&\lesssim\int_{\rho_1/t}^{+\infty}\mathrm{e}^{-2r^2}\,r^{-1}\,\mathrm{d}r\,\|\widetilde{\theta}_1 \|_{L^{1,1}}^2\lesssim \ln t\,\|\widetilde{\theta}_1 \|_{L^{1,1}}^2
\end{align*}
from the estimates \eqref{Est-03} and \eqref{Est-08}. Before estimating the crucial kernel, let us state
\begin{align*}
	\frac{\sin(\nu_1|\xi|t)}{\nu_1|\xi|}-\frac{\sin(\nu_2|\xi|t)}{\nu_2|\xi|}=\frac{\nu_1-\nu_2}{\nu_1}\left(\cos(\nu_3|\xi|t)-\frac{\sin(\nu_2|\xi|t)}{\nu_2|\xi|t}\right)t
\end{align*}
by the mean value theorem with $\nu_3\in(\nu_2,\nu_1)$.
Therefore, by repeating the deduction of $\ml{J}_{5,1}(t)$ from \eqref{Est-14} where $\Omega_k$ is renewed via \eqref{sequences-2} with $\beta_3=\nu_3$, we may conclude
\begin{align*}
	\|\widehat{\ml{G}}_3(t,\xi)\,\mathrm{e}^{-|\xi|^2}\|_{(L^2(|\xi|\geqslant \rho_1/t ))^2}^2&\gtrsim t^2\int_{|\xi|\geqslant \rho_1/t}\Big(|\cos(\nu_3|\xi|t)|^2-\Big|\frac{\sin(\nu_2|\xi|t)}{\nu_2|\xi|t}\Big|^2\,\Big)\,\mathrm{e}^{-2|\xi|^2}|\xi|^{-2}\,\mathrm{d}\xi\\
	&\gtrsim t^2\int_{\rho_1/t}^{+\infty}|\cos(\nu_3rt)|^2\,\mathrm{e}^{-2r^2}\,r^{-1}\,\mathrm{d}r-\int_{\rho_1/t}^{+\infty}\mathrm{e}^{-2r^2}\,r^{-3}\,\mathrm{d}r\\
	&\gtrsim t^2\sum\limits_{k=1}^{+\infty}\int_{\Omega_k/t}\left(\frac{1}{\sqrt{2}}\right)^2\,\mathrm{e}^{-2r^2}\,r^{-1}\,\mathrm{d}r-\mathrm{e}^{-2\rho_1^2/t^2}\int_{\rho_1/t}^{+\infty}\,r^{-3}\,\mathrm{d}r\\
	&\gtrsim t^2\int_{\rho_1/t}^{1}\mathrm{e}^{-2r^2}\,r^{-1}\,\mathrm{d}r-\mathrm{e}^{-2\rho_1^2/t^2}\,t^2\gtrsim t^2\ln t=[\ml{D}_2(1+t)]^2
\end{align*}
for large time $t\gg1$. Summarizing the obtained estimates in the above, due to the condition $P_{\widetilde{\theta}_1}\neq0$ our proof of sharp lower bound estimate for $n=2$ is completed.\\

\noindent\underline{The 4d Case.} We carry out the analogous steps as the 2d case, namely, we now shrink the domain into $\Omega_t^1:=\{1/t\leqslant |\xi|\leqslant 1/\sqrt{t} \,\}$ for large time and use $1\geqslant \mathrm{e}^{-|\xi|^2}$. Precisely,
\begin{align*}
	\|\widetilde{u}(t,\cdot)\|_{(L^2)^4}&\geqslant\|\widehat{\widetilde{u}}(t,\xi)\,\mathrm{e}^{-|\xi|^2}\|_{(L^2(\Omega_t^1 ))^4}\\
	&\geqslant \|\widehat{\ml{G}}_3(t,\xi)\,\mathrm{e}^{-|\xi|^2}\|_{(L^2(\Omega_t^1 ))^4}|P_{\widetilde{\theta}_1}|-\left\|\big(\widehat{\widetilde{u}}(t,\xi)-\widehat{\ml{G}}_3(t,\xi)P_{\widetilde{\theta}_1}\big)\,\mathrm{e}^{-|\xi|^2}\right\|_{(L^2(\Omega_t^1 ))^4}
\end{align*}
and then
\begin{align}\label{Est-18}
	&\left\|\big(\widehat{\widetilde{u}}(t,\xi)-\widehat{\ml{G}}_3(t,\xi)P_{\widetilde{\theta}_1}\big)\,\mathrm{e}^{-|\xi|^2}\right\|_{(L^2(\Omega_t^1 ))^4}\lesssim \|(\widetilde{u}_0,\widetilde{u}_1)\|_{(L^2)^4\times (L^2\cap L^1)^4}+\|(\widetilde{\theta}_0,\widetilde{\theta}_1)\|_{(L^2\cap L^1)\times  L^{1,1}}
\end{align}
as $t\gg1$, due to
\begin{align*}
	\left\|\widehat{\ml{G}}_3(t,\xi)\big(\widehat{\widetilde{\theta}}_1(\xi)-P_{\widetilde{\theta}_1}\big)\,\mathrm{e}^{-|\xi|^2}\right\|_{(L^2(\Omega_t^1 ))^4}^2&\lesssim\left\|\sin(\tilde{c}|\xi|t)\,\mathrm{e}^{-|\xi|^2}\,|\xi|^{-1}\right\|_{(L^2(\Omega_t^1 ))^4}^2\|\widetilde{\theta}_1 \|_{L^{1,1}}^2\\
	&\lesssim\int_{1/t}^{1/\sqrt{t}}\mathrm{e}^{-2r^2}\,r\,\mathrm{d}r\,\|\widetilde{\theta}_1 \|_{L^{1,1}}^2\lesssim \|\widetilde{\theta}_1 \|_{L^{1,1}}^2.
\end{align*}
Next, with the aid of the triangle inequality \eqref{Triangle-Ineq}, one investigates
\begin{align*}
&\|\widehat{\ml{G}}_3(t,\xi)\,\mathrm{e}^{-|\xi|^2}\|_{(L^2(\Omega_t^1 ))^4}^2\geqslant C\int_{1/t}^{1/\sqrt{t}}\left|\frac{\sin(\nu_1rt)}{\nu_1r}-\frac{\sin(\nu_2rt)}{\nu_2r}\right|^2\,\mathrm{e}^{-2r^2}\,r\,\mathrm{d}r\ \ \mbox{with}\ \ C>0\\
&\qquad\geqslant \frac{C(1-\epsilon_0)}{\nu_1^2}\int_{1/t}^{1/\sqrt{t}}|\sin(\nu_1rt)|^2\,\mathrm{e}^{-2r^2}\,r^{-1}\,\mathrm{d}r-\frac{C(1-\epsilon_0)}{\epsilon_0\nu_2^2}\int_{1/t}^{1/\sqrt{t}}|\sin(\nu_2rt)|^2\,\mathrm{e}^{-2r^2}\,r^{-1}\,\mathrm{d}r\\
&\qquad\geqslant C(1-\epsilon_0)\left(\frac{1}{2\nu_1^2}-\frac{1}{\epsilon_0\nu_2^2}\right)\int_{1/t}^{1/\sqrt{t}}\mathrm{e}^{-2r^2}\,r^{-1}\,\mathrm{d}r-\frac{C(1-\epsilon_0)}{2\nu_1^2}\int_{1/t}^{1/\sqrt{t}}\cos(2\nu_1rt)\,\mathrm{e}^{-2r^2}\,r^{-1}\,\mathrm{d}r\\
&\qquad\gtrsim \mathrm{e}^{-2/t}\,\ln t-1\gtrsim \ln t=[\ml{D}_4(1+t)]^2
\end{align*}
for large time $t\gg1$ if $2\nu_1^2<\nu_2^2$, in which we chose 
\begin{align*}
\epsilon_0=\frac{\nu_2^2+2\nu_1^2}{2\nu_2^2}\in(0,1)\ \ \mbox{such that}\ \ \frac{1}{2\nu_1^2}-\frac{1}{\epsilon_0\nu_2^2}=\frac{\nu_2^2-2\nu_1^2}{2\nu_1^2(\nu_2^2+2\nu_1^2)}>0,
\end{align*}
as well as employed
\begin{align}\label{Est-15}
\int_{1/t}^{1/\sqrt{t}}\cos(2\nu_1rt)\,\mathrm{e}^{-2r^2}\,r^{-1}\,\mathrm{d}r\lesssim t^{-\frac{1}{2}}\,\mathrm{e}^{-2/t}+\mathrm{e}^{-2/t^2}+t^{-1}\int_{1/t}^{1/\sqrt{t}}\,\mathrm{d}r+t^{-1}\int_{1/t}^{1/\sqrt{t}}r^{-2}\,\mathrm{d}r\lesssim 1
\end{align}
thanks to \eqref{Est-11}. Symmetrically, this large time lower bound also holds if $2\nu_2^2<\nu_1^2$. In the view of \eqref{Condition-01}, under the restriction $\alpha_1<3\alpha_2$, our proof for the sharp lower bound estimate for $n=4$ is completed.

\subsubsection{Lower bounds for the temperature difference in low dimensions}
\hspace{5mm}Let us turn to the lower bound for $\widetilde{\theta}$ when $n\in\{1,2\}$. Our philosophy is the same as the last proofs, so we just sketch the main difference. For the sake of convenient, we denote in this part
\begin{align*}
\ell_1=\frac{\nu_2^2-\kappa-\gamma^2}{\nu_2^2-\nu_1^2} \ \ \mbox{and}\ \ \ell_2=\frac{\nu_1^2-\kappa-\gamma^2}{\nu_2^2-\nu_1^2}
\end{align*}
which are non-trivial and $\ell_1\neq\ell_2$ surely.\\

\noindent\underline{The 1d Case.} By shrinking the domain into $\Omega_t^0$ with a small parameter $\mu_6>0$ such that $\nu_j\mu_6\ll 1$ for $j\in\{2,3\}$, one notices that
\begin{align*}
\|\widetilde{\theta}(t,\cdot)\|_{L^2}^2&\geqslant \|\widehat{\ml{G}}_4(t,|\xi|)\|_{L^2(\Omega_t^0)}^2|P_{\widetilde{\theta}_1}|^2-\|\widehat{\widetilde{\theta}}(t,\xi)-\widehat{\ml{G}}_4(t,|\xi|)\widehat{\widetilde{\theta}}_1(\xi)\|_{L^2}^2-\left\|\widehat{\ml{G}}_4(t,|\xi|)\big(\widehat{\widetilde{\theta}}_1(\xi)-P_{\widetilde{\theta}_1}\big)\right\|_{L^2(\Omega_t^0)}^2\\
&\gtrsim t^2\int_{0}^{\mu_6/t}\left|\ell_1\frac{\sin(\nu_1rt)}{\nu_1rt}-\ell_2\frac{\sin(\nu_2rt)}{\nu_2rt}\right|^2\,\mathrm{d}r\,|P_{\widetilde{\theta}_1}|^2-\left(\|(\widetilde{u}_0,\widetilde{u}_1)\|_{H^1\times L^2}^2+\|(\widetilde{\theta}_0,\widetilde{\theta}_1)\|_{L^2\times L^{1,1}}^2\right)\\
&\gtrsim t\,|P_{\widetilde{\theta}_1}|^2=[\ml{E}_1(1+t)|P_{\widetilde{\theta}_1}|\,]^2
\end{align*}
for large time $t\gg1$, where we used \eqref{Est-10}, \eqref{Est-08} and the Taylor expansion for $\nu_jrt\leqslant \nu_j\mu_6\ll 1$ that
\begin{align*}
\int_{0}^{\mu_6/t}\left|\ell_1\frac{\sin(\nu_1rt)}{\nu_1rt}-\ell_2\frac{\sin(\nu_2rt)}{\nu_2rt}\right|^2\,\mathrm{d}r\gtrsim (\ell_1-\ell_2)^2\int_0^{\mu_6/t}\mathrm{d}r\gtrsim t^{-1}.
\end{align*}
In the above inequality, it was also used
\begin{align*}
\left\|\widehat{\ml{G}}_4(t,|\xi|)\big(\widehat{\widetilde{\theta}}_1(\xi)-P_{\widetilde{\theta}_1}\big)\right\|_{L^2(\Omega_t^0)}^2\lesssim \|\sin(\tilde{c}|\xi|t)\|_{L^2(\Omega_t^0)}^2\|\widetilde{\theta}_1\|_{L^{1,1}}^2\lesssim t^{-1}\|\widetilde{\theta}_1\|_{L^{1,1}}^2.
\end{align*}
Thus, the sharp lower bound estimate for $n=1$ is finished.\\

\noindent\underline{The 2d Case.} We are able to shrink the domain into $\Omega_t^1$ for large time and use $1\geqslant \mathrm{e}^{-|\xi|^2}$ to arrive at
\begin{align*}
	\|\widetilde{\theta}(t,\cdot)\|_{L^2}^2&\geqslant \|\widehat{\ml{G}}_4(t,|\xi|)\,\mathrm{e}^{-|\xi|^2}\|_{L^2(\Omega_t^1)}^2|P_{\widetilde{\theta}_1}|^2-\|\widehat{\widetilde{\theta}}(t,\xi)-\widehat{\ml{G}}_4(t,|\xi|)\widehat{\widetilde{\theta}}_1(\xi)\|_{L^2}^2\\
	& \quad-\left\|\widehat{\ml{G}}_4(t,|\xi|)\big(\widehat{\widetilde{\theta}}_1(\xi)-P_{\widetilde{\theta}_1}\big)\,\mathrm{e}^{-|\xi|^2}\right\|_{L^2(\Omega_t^1)}^2\\
	&\gtrsim \int_{1/t}^{1/\sqrt{t}}\left|\ell_1\frac{\sin(\nu_1rt)}{\nu_1r}-\ell_2\frac{\sin(\nu_2rt)}{\nu_2r}\right|^2\,\mathrm{e}^{-2r^2}\,r\,\mathrm{d}r\,|P_{\widetilde{\theta}_1}|^2\\
	&\quad-\left(\|(\widetilde{u}_0,\widetilde{u}_1)\|_{(H^1)^2\times (L^2)^2}^2+\|(\widetilde{\theta}_0,\widetilde{\theta}_1)\|_{L^2\times L^{1,1}}^2\right)\\
	&\gtrsim \ln t\,|P_{\widetilde{\theta}_1}|^2=[\ml{E}_2(1+t)|P_{\widetilde{\theta}_1}|\,]^2
\end{align*}
for large time $t\gg1$, where we used \eqref{Est-10} and \eqref{Est-08}, moreover, the next lower bound (we apply \eqref{Est-15} directly):
\begin{align*}
&\int_{1/t}^{1/\sqrt{t}}\left|\ell_1\frac{\sin(\nu_1rt)}{\nu_1r}-\ell_2\frac{\sin(\nu_2rt)}{\nu_2r}\right|^2\,\mathrm{e}^{-2r^2}\,r\,\mathrm{d}r\\
&\qquad\geqslant  \frac{(1-\epsilon_0)\ell_1^2}{\nu_1^2}\int_{1/t}^{1/\sqrt{t}}|\sin(\nu_1rt)|^2\,\mathrm{e}^{-2r^2}\,r^{-1}\,\mathrm{d}r-\frac{(1-\epsilon_0)\ell_2^2}{\epsilon_0\nu_2^2}\int_{1/t}^{1/\sqrt{t}}|\sin(\nu_2rt)|^2\,\mathrm{e}^{-2r^2}\,r^{-1}\,\mathrm{d}r\\
&\qquad\geqslant (1-\epsilon_0)\left(\frac{\ell_1^2}{2\nu_1^2}-\frac{\ell_2^2}{\epsilon_0\nu_2^2}\right)\int_{1/t}^{1/\sqrt{t}}\mathrm{e}^{-2r^2}\,r^{-1}\,\mathrm{d}r-\frac{(1-\epsilon_0)\ell_1^2}{2\nu_1^2}\int_{1/t}^{1/\sqrt{t}}\cos(2\nu_1rt)\,\mathrm{e}^{-2r^2}\,r^{-1}\,\mathrm{d}r\\
&\qquad\gtrsim \mathrm{e}^{-2/t}\,\ln t-1\gtrsim \ln t,
\end{align*}
 holds if $\ell_1^2\nu_2^2>2\ell_2^2\nu_1^2$, via choosing
 \begin{align*}
 \epsilon_0=\frac{2\ell_2^2\nu_1^2+\ell_1^2\nu_2^2}{2\ell_1^2\nu_2^2}\in(0,1)\ \ \mbox{such that}\ \ \frac{\ell_1^2}{2\nu_1^2}-\frac{\ell_2^2}{\epsilon_0\nu_2^2}=\frac{\ell_1^2(\ell_1^2\nu_2^2-2\ell_2^2\nu_1^2)}{2\nu_1^2(2\ell_2^2\nu_1^2+\ell_1^2\nu_2^2)}>0.
 \end{align*}
Symmetrically, this large time lower bound also holds if $\ell_2^2\nu_1^2>2\ell_1^2\nu_2^2$. Under the present restriction
\begin{align*}
\frac{(\nu_2^2-\kappa-\gamma^2)^2\nu_2^2}{(\nu_1^2-\kappa-\gamma^2)^2\nu_1^2}=\frac{\ell_1^2\nu_2^2}{\ell_2^2\nu_1^2}\in(0,1/2)\cup(2,+\infty),
\end{align*}
 our proof for the sharp lower bound estimate for $n=2$ is finished.
 
\section{Final remarks on different waves effect}
\hspace{5mm}Different from the energy decay statements in \cite{Zhang-Zuazua=2003,Quintanilla-Racke=2003,Reissig-Wang=2005,Yang-Wang=2006}, we in the present paper show that the thermal waves of type II and type III models grow optimally as those of the transversal waves (cf. \cite[Theorems 1.1 and 1.2]{Ikehata=2023} for the free wave equation) in low dimensions $n\in\{1,2\}$, but the longitudinal waves grow faster than them; the longitudinal waves decay slower than the corresponding thermal waves in high dimensions $n\geqslant 3$. See the next table in details. 
\begin{table}[h!]
	\begin{center}
		\caption{Optimal large time estimates in the $L^2$ framework}
		\medskip
		\label{Table-01}
		\begin{tabular}{ccccccc} 
			\toprule
		\multirow{2}{*}{Different Waves in Thermoelasticity} & \multirow{2}{*}{Reference}	& \multirow{2}{*}{$n=1$} & \multirow{2}{*}{$n=2$} & \multirow{2}{*}{$n=3$} & \multirow{2}{*}{$n=4$} & \multirow{2}{*}{$n\geqslant 5$}\\
		& & & &&  &\\
			\midrule
			\multirow{2}{*}{Transversal Waves $\ml{U}^{s_0}$} & \multirow{2}{*}{\cite{Ikehata=2023}} & \multirow{2}{*}{$\sqrt{t}$} & \multirow{2}{*}{$\sqrt{\ln t}$} & \multirow{2}{*}{$\lesssim 1$} & \multirow{2}{*}{$\lesssim 1$} & \multirow{2}{*}{$\lesssim 1$}\\
			&&&&&&\\
			\midrule
			\multirow{2}{*}{Longitudinal Waves $\widetilde{\ml{U}}^{p_0}$ of Type II Model }& \multirow{2}{*}{Theorem \ref{Thm-II-U}}&\multirow{2}{*}{$t\sqrt{t}$} & \multirow{2}{*}{$t\sqrt{\ln t}$} & \multirow{2}{*}{$\sqrt{t}$} & \multirow{2}{*}{$\sqrt{\ln t}$} & \multirow{2}{*}{$\lesssim 1$}\\
			& & & && &\\
			\midrule
			\multirow{2}{*}{Longitudinal Waves $\ml{U}^{p_0}$ of Type III Model}& \multirow{2}{*}{Theorem \ref{Thm-III-U}}&\multirow{2}{*}{$t\sqrt{t}$} & \multirow{2}{*}{$t\sqrt{\ln t}$} & \multirow{2}{*}{$\sqrt{t}$} & \multirow{2}{*}{$\sqrt{\ln t}$} & \multirow{2}{*}{$t^{1-\frac{n}{4}}$}\\
			& & & && &\\
			\midrule
			\multirow{2}{*}{Thermal Waves $\widetilde{\theta}$ of Type II Model}& \multirow{2}{*}{Theorem \ref{Thm-II-theta}}&\multirow{2}{*}{$\sqrt{t}$} & \multirow{2}{*}{$\sqrt{\ln t}$} & \multirow{2}{*}{$\lesssim 1$} & \multirow{2}{*}{$\lesssim 1$} & \multirow{2}{*}{$\lesssim 1$}\\
			& & && & &\\
			\midrule
						\multirow{2}{*}{Thermal Waves $\theta$ of Type III Model}& \multirow{2}{*}{Theorem \ref{Thm-III-theta}} &\multirow{2}{*}{$\sqrt{t}$} & \multirow{2}{*}{$\sqrt{\ln t}$} & \multirow{2}{*}{$t^{-\frac{1}{4}}$} & \multirow{2}{*}{$t^{-\frac{1}{2}}$} & \multirow{2}{*}{$t^{\frac{1}{2}-\frac{n}{4}}$}\\
			& & && & &\\
			\bottomrule
			\multicolumn{7}{l}{\emph{$*$The notation $\lesssim 1$ denotes bounded estimates from its upper side.}}\\
		\end{tabular}
	\end{center}
\end{table}

That is to say the longitudinal waves play a decisive role for large time behavior in comparison with the transversal waves in the type II and type III models for low dimensions, which is completely different from the classical thermoelastic system (\cite[Remark 2.3]{Chen-Takeda=2023}). This hyperbolic thermal law produces stronger singularities than the classical parabolic Fourier law in thermoelastic systems. These strong singularities lead to the further growth properties.

Furthermore, according to the Helmholtz decomposition $L^2=\overline{\nabla H^1}\oplus D_0$ with the function spaces
\begin{align*}
	\nabla H^1:=\big\{\nabla f:\   f\in H^1\big\},\ \ D_0:=\big\{f\in L^2:\  (\nabla\phi,f)_{L^2}=0\ \ \mbox{with}\ \ \forall \phi\in \ml{C}_0^{\infty}\big\},
\end{align*} 
in physical dimensions $n\in\{1,2,3\}$ the solution is expressed directly by $\ml{U}=\ml{U}^{p_0}+\ml{U}^{s_0}$. In other words, by the last derived estimates in Table \ref{Table-01}, we claim
\begin{align*}
\|\ml{U}(t,\cdot)\|_{(L^2)^n}\approx\ml{D}_n(1+t)\ \ \mbox{as well as} \ \ \|\theta(t,\cdot)\|_{L^2}\approx\ml{E}_n(1+t)
\end{align*}
for large time $t\gg1$, provided that $P_{\theta_1}\neq0$ for the original thermoelastic coupled systems \eqref{Eq-Original-01} of type II and type III.

\appendix
\section{Fourth order differential equations in the Fourier space}
\setcounter{equation}{0}
\renewcommand{\theequation}{A.\arabic{equation}}
\hspace{5mm}For the sake of convenience, we in this appendix state the explicit representations of solutions to some fourth order differential equations with the frequency $\xi\in\mb{R}^n$. These solution's formulas are used effectively in studying $\widehat{v}=\widehat{u},\widehat{\theta}$ for the thermoelastic systems of type II and type III. 

Let us consider the Cauchy problem in the Fourier space, precisely,
\begin{align}\label{General-Cauchy-problem}
\begin{cases}
\ml{L}_{\mathrm{General}}(\partial_t^4,\partial_t^3,\partial_t^2,\partial_t,|\xi|)\widehat{v}=0,&\xi\in\mb{R}^n,\ t>0,\\
\partial^j_t\widehat{v}(0,\xi)=\widehat{v}_j(\xi)\ \ \mbox{for}\ \ j\in\{0,1,2,3\},&\xi\in\mb{R}^n,
\end{cases}
\end{align}
where the corresponding characteristic roots to the quartic $\ml{P}_{\mathrm{General}}(\lambda^4,\lambda^3,\lambda^2,\lambda,|\xi|)=0$ satisfy one of the following cases (due to our need in the models).
\begin{description}
	\item[Case (A-1):] There exist two pairs of distinct conjugate roots
	\begin{align*}
		\lambda_{1,2}=\lambda_{\mathrm{R}}^{(1)}\pm i\lambda_{\mathrm{I}}^{(1)}\ \ \mbox{and}\ \ \lambda_{3,4}=\lambda_{\mathrm{R}}^{(2)}\pm i\lambda_{\mathrm{I}}^{(2)}
	\end{align*}
	with $\lambda_{\mathrm{R}}^{(1)},\lambda_{\mathrm{R}}^{(2)},\lambda_{\mathrm{I}}^{(1)},\lambda_{\mathrm{I}}^{(2)}\in\mb{R}$.
	\item[Case (A-2):] There exist two distinct real roots and a pair of conjugate roots
	\begin{align*}
		\lambda_{1},\lambda_2\ \ \mbox{and}\ \  \lambda_{3,4}=\lambda_{\mathrm{R}}\pm i\lambda_{\mathrm{I}}
	\end{align*}
	with $\lambda_1,\lambda_2,\lambda_{\mathrm{R}},\lambda_{\mathrm{I}}\in\mb{R}$.
\end{description}

\subsection{Representation of solution in Case (A-1)}
\hspace{5mm}Benefiting from two pairs of distinct conjugate roots in Case (A-1), the solution to the Cauchy problem \eqref{General-Cauchy-problem} is given by
\begin{align*}
	\widehat{v}=\sum\limits_{k=1,2}\left[(d_k^{(1)}+d_k^{(2)})\cos(\lambda_{\mathrm{I}}^{(k)}t)+i(d_k^{(1)}-d_k^{(2)})\sin(\lambda_{\mathrm{I}}^{(k)}t)\right]\mathrm{e}^{\lambda_{\mathrm{R}}^{(k)}t},
\end{align*}
where the coefficients $d_k^{(1)}$ and $d_k^{(2)}$ are determined according to
\begin{align}\label{linear-argebra}
	\underbrace{\left(
		{\begin{array}{*{20}c}
				1 & 1 &1 &1\\
				\lambda_1 & \lambda_2 & \lambda_3 & \lambda_4\\
				\lambda_1^2 & \lambda_2^2 & \lambda_3^2 & \lambda_4^2\\
				\lambda_1^3 & \lambda_2^3 & \lambda_3^3 & \lambda_4^3\\
		\end{array}}
		\right)}_{=:\mb{V}}
	\left(
	{\begin{array}{*{20}c}
			d_1^{(1)} \\
			d_1^{(2)} \\
			d_2^{(1)} \\
			d_2^{(2)}\\
	\end{array}}
	\right)
	=\underbrace{\left(
		{\begin{array}{*{20}c}
				\widehat{v}_0 \\
				\widehat{v}_1 \\
				\widehat{v}_2 \\
				\widehat{v}_3 \\
		\end{array}}
		\right)}_{=:\mb{D}}.
\end{align}
The determinant of this Vandermonde matrix $\mb{V}$ is represented by
\begin{align*}
	\det(\mb{V})&=\prod\limits_{1\leqslant j<k\leqslant 4}(\lambda_k-\lambda_j)\\
	&=-4\lambda_{\mathrm{I}}^{(1)}\lambda_{\mathrm{I}}^{(2)}\left[(\lambda_{\mathrm{R}}^{(2)}-\lambda_{\mathrm{R}}^{(1)})^2+(\lambda_{\mathrm{I}}^{(2)}+\lambda_{\mathrm{I}}^{(1)})^2\right]\left[(\lambda_{\mathrm{R}}^{(2)}-\lambda_{\mathrm{R}}^{(1)})^2+(\lambda_{\mathrm{I}}^{(2)}-\lambda_{\mathrm{I}}^{(1)})^2\right].
\end{align*}
From the well-known Cramer rule in the system of linear equations \eqref{linear-argebra}, the coefficients $d_{1,2}:=d_1^{(1,2)}$ and $d_{3,4}:=d_2^{(1,2)}$ are expressed by
\begin{align}\label{dj}
	d_j=\frac{\det(\mb{V}_j)}{\det(\mb{V})}\ \ \mbox{for}\ \ j\in\{1,2,3,4\},
\end{align}
in which $\mb{V}_j$ stands for the matrix formed by replacing the corresponding $j$-th column of $\mb{V}$ by the column vector $\mb{D}$ defined in \eqref{linear-argebra}. Then, carrying out lengthy but straightforward computations, the determinants of $\mb{V}_j$ can be shown explicitly as follows:
\begin{align*}
	\det(\mb{V}_1)&=-2i\lambda_2\lambda_{\mathrm{I}}^{(2)}\left[(\lambda_{\mathrm{R}}^{(2)})^2+(\lambda_{\mathrm{I}}^{(2)})^2\right]\left[(\lambda_{\mathrm{R}}^{(2)}-\lambda_2)^2+(\lambda_{\mathrm{I}}^{(2)})^2\right]\widehat{v}_0\\
	&\quad+\left[2i\lambda_{\mathrm{I}}^{(2)}\left[(\lambda_{\mathrm{R}}^{(2)})^2+(\lambda_{\mathrm{I}}^{(2)})^2\right]^2+\lambda_2^2\lambda_4^2(\lambda_4-\lambda_2)-\lambda_2^2\lambda_3^2(\lambda_3-\lambda_2) \right]\widehat{v}_1\\
	&\quad+\left[-4i\lambda_{\mathrm{R}}^{(2)}\lambda_{\mathrm{I}}^{(2)}\left[(\lambda_{\mathrm{R}}^{(2)})^2+(\lambda_{\mathrm{I}}^{(2)})^2\right]-\lambda_2\lambda_4(\lambda_4^2-\lambda_2^2)+\lambda_2\lambda_3(\lambda_3^2-\lambda_2^2) \right]\widehat{v}_2\\
	&\quad+2i\lambda_{\mathrm{I}}^{(2)}\left[(\lambda_{\mathrm{R}}^{(2)}-\lambda_2)^2+(\lambda_{\mathrm{I}}^{(2)})^2\right]\widehat{v}_3,
\end{align*}
\begin{align*}
	\det(\mb{V}_2)&=2i\lambda_1\lambda_{\mathrm{I}}^{(2)}\left[(\lambda_{\mathrm{R}}^{(2)})^2+(\lambda_{\mathrm{I}}^{(2)})^2\right]\left[(\lambda_{\mathrm{R}}^{(2)}-\lambda_1)^2+(\lambda_{\mathrm{I}}^{(2)})^2\right]\widehat{v}_0\\
	&\quad-\left[2i\lambda_{\mathrm{I}}^{(2)}\left[(\lambda_{\mathrm{R}}^{(2)})^2+(\lambda_{\mathrm{I}}^{(2)})^2\right]^2+\lambda_1^2\lambda_4^2(\lambda_4-\lambda_1)-\lambda_1^2\lambda_3^2(\lambda_3-\lambda_1) \right]\widehat{v}_1\\
	&\quad-\left[-4i\lambda_{\mathrm{R}}^{(2)}\lambda_{\mathrm{I}}^{(2)}\left[(\lambda_{\mathrm{R}}^{(2)})^2+(\lambda_{\mathrm{I}}^{(2)})^2\right]-\lambda_1\lambda_4(\lambda_4^2-\lambda_1^2)+\lambda_1\lambda_3(\lambda_3^2-\lambda_1^2) \right]\widehat{v}_2\\
	&\quad-2i\lambda_{\mathrm{I}}^{(2)}\left[(\lambda_{\mathrm{R}}^{(2)}-\lambda_1)^2+(\lambda_{\mathrm{I}}^{(2)})^2\right]\widehat{v}_3,
\end{align*}
\begin{align*}
	\det(\mb{V}_3)&=-2i\lambda_4\lambda_{\mathrm{I}}^{(1)}\left[(\lambda_{\mathrm{R}}^{(1)})^2+(\lambda_{\mathrm{I}}^{(1)})^2\right]\left[(\lambda_{\mathrm{R}}^{(1)}-\lambda_4)^2+(\lambda_{\mathrm{I}}^{(1)})^2\right]\widehat{v}_0\\
	&\quad+\left[2i\lambda_{\mathrm{I}}^{(1)}\left[(\lambda_{\mathrm{R}}^{(1)})^2+(\lambda_{\mathrm{I}}^{(1)})^2\right]^2+\lambda_1^2\lambda_4^2(\lambda_4-\lambda_1)-\lambda_2^2\lambda_4^2(\lambda_4-\lambda_2) \right]\widehat{v}_1\\
	&\quad+\left[-4i\lambda_{\mathrm{R}}^{(1)}\lambda_{\mathrm{I}}^{(1)}\left[(\lambda_{\mathrm{R}}^{(1)})^2+(\lambda_{\mathrm{I}}^{(1)})^2\right]-\lambda_1\lambda_4(\lambda_4^2-\lambda_1^2)+\lambda_2\lambda_4(\lambda_4^2-\lambda_2^2) \right]\widehat{v}_2\\
	&\quad+2i\lambda_{\mathrm{I}}^{(1)}\left[(\lambda_{\mathrm{R}}^{(1)}-\lambda_4)^2+(\lambda_{\mathrm{I}}^{(1)})^2\right]\widehat{v}_3,
\end{align*}
\begin{align*}
	\det(\mb{V}_4)&=2i\lambda_3\lambda_{\mathrm{I}}^{(1)}\left[(\lambda_{\mathrm{R}}^{(1)})^2+(\lambda_{\mathrm{I}}^{(1)})^2\right]\left[(\lambda_{\mathrm{R}}^{(1)}-\lambda_3)^2+(\lambda_{\mathrm{I}}^{(1)})^2\right]\widehat{v}_0\\
	&\quad-\left[2i\lambda_{\mathrm{I}}^{(1)}\left[(\lambda_{\mathrm{R}}^{(1)})^2+(\lambda_{\mathrm{I}}^{(1)})^2\right]^2+\lambda_1^2\lambda_3^2(\lambda_3-\lambda_1)-\lambda_2^2\lambda_3^2(\lambda_3-\lambda_2) \right]\widehat{v}_1\\
	&\quad-\left[-4i\lambda_{\mathrm{R}}^{(1)}\lambda_{\mathrm{I}}^{(1)}\left[(\lambda_{\mathrm{R}}^{(1)})^2+(\lambda_{\mathrm{I}}^{(1)})^2\right]-\lambda_1\lambda_3(\lambda_3^2-\lambda_1^2)+\lambda_2\lambda_3(\lambda_3^2-\lambda_2^2) \right]\widehat{v}_2\\
	&\quad -2i\lambda_{\mathrm{I}}^{(1)}\left[(\lambda_{\mathrm{R}}^{(1)}-\lambda_3)^2+(\lambda_{\mathrm{I}}^{(1)})^2\right]\widehat{v}_3.
\end{align*}
Let us denote $\mb{V}_{1,2}=:\mb{V}_1^{(1,2)}$ and $\mb{V}_{3,4}=:\mb{V}_2^{(1,2)}$ for briefness. Summing up the last determinants of $\mb{V}_j$  and the coefficients $d_j^{(k)}$ defined in \eqref{dj}, we conclude the refined representation of solution 
\begin{align}\label{Rep-two-pairs}
	\widehat{v}&=\sum\limits_{k=1,2}\left[\cos(\lambda_{\mathrm{I}}^{(k)}t)\,\frac{\det(\mb{V}_k^{(1)})+\det(\mb{V}_k^{(2)})}{\det(\mb{V})}+i\sin(\lambda_{\mathrm{I}}^{(k)}t)\,\frac{\det(\mb{V}_k^{(1)})-\det(\mb{V}_k^{(2)})}{\det(\mb{V})}\right]\mathrm{e}^{\lambda_{\mathrm{R}}^{(k)}t}.
\end{align}
\subsection{Representation of solution in Case (A-2)}
\hspace{5mm}Different from Case (A-1), there are a pair of conjugate roots and two distinct real roots in Case (A-2). Thus, the solution to the Cauchy problem \eqref{General-Cauchy-problem} is shown as
\begin{align*}
	\widehat{v}
	=d_1^{(1)}\mathrm{e}^{\lambda_1t}+d_1^{(2)}\mathrm{e}^{\lambda_2t}+\left[(d_2^{(1)}+d_2^{(2)})\cos(\lambda_{\mathrm{I}}t)+i(d_2^{(1)}-d_2^{(2)})\sin(\lambda_{\mathrm{I}}t)\right]\mathrm{e}^{\lambda_{\mathrm{R}}t},
\end{align*}
where the coefficients $d_k^{(1)}$ and $d_k^{(2)}$ are determined according to \eqref{linear-argebra} with different structure of $\lambda_{1,2}$, and the determinant of the present Vandermonde matrix is represented via
\begin{align*}
	\det(\mb{V})=
	-2i\lambda_{\mathrm{I}}\left[(\lambda_{\mathrm{R}}-\lambda_2)^2+\lambda_{\mathrm{I}}^2\right]\left[(\lambda_{\mathrm{R}}-\lambda_1)^2+\lambda_{\mathrm{I}}^2\right](\lambda_2-\lambda_1).
\end{align*}
Let us apply the Cramer rule again with \eqref{dj} and the following determinants by lengthy but straightforward calculations: 
\begin{align*}
	\det(\mb{V}_1)&=-2i\lambda_2\lambda_{\mathrm{I}}(\lambda_{\mathrm{R}}^2+\lambda_{\mathrm{I}}^2)\left[(\lambda_{\mathrm{R}}-\lambda_2)^2+\lambda_{\mathrm{I}}^2\right]\widehat{v}_0\\
	&\quad+\left[2i\lambda_{\mathrm{I}}(\lambda_{\mathrm{R}}^2+\lambda_{\mathrm{I}}^2)^2+\lambda_2^2\lambda_4^2(\lambda_4-\lambda_2)-\lambda_2^2\lambda_3^2(\lambda_3-\lambda_2)\right]\widehat{v}_1\\
	&\quad+\left[-4i\lambda_{\mathrm{R}}\lambda_{\mathrm{I}}(\lambda_{\mathrm{R}}^2+\lambda_{\mathrm{I}}^2)-\lambda_2\lambda_4(\lambda_4^2-\lambda_2^2)+\lambda_2\lambda_3(\lambda_3^2-\lambda_2^2)\right]\widehat{v}_2\\
	&\quad+2i\lambda_{\mathrm{I}}\left[(\lambda_{\mathrm{R}}-\lambda_2)^2+\lambda_{\mathrm{I}}^2\right]\widehat{v}_3,
\end{align*}
\begin{align*}
	\det(\mb{V}_2)&=2i\lambda_1\lambda_{\mathrm{I}}(\lambda_{\mathrm{R}}^2+\lambda_{\mathrm{I}}^2)\left[(\lambda_{\mathrm{R}}-\lambda_1)^2+\lambda_{\mathrm{I}}^2\right]\widehat{v}_0\\
	&\quad-\left[2i\lambda_{\mathrm{I}}(\lambda_{\mathrm{R}}^2+\lambda_{\mathrm{I}}^2)^2+\lambda_1^2\lambda_4^2(\lambda_4-\lambda_1)-\lambda_1^2\lambda_3^2(\lambda_3-\lambda_1)\right]\widehat{v}_1\\
	&\quad-\left[-4i\lambda_{\mathrm{R}}\lambda_{\mathrm{I}}(\lambda_{\mathrm{R}}^2+\lambda_{\mathrm{I}}^2)-\lambda_1\lambda_4(\lambda_4^2-\lambda_1^2)+\lambda_1\lambda_3(\lambda_3^2-\lambda_1^2)\right]\widehat{v}_2\\
	&\quad-2i\lambda_{\mathrm{I}}\left[(\lambda_{\mathrm{R}}-\lambda_1)^2+\lambda_{\mathrm{I}}^2\right]\widehat{v}_3,
\end{align*}
\begin{align*}
	\det(\mb{V}_3)&=\lambda_1\lambda_2\lambda_4(\lambda_4-\lambda_2)(\lambda_4-\lambda_1)(\lambda_2-\lambda_1)\widehat{v}_0\\
	&\quad-\left[\lambda_2^2\lambda_4^2(\lambda_4-\lambda_2)-\lambda_1^2\lambda_4^2(\lambda_4-\lambda_1)+\lambda_1^2\lambda_2^2(\lambda_2-\lambda_1)\right]\widehat{v}_1\\
	&\quad+\left[\lambda_2\lambda_4(\lambda_4^2-\lambda_2^2)-\lambda_1\lambda_4(\lambda_4^2-\lambda_1^2)+\lambda_1\lambda_2(\lambda_2^2-\lambda_1^2)\right]\widehat{v}_2\\
	&\quad-(\lambda_4-\lambda_2)(\lambda_4-\lambda_1)(\lambda_2-\lambda_1)\widehat{v}_3,
\end{align*}
\begin{align*}
	\det(\mb{V}_4)&=-\lambda_1\lambda_2\lambda_3(\lambda_3-\lambda_2)(\lambda_3-\lambda_1)(\lambda_2-\lambda_1)\widehat{v}_0\\
	&\quad+\left[\lambda_2^2\lambda_3^2(\lambda_3-\lambda_2)-\lambda_1^2\lambda_3^2(\lambda_3-\lambda_1)+\lambda_1^2\lambda_2^2(\lambda_2-\lambda_1)\right]\widehat{v}_1\\
	&\quad-\left[\lambda_2\lambda_3(\lambda_3^2-\lambda_2^2)-\lambda_1\lambda_3(\lambda_3^2-\lambda_1^2)+\lambda_1\lambda_2(\lambda_2^2-\lambda_1^2)\right]\widehat{v}_2\\
	&\quad+(\lambda_3-\lambda_2)(\lambda_3-\lambda_1)(\lambda_2-\lambda_1)\widehat{v}_3.
\end{align*}
Finally, summarizing the last determinants of $\mb{V}_j$ and the coefficients $d_j^{(k)}$,  we are able to conclude the refined representation of solution
\begin{align}\label{Rep-one-pair}
	\widehat{v}&=\frac{\det(\mb{V}_1)}{\det(\mb{V})}\,\mathrm{e}^{\lambda_1t}+\frac{\det(\mb{V}_2)}{\det(\mb{V})}\,\mathrm{e}^{\lambda_2t}+\cos(\lambda_{\mathrm{I}}t)\,\frac{\det(\mb{V}_3)+\det(\mb{V}_4)}{\det(\mb{V})}\,\mathrm{e}^{\lambda_{\mathrm{R}}t}\notag\\
	&\quad+i\sin(\lambda_{\mathrm{I}}t)\,\frac{\det(\mb{V}_3)-\det(\mb{V}_4)}{\det(\mb{V})}\,\mathrm{e}^{\lambda_{\mathrm{R}}t}.
\end{align}

\section*{Acknowledgments} 
 Wenhui Chen is supported in part by the National Natural Science Foundation of China (grant No. 12301270, grant No. 12171317), Guangdong Basic and Applied Basic Research Foundation (grant No. 2025A, grant No. 2023A1515012044). 2024 Basic and Applied Basic Research Topic--Young Doctor Set Sail Project (grant No. 2024A04J0016). Ryo Ikehata is supported in part by Grant-in-Aid for scientific Research (C) 20K03682 of JSPS.

\end{document}